\documentclass[a4paper,11pt]{article}
\usepackage{bm,mathrsfs,amsmath,amsthm,amssymb,amscd,longtable,array}
\usepackage[dvipdfmx]{graphicx,color}
\usepackage{tikz}
\usepackage[frame,cmtip,curve,arrow,matrix,line,graph]{xy}
\newcommand{\nc}{\newcommand}
\nc{\rnc}{\renewcommand}
\nc{\nn}{\nonumber}
\nc{\der}{{\partial}}
\rnc{\Im}{{\rm{Im}\,}}
\rnc{\Re}{{\rm{Re}\,}}
\nc{\db}{\displaybreak[0]\\}
\nc{\bra}{\langle}
\nc{\ket}{\rangle}
\nc{\bs}{\boldsymbol}

\DeclareMathOperator{\sh}{sh}
\DeclareMathOperator{\ch}{ch}

\def\C{\mathbb{C}}
\def\td{\mathop{\mathrm{Td}}\nolimits}
\def\eu{\mathop{\mathrm{Eu}}\nolimits}

\def\euk{\mathop{\mathrm{Eu}}\nolimits_{K}}
\def\eukt{\mathop{\mathrm{Eu}}\nolimits_{K}^{t}}
\def\intk{\int^{[K]}}
\def\mbi#1{\boldsymbol{#1}}
\def\Eu{\mathop{\mathrm{Eu}}\nolimits}

\newcommand{\Dec}{\operatornamewithlimits{Dec}}
\newcommand{\mk}{\mathfrak}
\def\Hom{\mathop{\mathrm{Hom}}\nolimits}
\def\GL{\operatorname{GL}}

\newtheorem{theorem}{Theorem}[section]
\newtheorem{lemma}[theorem]{Lemma}

\newtheorem{proposition}[theorem]{Proposition}
\newtheorem{corollary}[theorem]{Corollary}

\theoremstyle{definition}

\numberwithin{equation}{section}

\numberwithin{equation}{section}

\begin{document}%
%
\title{Algebraic formulas and Geometric derivation \\ of Source Identities }

\author{
Kohei Motegi 
\thanks{E-mail: kmoteg0@kaiyodai.ac.jp 
{\it Faculty of Marine Technology, 
Tokyo University of Marine Science and Technology,}
 {\it Etchujima 2-1-6, Koto-Ku, Tokyo, 135-8533, Japan},
}
\
and \
Ryo Ohkawa 
\thanks{Corresponding author E-mail: ohkawa.ryo@gmail.com
{\it Osaka Central Advanced Mathematical Institute, 
Osaka Metropolitan
University, 558-8585, Osaka, Japan.}
}
}

\date{\today}

\maketitle

\begin{abstract}
Source identities are fundamental identities between multivariable special functions.
We give a geometric derivation of rational and trigonometric source identities.
We also give a systematic derivation and extension of
various determinant representations for source functions
which appeared in previous literature as well as introducing the elliptic version of the determinants,
and obtain identities between determinants.
We also show several symmetrization formulas for the rational version.
\end{abstract}

\section{Introduction}
Identities between special functions play
important roles not only in special function theory itself but also various fields in mathematics and mathematical physics
such as geometric representation theory, supersymmetric gauge theory and integrable systems.
Typicial ones are transformation formulas for hypergeometric functions
which appear as wall-crossing phenomena in supersymmetric gauge theory.
For example, Kajihara's transformation formulas \cite{Kajihara} for hypergeometric functions
appeared in the context of supersymmetric gauge theory \cite{GF,HYY},
and a geometric derivation of the rational version was given by one of the authors and Yoshida \cite{OhkawaYoshida}.
The setting of the algebraic variety used for the derivation was the handsaw quiver variety.

The derivation of transformation formulas for hypergeometric functions by Kajihara \cite{Kajihara} and
an elliptic version by Kajihara-Noumi \cite{KajiharaNoumi} (see also Rosengren \cite{Rosengren} for a different derivation)
is based on specializations of more fundamental identities called as source identities.
The basic idea is that taking multiple principal specializations of the original variables
for the source identities lead to transformation formulas for hypergeometric functions.

For a positive integer $n$, we define
$[1,\dots,n]:=\{ j \ | \ j=1,\dots,n \}$
and denote the sum over all subsets of $[1,\dots,n]$ as 
$\displaystyle \sum_{K \subset [1,\dots,n]} $.
For a set $K$ of integers, we denote the number of elements of the set $K$ as $|K|$.
We introduce the $q$-infinite product
$(u;q)_{\infty}:=\prod_{j=0}^\infty (1-u q^j)$
{\color{black} for $|q|<1$ for $u \in \mathbb{C}$.}
We also introduce the odd theta function
\begin{align}
\theta(u;p):=(u;p)_{\infty}(p/u;p)_{\infty}=\prod_{j=0}^\infty (1-u p^j)(1-p^{j+1}/u),
\end{align}
{ \color{black}
for $0<|p|<1$ and $u \in \mathbb{C} \backslash \{ 0 \}$.}

Elliptic version of the source identity is 
\begin{align}
&\sum_{K \subset [1,\dots,n]} (-z)^{|K|} q^{|K|(|K|-1)/2}
\theta \Big(q^{|K|} \Lambda \prod_{i=1}^n u_i/ \prod_{j=1}^n v_j;p    \Big)
\prod_{\substack{i \in K \\ j \not\in K}} \frac{\theta(qv_j/v_i;p)}{\theta(v_j/v_i;p)}
\prod_{\substack{ i \in K  \\ 1 \le k \le n }} \frac{ \theta(u_k/v_i;p)}{\theta(qu_k/v_i;p)} \nn \\
=& \sum_{K \subset [1,\dots,n]} (-z)^{|K|} q^{|K|(|K|-1)/2} 
\theta \Big(q^{|K|} \Lambda \prod_{i=1}^n u_i/ \prod_{j=1}^n v_j;p    \Big)
\prod_{\substack{i \in K \\ j \not\in K}} 
\frac{\theta(q u_i/u_j;p)}{\theta(u_i/u_j;p)}
\prod_{\substack{ i \in K  \\ 1 \le k \le n }} \frac{\theta(u_i/v_k;p)}{\theta(qu_i/v_k;p)},
\label{ellipticKajiharaNoumi}
\end{align}
where
$q,p,\Lambda,z,v_1,\dots,v_n,u_1,\dots,u_n$ are complex parameters
and
{\color{black}
are in the general positions in the sense that the zeros of the factors in the denominator
are avoided. We also assume $0<|p|<1$.}
The elliptic identity \eqref{ellipticKajiharaNoumi} was derived in {\color{black} \cite[Theorem 1.2]{KajiharaNoumi} }
 by applying difference operators
on the Frobenius determinant formula
which is an elliptic version of the Cauchy determinant formula,
and was used to derive transformation formulas for elliptic hypergeometric functions
by multiple principal specializations.

For the trigonometric case, the source identity is
\begin{align}
&\sum_{K \subset [1,\dots,m]} (-z)^{|K|} q^{|K|(|K|-1)/2} \prod_{\substack{i \in K \\ j \not\in K}} \frac{v_i-q v_j}{v_i-v_j}
\prod_{\substack{ i \in K  \\ 1 \le k \le n }} \frac{v_i-u_k}{v_i-qu_k}
\nonumber \\
=&\frac{(z;q)_{\infty}}{(q^{m-n}z;q)_{\infty}} \sum_{K \subset [1,\dots,n]} (-q^{m-n} z)^{|K|} q^{|K|(|K|-1)/2} \prod_{\substack{i \in K \\ j \not\in K}} \frac{q u_i- u_j}{u_i-u_j}
\prod_{\substack{ i \in K  \\ 1 \le k \le m }} \frac{v_k-u_i}{v_k-qu_i},
\label{trigonometricKajihara}
\end{align}
where $q,z,v_1,\dots,v_m,u_1,\dots,u_n$ are complex parameters {\color{black} and in general positions,}
and the integers $m$ and $n$ are not necessarily the same.
{\color{black}
Note the factor $\displaystyle \frac{(z;q)_{\infty}}{(q^{m-n}z;q)_{\infty}} $
in the right hand side is explicitly 
\begin{align}
\frac{(z;q)_{\infty}}{(q^{m-n}z;q)_{\infty}}
=
\begin{cases}
\displaystyle \prod_{j=0}^{m-n-1} (1-q^j z) & n<m  \\
1 & n=m \\
\displaystyle \frac{1}{\prod_{j=1}^{n-m} (1-q^{-j} z)} & n>m
\end{cases}
. \label{ratiocases}
\end{align}
Since the ratio of the $q$-inifnite products are explicitly \eqref{ratiocases} which consists of finite number of factors,
we can assume $q$ in \eqref{trigonometricKajihara} to be $|q| \neq 1$ instead of $|q|<1$.
Using the $q$-Pochhammer symbol ($q$-shifted factorial) $\displaystyle (u;q)_n:=\frac{(u;q)_\infty}{(uq^n;q)_{\infty}}$ for $n$ an integer,
the factor can be written as
$\displaystyle \frac{(z;q)_{\infty}}{(q^{m-n}z;q)_{\infty}}=(z;q)_{m-n}$.
Let us discuss two limits. The first is the rational limit.
When $n<m$,
introducing $c$ and $w$ through $q=e^{2c}$, $z=e^{2w}$
and using \eqref{ratiocases}, we have
$q^{-(m-n)(m-n-1)/4} z^{-(m-n)/2} (z;q)_{m-n}=(-2)^{m-n} \prod_{j=0}^{m-n-1} \sh(w+cj)$,
and the rational limit of $q^{-(m-n)(m-n-1)/4} z^{-(m-n)/2} (z;q)_{m-n}$
obtained by dividing by $\epsilon^{m-n}$, replacing $c$ and $w$ by $\epsilon c$ and $\epsilon w$ and taking the limit $\epsilon \to 0$
is $(-2)^{m-n} \prod_{j=0}^{m-n-1} (w+cj)$.
In the same way,
for $n>m$, one finds that the rational limit of $q^{-(n-m+1)(n-m)/4} z^{(n-m)/2}
(z;q)_{m-n}$ is
$\frac{1}{(-2)^{n-m} \prod_{j=1}^{n-m}  (w-cj)}
$.
Another type of limit is taking $q \to 1$,
which is simply
$(z;q)_{m-n} \to (1-z)^{m-n}$.
Note the factor appearing in the rational version of the source identity presented later
is this type.}

The trigonometric source identity \eqref{trigonometricKajihara} was derived by Kirillov-Noumi \cite{KN} and Mimachi-Noumi \cite{MN},
and was used in \cite{MN} to construct integral representations for  eigenfunctions of Macdonald operators,
and later used by Kajihara \cite{Kajihara}
which is equivalent to \eqref{trigonometricKajihara} under change of variables)
to derive transformation formulas for $q$-hypergeometric functions.
{\color{black}
For example, the identity
\cite[(3.3)]{Kajihara} (see also \cite[(1.14)]{KN}, \cite[(2.3)]{MN})
which is written as
\begin{align}
F(u|z;w)=(u;t)_{r-p} F(ut^{r-p}|w;z),
\end{align}
where
\begin{align}
F(u|z;w)=
\sum_{K \subset [1,\dots,r]} (-u)^{|K|} t^{|K|(|K|-1)/2} \prod_{\substack{i \in K \\ j \not\in K}} \frac{1-t z_i/z_j}{1-z_i/z_j}
\prod_{\substack{ i \in K  \\ 1 \le k \le p }} \frac{1-z_i w_k}{1-tz_i w_k},
\end{align}
and $z=(z_1,\dots,z_r)$, $w=(w_1,\dots,w_p)$
is equivalent to \eqref{trigonometricKajihara}
under the identification
$z \leftrightarrow u$, $q \leftrightarrow t$, $m \leftrightarrow r$, $n \leftrightarrow p$,
$v_i \leftrightarrow z_i^{-1}$, $u_j \leftrightarrow w_j$.

We also remark that from the identity \eqref{trigonometricKajihara} for the case 
$n<m$, we can get the case $n>m$ as follows.
Replacing $v_i$, $u_j$, $z$ by $u_i^{-1}$, $v_j^{-1}$, $q^{n-m}z$ in
\eqref{trigonometricKajihara} for the case $n<m$,  we have
\begin{align}
&\sum_{K \subset [1,\dots,m]} (-q^{n-m} z)^{|K|} q^{|K|(|K|-1)/2} \prod_{\substack{i \in K \\ j \not\in K}} \frac{q u_i- u_j}{u_i-u_j}
\prod_{\substack{ i \in K  \\ 1 \le k \le n }} \frac{v_k-u_i}{v_k-qu_i} \nn \\
=
&\prod_{j=1}^{m-n} (1-q^{-j} z)
\sum_{K \subset [1,\dots,n]} (-z)^{|K|} q^{|K|(|K|-1)/2} \prod_{\substack{i \in K \\ j \not\in K}} \frac{v_i-q v_j}{v_i-v_j}
\prod_{\substack{ i \in K  \\ 1 \le k \le m }} \frac{v_i-u_k}{v_i-qu_k}.
\end{align}
Dividing both hand sides by $\prod_{j=1}^{m-n} (1-q^{-j} z)$ and replacing $m$ and $n$ by $n$ and $m$,
we get the identity \eqref{trigonometricKajihara} for the case $n>m$.
}

Rational version of the source identity is
\begin{align}
&\sum_{K \subset [1,\dots,m]} (-z)^{|K|} \prod_{\substack{i \in K \\ j \not\in K}} \frac{v_i-v_j-c}{v_i-v_j}
\prod_{\substack{ i \in K  \\ 1 \le k \le n }} \frac{v_i-u_k}{v_i-u_k-c} \nonumber \\
=&(1-z)^{m-n} \sum_{K \subset [1,\dots,n]} (-z)^{|K|} \prod_{\substack{i \in K \\ j \not\in K}} \frac{u_i-u_j+c}{u_i-u_j}
\prod_{\substack{ i \in K  \\ 1 \le k \le m }} \frac{u_i-v_k}{u_i-v_k+c},
\label{rationalKajihara}
\end{align}
where $c,z,v_1,\dots,v_m,u_1,\dots,u_n$ are complex parameters {\color{black} and in general positions.
We also assume $z \neq 1$ for $n>m$. The function in the left hand side itself does not need this assumption
for $n>m$.}
Although not written in this form,
one observes that the identity between determinants
in Gorsky-Zabrodin-Zotov {\color{black} \cite[(4.6)]{GZZ}}
is equivalent to this rational version of source identity.
One can see the equivalence by using the expansion of the determinants given in
Belliard-Slavnov-Vallet {\color{black} \cite[(A.12)]{BSV}.}
In \cite{GZZ}, the identity was combined with the Bethe ansatz equations
to derive duality between classical and quantum integrable systems.
The trigonometric version of the identity between determinants is given in {\color{black} \cite[(4.3)]{BLZZ}}
which is equivalent to \eqref{trigonometricKajihara} using the determinant expansion.
 
One of the main purposes of this paper is to give a geometric 
understanding of these identities.
The summations of both sides in the source identities are regarded
as torus equivariant integrals over moduli of the following framed quiver
$Q \colon$
\begin{center}
\includegraphics[scale=1]{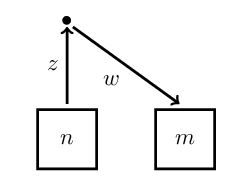}
\end{center}
In \cite{O} and \cite{OS}, we derived wall-crossing formula 
for general setting of framed quiver moduli.
Applying this to the above framed quiver $Q$, we derive the
source identities.

Another purpose of this paper is to
give various determinant forms of the rational functions
appearing in the source identities.
As mentioned above,
rational/trigonometric source identities appear in the works in
integrable systems in disguise, i.e. in the form of identities between determinants.
Besides application to quantum-classical duality \cite{GZZ},
one notes that special cases $z=1$ and/or $m=n$ of the rational functions or determinants 
appear as representations for (partial) domain wall boundary partition functions
and scalar products
of the six-vertex models in the works of
Izergin-Korepin \cite{Izergin,Korepin},
Gromov-Sever-Vieira \cite{GSV},
Kostov \cite{Kostovone,Kostovtwo}, Foda-Wheeler \cite{FodaWheeler}, Belliard-Slavnov \cite{BS} and Minin-Pronko-Tarasov \cite{MPT,PT}.
{\color{black} For example, see \cite[(1.3)]{GSV}, \cite[(3.39)]{Kostovtwo}, \cite[(3.14), (4.13), (4.17)]{FodaWheeler}, \cite[(4.5)]{BS},
\cite[(4.9), (5.9)]{MPT}.}
Recently, $z$ generic and $m \neq n$ case of the rational version of source functions
appear in a more generalized version of domain wall boundary partition functions in
Belliard-Pimenta-Slavnov {\color{black} \cite[(92), (93)]{BPS}.}
We also remark this type of determinant form first appeared in the context of the Bose gas in the work by Gaudin
{\color{black} \cite[Appendix B]{Gaudin}.}
We give a systematic derivation and extension of various forms
of the determinants which appeared in these literature.
The idea is to start from the expressions of source functions as rational functions with products of difference operators applied,
which appeared in Kajihara-Noumi \cite{Kajihara,KajiharaNoumi}.
The systematic derivation can be extended to an elliptic version
by noting that \eqref{ellipticKajiharaNoumi} is originally obtained by
acting products of difference operators on
an elliptic version of Cauchy determinant
in Kajihara-Noumi \cite{KajiharaNoumi}.
As a consequence of this systematic derivation,
we obtain identities between elliptic determinants.

This paper is organized as follows.
We discuss the elliptic, trigonometric and rational
version of source functions and source identities
in section 2, 3 and 4 respectively.
We give a systematic derivation of various determinant forms
of source functions, and also a standard complex analytic proof
of source identities.
For the trigonometric and rational case, we discuss a degeneration from the elliptic and trigonometric version
and also a geometric derivation of the source identities.
Additionaly, we prove a rational version of symmetrization formulas
of Lascoux for the rational case.

We summarize the notations which are frequently used in this paper.
We denote the set of variables $x_1,\dots,x_\ell$ as $\overline{x}$.
We frequently use the sets
$\overline{u}=\{ u_1,\dots,u_n \}$ and $\overline{v}=\{ v_1,\dots,v_m \}$.
We define additive and multiplicative difference operators $T_x$ and $T_{q,x}$
acting on function $f(x)$
as $T_xf(x):=f(x+c)$ and $T_{q,x}f(x):=f(qx)$ where $c$ and $q$ are complex parameters.
The inverse operators 
$T_x^{-1}$ and $T_{q,x}^{-1}$
act on function $f(x)$
as $T_x^{-1} f(x):=f(x-c)$ and $T_{q,x}^{-1}f(x):=f(q^{-1}x)$.

\section{Elliptic version}
In this section, we derive several determinant representations
for elliptic source functions
and also give a complex analytic proof of the elliptic source identity.

\subsection{Determinant representations of elliptic source functions}
We introduce two types of determinant representations
for elliptic source functions.
We give an elliptic version of the determinants introduced in
Minin-Pronko-Tarasov \cite{MPT,PT} and
Belliard-Slavnov \cite{BS}.

Recall the $q$-infinite product and the version of the odd theta function which is used in this paper
\begin{align}
&(u;q)_{\infty}=\prod_{j=0}^\infty (1-u q^j), \\
&\theta(u;p)=(u;p)_{\infty}(p/u;p)_{\infty}=\prod_{j=0}^\infty (1-u p^j)(1-p^{j+1}/u).
\end{align}
Let us
introduce notations for the left and right hand side of \eqref{ellipticKajiharaNoumi}:
\begin{align}
&F^{ell(z)}_n(\overline{u}|\overline{v}) \nn \\
=&\sum_{K \subset [1,\dots,n]} (-z)^{|K|} q^{|K|(|K|-1)/2}
\theta \Big(q^{|K|} \Lambda \prod_{i=1}^n u_i/ \prod_{j=1}^n v_j;p    \Big)
\prod_{\substack{i \in K \\ j \not\in K}} \frac{\theta(qv_j/v_i;p)}{\theta(v_j/v_i;p)}
\prod_{\substack{ i \in K  \\ 1 \le k \le n }} \frac{ \theta(u_k/v_i;p)}{\theta(qu_k/v_i;p)},
\label{ellsourceleft} \\
&G^{ell(z)}_n(\overline{u}|\overline{v}) \nn \\
=& \sum_{K \subset [1,\dots,n]} (-z)^{|K|} q^{|K|(|K|-1)/2} 
\theta \Big(q^{|K|} \Lambda \prod_{i=1}^n u_i/ \prod_{j=1}^n v_j;p    \Big)
\prod_{\substack{i \in K \\ j \not\in K}} 
\frac{\theta(q u_i/u_j;p)}{\theta(u_i/u_j;p)}
\prod_{\substack{ i \in K  \\ 1 \le k \le n }} \frac{\theta(u_i/v_k;p)}{\theta(qu_i/v_k;p)}.
\label{ellsourceright}
\end{align}
We call these functions as elliptic source functions in this paper.

The following elliptic analog of factorization of Vandermonde determinants
is derived in Rosengren-Schlosser {\color{black} \cite[Proposition 6.1]{RS}.}
\begin{align}
\det_{1 \le j,k \le n} \Big( \psi_j^{A_{n-1}}(u_k;p,r) \Big)
=\frac{(p;p)_\infty^n}{(p^n;p^n)_\infty^n}
\theta \Bigg( r \prod_{\ell=1}^n u_\ell;p \Bigg)
\prod_{1 \le i < j \le n} u_j \theta(u_i/u_j;p),
\label{ellipticvandermonde}
\end{align}
where
\begin{align}
\psi_j^{A_{n-1}}(u;p,r):=u^{j-1} \theta (p^{j-1}(-1)^{n-1}r u^n;p^n), \ \ \ j=1,\dots,n.
\end{align}
See \cite{Katori} for applications of the determinants to elliptic stochastic systems.

We also use the Frobenius determinant formula \cite{Frobenius}
which is an elliptic analogue of Cauchy's determinant formula
\begin{align}
\det_{1 \le i,j \le n} \Bigg( 
\frac{\theta(\Lambda u_i /v_j;p)}{\theta(\Lambda;p)\theta(u_i/ v_j;p)}
\Bigg)
=\frac{ \theta \Big( \Lambda \prod_{i=1}^n u_i/ \prod_{j=1}^n v_j;p    \Big) \prod_{1 \le i < j \le n} u_j \theta(u_i/u_j;p) v_j^{-1} \theta(v_j/v_i;p)   }
{\theta(\Lambda;p) \prod_{1 \le i,j \le n} \theta(u_i /v_j;p) }.
\label{Frobenius}
\end{align}

The source functions can be rewritten using multiplicative difference operators as
\begin{align}
&F^{ell(z)}_n(\overline{u}|\overline{v}) \nn \\
=&\frac{ \theta(\Lambda;p) \prod_{i,j=1}^n \theta (u_i/v_j;p)  }{\prod_{1 \le i < j \le n} u_j \theta(u_i/u_j;p) v_j^{-1} \theta(v_j/v_i;p) }
\prod_{j=1}^n (1-z T_{q,v_j}^{-1}) \det_{1 \le i,j \le n} \Bigg( 
\frac{\theta(\Lambda u_i/ v_j;p)}{\theta(\Lambda;p)\theta(u_i/ v_j;p)}
\Bigg), \label{KNF} \\
&G^{ell(z)}_n(\overline{u}|\overline{v}) \nn \\
=&\frac{ \theta(\Lambda;p) \prod_{i,j=1}^n \theta (u_i/v_j;p)  }{\prod_{1 \le i < j \le n} u_j \theta(u_i/u_j;p) v_j^{-1} \theta(v_j/v_i;p) }
\prod_{j=1}^n (1-z T_{q,u_j}) \det_{1 \le i,j \le n} \Bigg( 
\frac{\theta(\Lambda u_i/ v_j;p)}{\theta(\Lambda;p)\theta(u_i/ v_j;p)}
\Bigg), \label{KNleft}
\end{align}
which is essentially the same object used by
Kajihara-Noumi {\color{black} \cite[(1.14)]{KajiharaNoumi}} (see also {\color{black} \cite[(3.1)]{Kajihara} } for the trigonometric version)
to derive transformation formulas for elliptic hypergeometric functions
from the source identity $F^{ell(z)}_n(\overline{u}|\overline{v})=G^{ell(z)}_n(\overline{u}|\overline{v})$.
Starting from these forms, it is easy to see $F^{ell(z)}_n(\overline{u}|\overline{v})=G^{ell(z)}_n(\overline{u}|\overline{v})$
by inserting the difference operators into the determinants.
Using the factorization of the elliptic determinant
\eqref{Frobenius}, the right hand side of \eqref{KNF} and  \eqref{KNleft} 
can be rewritten as
\begin{align}
&F^{ell(z)}_n(\overline{u}|\overline{v}) \nn \\
=&\frac{
\prod_{1 \le i,j \le n} \theta(u_i /v_j;p)
}{\prod_{1 \le i < j \le n} v_j^{-1} \theta(v_j/v_i;p)}
\prod_{j=1}^n (1-z T_{q,v_j}^{-1}) 
\frac{ \theta \Big( \Lambda \prod_{i=1}^n u_i/ \prod_{j=1}^n v_j;p    \Big) \prod_{1 \le i < j \le n} v_j^{-1} \theta(v_j/v_i;p)  }
{ \prod_{1 \le i,j \le n} \theta(u_i / v_j;p) }, \label{ellipticlhsrewrite} \\
&G^{ell(z)}_n(\overline{u}|\overline{v}) \nn \\
=&\frac{
\prod_{1 \le i,j \le n} \theta(u_i /v_j;p)
}{\prod_{1 \le i < j \le n} u_j \theta(u_i/u_j;p)}
\prod_{j=1}^n (1-z T_{q,u_j}) 
\frac{ \theta \Big( \Lambda \prod_{i=1}^n u_i/ \prod_{j=1}^n v_j;p    \Big) \prod_{1 \le i < j \le n} u_j \theta(u_i/u_j;p)  }
{ \prod_{1 \le i,j \le n} \theta(u_i / v_j;p) }, \label{ellipticrhsrewrite}
\end{align}
and further expanding using
\begin{align}
&\Big( \prod_{1 \le i < j \le n} v_j^{-1} \theta(v_j/v_i;p) \Big)^{-1}
\prod_{k \in K} T_{q,v_k}^{-1}  \Big( \prod_{1 \le i < j \le n} v_j^{-1} \theta(v_j/v_i;p) \Big)
=q^{|K|(|K|-1)/2} \prod_{\substack{  i \in K \\ j \not\in K }}  \frac{\theta(qv_j/v_i;p)}{\theta(v_j/v_i;p)}, \\
&\Big( \prod_{1 \le i < j \le n} u_j \theta(u_i/u_j;p) \Big)^{-1}
\prod_{k \in K} T_{q,u_k}  \Big( \prod_{1 \le i < j \le n} u_j \theta(u_i/u_j;p) \Big)
=q^{|K|(|K|-1)/2} \prod_{\substack{  i \in K \\ j \not\in K }}  \frac{\theta(qu_i/u_j;p)}{\theta(u_i/u_j;p)},
\end{align}
gives
\eqref{ellsourceleft} and \eqref{ellsourceright} respectively, and 
this is how the elliptic source identity \eqref{ellipticKajiharaNoumi}
was derived in \cite{KajiharaNoumi}.

We derive an elliptic version of determinant representations
whose trigonometric and rational version 
appeared in Minin-Pronko-Tarasov {\color{black} \cite[(4.9), (5.9)]{MPT}.} 
\begin{proposition} \label{propellipticMPT}
We have the following determinant representations:
\begin{align}
&F^{ell(z)}_n(\overline{u}|\overline{v}) 
=\frac{\theta ( r \prod_{\ell=1}^n v_\ell^{-1};p )}{\det_{1 \le i,j \le n} \Big( \sum_{k=1}^n p_{ik} \psi_k^{A_{n-1}}(v_j^{-1};p,r) \Big)} \nn \\
&\times
\det_{1 \le i,j \le n} \Bigg(
\Bigg( \frac{\theta \Big( \Lambda \prod_{\ell=1}^n u_\ell / \prod_{\ell=1}^n v_\ell;p    \Big)}{\theta ( r \prod_{\ell=1}^n v_\ell^{-1};p )  } 
\Bigg)^{\delta_{i1}}
 \sum_{k=1}^n p_{ik} \psi_k^{A_{n-1}}(v_j^{-1};p,r)
\nn \\
&-z
\Bigg( \frac{\theta \Big(q \Lambda \prod_{\ell=1}^n u_\ell / \prod_{\ell=1}^n v_\ell;p    \Big)}{\theta ( qr \prod_{\ell=1}^n v_\ell^{-1};p )  } 
\Bigg)^{\delta_{i1}}
\sum_{k=1}^n p_{ik} \psi_k^{A_{n-1}}(qv_j^{-1};p,r)
\prod_{\ell=1}^n
\frac{
\theta(u_\ell/v_j;p)
}
{
\theta(qu_\ell /v_j;p) 
} \Bigg),
 \label{ellipticMPTlhs} \\
&G^{ell(z)}_n(\overline{u}|\overline{v}) 
=\frac{\theta ( r \prod_{\ell=1}^n u_\ell;p )}{\det_{1 \le i,j \le n} \Big( \sum_{k=1}^n q_{ik} \psi_k^{A_{n-1}}(u_j;p,r) \Big)} \nn \\
&\times
\det_{1 \le i,j \le n} \Bigg(
\Bigg( \frac{\theta \Big( \Lambda \prod_{\ell=1}^n u_\ell / \prod_{\ell=1}^n v_\ell;p    \Big)}{\theta ( r \prod_{\ell=1}^n u_\ell;p )  } 
\Bigg)^{\delta_{i1}}
 \sum_{k=1}^n q_{ik} \psi_k^{A_{n-1}}(u_j;p,r)
\nn \\
&-z
\Bigg( \frac{\theta \Big(q \Lambda \prod_{\ell=1}^n u_\ell / \prod_{\ell=1}^n v_\ell;p    \Big)}{\theta ( qr \prod_{\ell=1}^n u_\ell;p )  } 
\Bigg)^{\delta_{i1}}
\sum_{k=1}^n q_{ik} \psi_k^{A_{n-1}}(qu_j;p,r)
\prod_{\ell=1}^n
\frac{
\theta(u_j/v_\ell;p)
}
{
\theta(qu_j /v_\ell;p) 
} \Bigg).
 \label{ellipticMPTrhs}
\end{align}
Here, $\delta_{ij}$ is the Kronecker delta,
and $r$, $p_{ij}$ and $q_{ij}$ $(i,j=1,\dots,n)$ are additional parameters
such that $\det_{1 \le i,j \le n}(p_{ij}) \not\equiv 0$, $\det_{1 \le i,j \le n}(q_{ij}) \not\equiv 0$.
\end{proposition}

\begin{proof}
We show the case for $G^{ell(z)}_n(\overline{u}|\overline{v})$ \eqref{ellipticMPTrhs}
as \eqref{ellipticMPTlhs} can be derived  in the same way.
We start from the expression
\eqref{ellipticrhsrewrite}. Using the elliptic analog of the factorization of Vandermonde determinants
\eqref{ellipticvandermonde}, we can rewrite \eqref{ellipticrhsrewrite} as
\begin{align}
G^{ell(z)}_n(\overline{u}|\overline{v})
=&\frac{\theta ( r \prod_{\ell=1}^n u_\ell;p )
\prod_{1 \le i,j \le n} \theta(u_i /v_j;p)
}{\det_{1 \le j,k \le n} \Big( \psi_k^{A_{n-1}}(u_j;p,r) \Big)} \nn \\
&\times 
\prod_{j=1}^n (1-z T_{q,u_j}) 
\frac{ \theta \Big( \Lambda \prod_{i=1}^n u_i \prod_{j=1}^n v_j;p    \Big) \det_{1 \le j,k \le n} \Big( \psi_k^{A_{n-1}}(u_j;p,r) \Big)  }
{\theta ( r \prod_{\ell=1}^n u_\ell;p )  \prod_{1 \le i,j \le n} \theta(u_i /v_j;p) }.
\label{elliptictochutwo}
\end{align}
Multiplying the denominator and the numerator by some $\det_{1 \le i,k \le n} (q_{ik}) \not\equiv 0$,
where $q_{ik}$, $i,k=1,\dots,n$ are parameters independent of $u_i$, $i=1,\dots,n$,
we can rewrite \eqref{elliptictochutwo} as
\begin{align}
G^{ell(z)}_n(\overline{u}|\overline{v})=
&\frac{\theta ( r \prod_{\ell=1}^n u_\ell;p )  \prod_{1 \le i,j \le n} \theta(u_i /v_j;p) }{ \det_{1 \le i,k \le n} (q_{ik}) \det_{1 \le j,k \le n} \Big( \psi_k^{A_{n-1}}(u_j;p,r) \Big)}
\prod_{j=1}^n (1-z T_{q,u_j}) \nn \\
&\times \frac{ \theta \Big( \Lambda \prod_{i=1}^n u_i / \prod_{j=1}^n v_j;p    \Big) \det_{1 \le i,k \le n}  (q_{ik})  \det_{1 \le j,k \le n} \Big( \psi_k^{A_{n-1}}(u_j;p,r) \Big)  }
{\theta ( r \prod_{\ell=1}^n u_\ell;p ) \prod_{1 \le i,j \le n} \theta(u_i /v_j;p) } \nn \\
=&\frac{\theta ( r \prod_{\ell=1}^n u_\ell;p )  \prod_{1 \le i,j \le n} \theta(u_i /v_j;p)  }{\det_{1 \le i,j \le n} \Big( \sum_{k=1}^n q_{ik} \psi_k^{A_{n-1}}(u_j;p,r) \Big)}
\prod_{j=1}^n (1-z T_{q,u_j}) \nn \\
&\times \frac{ \theta \Big( \Lambda \prod_{i=1}^n u_i / \prod_{j=1}^n v_j;p    \Big) \det_{1 \le i,j \le n} \Big( \sum_{k=1}^n q_{ik} \psi_k^{A_{n-1}}(u_j;p,r) \Big)  }
{\theta ( r \prod_{\ell=1}^n u_\ell;p ) \prod_{1 \le i,j \le n} \theta(u_i /v_j;p) } \nn \\
=&\frac{\theta ( r \prod_{\ell=1}^n u_\ell;p ) \prod_{1 \le i,j \le n} \theta(u_i /v_j;p)}{\det_{1 \le i,j \le n} \Big( \sum_{k=1}^n q_{ik} \psi_k^{A_{n-1}}(u_j;p,r) \Big)}
\prod_{j=1}^n (1-z T_{q,u_j}) \nn \\
&\times
\det_{1 \le i,j \le n} \Bigg(
\Bigg( \frac{\theta \Big( \Lambda \prod_{\ell=1}^n u_\ell / \prod_{\ell=1}^n v_\ell;p    \Big)}{\theta ( r \prod_{\ell=1}^n u_\ell;p )  } 
\Bigg)^{\delta_{i1}}
\frac{
 \sum_{k=1}^n q_{ik} \psi_k^{A_{n-1}}(u_j;p,r)
}
{
\prod_{\ell=1}^n \theta(u_j /v_\ell;p) 
}
\Bigg).
\end{align}
Inserting the difference operators inside the determinant, we have
\begin{align}
&G^{ell(z)}_n(\overline{u}|\overline{v}) 
=
\frac{\theta ( r \prod_{\ell=1}^n u_\ell;p ) \prod_{1 \le i,j \le n} \theta(u_i /v_j;p)}{\det_{1 \le i,j \le n} \Big( \sum_{k=1}^n q_{ik} \psi_k^{A_{n-1}}(u_j;p,r) \Big)} \nn \\
&\times
\det_{1 \le i,j \le n} \Bigg(
\prod_{j=1}^n (1-z T_{q,u_j})
\Bigg(
\Bigg( \frac{\theta \Big( \Lambda \prod_{\ell=1}^n u_\ell / \prod_{\ell=1}^n v_\ell;p    \Big)}{\theta ( r \prod_{\ell=1}^n u_\ell;p )  } 
\Bigg)^{\delta_{i1}}
\frac{
 \sum_{k=1}^n q_{ik} \psi_k^{A_{n-1}}(u_j;p,r)
}
{
\prod_{\ell=1}^n \theta(u_j /v_\ell;p) 
} \Bigg)
\Bigg) \nn \\
=&\frac{\theta ( r \prod_{\ell=1}^n u_\ell;p )}{\det_{1 \le i,j \le n} \Big( \sum_{k=1}^n q_{ik} \psi_k^{A_{n-1}}(u_j;p,r) \Big)} \nn \\
&\times
\det_{1 \le i,j \le n} \Bigg(
\Bigg( \frac{\theta \Big( \Lambda \prod_{\ell=1}^n u_\ell / \prod_{\ell=1}^n v_\ell;p    \Big)}{\theta ( r \prod_{\ell=1}^n u_\ell;p )  } 
\Bigg)^{\delta_{i1}}
 \sum_{k=1}^n q_{ik} \psi_k^{A_{n-1}}(u_j;p,r)
\nn \\
&-z
\Bigg( \frac{\theta \Big(q \Lambda \prod_{\ell=1}^n u_\ell / \prod_{\ell=1}^n v_\ell;p    \Big)}{\theta ( qr \prod_{\ell=1}^n u_\ell;p )  } 
\Bigg)^{\delta_{i1}}
\sum_{k=1}^n q_{ik} \psi_k^{A_{n-1}}(qu_j;p,r)
\prod_{\ell=1}^n
\frac{
\theta(u_j/v_\ell;p)
}
{
\theta(qu_j /v_\ell;p) 
} \Bigg).
\label{elliptictransform}
\end{align}

\end{proof}

This form can be regarded as an elliptic analog of Minin-Pronko-Tarasov \cite{MPT},
which will be explained in later sections.

We apply the same idea due to Kajihara-Noumi
\cite{KajiharaNoumi} now to determinant representations instead of source functions.
From \eqref{ellipticMPTlhs}, \eqref{ellipticMPTrhs}
and $F^{ell(z)}_n(\overline{u}|\overline{v}) 
=G^{ell(z)}_n(\overline{u}|\overline{v})$, we have the following
identity between determinants.

\begin{theorem}
We have the following identity:
\begin{align}
&\frac{\theta ( r \prod_{\ell=1}^n v_\ell^{-1};p )}{\det_{1 \le i,j \le n} \Big( \sum_{k=1}^n p_{ik} \psi_k^{A_{n-1}}(v_j^{-1};p,r) \Big)} \nn \\
&\times
\det_{1 \le i,j \le n} \Bigg(
\Bigg( \frac{\theta \Big( \Lambda \prod_{\ell=1}^n u_\ell / \prod_{\ell=1}^n v_\ell;p    \Big)}{\theta ( r \prod_{\ell=1}^n v_\ell^{-1};p )  } 
\Bigg)^{\delta_{i1}}
 \sum_{k=1}^n p_{ik} \psi_k^{A_{n-1}}(v_j^{-1};p,r)
\nn \\
&-z
\Bigg( \frac{\theta \Big(q \Lambda \prod_{\ell=1}^n u_\ell / \prod_{\ell=1}^n v_\ell;p    \Big)}{\theta ( qr \prod_{\ell=1}^n v_\ell^{-1};p )  } 
\Bigg)^{\delta_{i1}}
\sum_{k=1}^n p_{ik} \psi_k^{A_{n-1}}(qv_j^{-1};p,r)
\prod_{\ell=1}^n
\frac{
\theta(u_\ell/v_j;p)
}
{
\theta(qu_\ell /v_j;p) 
} \Bigg) \nn \\
=&
\frac{\theta ( r \prod_{\ell=1}^n u_\ell;p )}{\det_{1 \le i,j \le n} \Big( \sum_{k=1}^n q_{ik} \psi_k^{A_{n-1}}(u_j;p,r) \Big)} \nn \\
&\times
\det_{1 \le i,j \le n} \Bigg(
\Bigg( \frac{\theta \Big( \Lambda \prod_{\ell=1}^n u_\ell / \prod_{\ell=1}^n v_\ell;p    \Big)}{\theta ( r \prod_{\ell=1}^n u_\ell;p )  } 
\Bigg)^{\delta_{i1}}
 \sum_{k=1}^n q_{ik} \psi_k^{A_{n-1}}(u_j;p,r)
\nn \\
&-z
\Bigg( \frac{\theta \Big(q \Lambda \prod_{\ell=1}^n u_\ell / \prod_{\ell=1}^n v_\ell;p    \Big)}{\theta ( qr \prod_{\ell=1}^n u_\ell;p )  } 
\Bigg)^{\delta_{i1}}
\sum_{k=1}^n q_{ik} \psi_k^{A_{n-1}}(qu_j;p,r)
\prod_{\ell=1}^n
\frac{
\theta(u_j/v_\ell;p)
}
{
\theta(qu_j /v_\ell;p) 
} \Bigg).
\end{align}
\end{theorem}

We next derive an elliptic analog of determinant representations
which can be regarded as an elliptic lift of the rational version
introduced in Belliard-Slavnov {\color{black} \cite[(4.5)]{BS}.}

\begin{proposition} \label{propellipticBS}
We have the following determinant representations:
\begin{align}
&F_{n}^{ell(z)}(\overline{u}|\overline{v}) \nonumber \\
=&\frac{\theta(\Delta;p)
}{\prod_{1 \le i < j \le n} v_j^{-1} \theta(v_j/v_i;p) \eta_j \theta(\eta_i/\eta_j;p) } 
\nn \\
&\times
\det_{1 \le i,j \le n}
\Bigg(
\Bigg(
\frac{  \theta \Big( \Lambda \prod_{i=1}^n u_i/ \prod_{j=1}^n v_j;p    \Big)   }
{\theta \Big( \Delta \prod_{i=1}^n \eta_i/ \prod_{j=1}^n v_j;p    \Big)   }
\Bigg)^{\delta_{i1}}
\frac{\theta(\Delta \eta_i/v_j;p) \prod_{k=1,k \neq i}^n \theta(\eta_k/v_j;p)  }{\theta(\Delta;p)}
\nn \\
&-z
\Bigg(
\frac{  \theta \Big(q \Lambda  \prod_{i=1}^n u_i/ \prod_{j=1}^n v_j;p    \Big)   }
{\theta \Big(q \Delta  \prod_{i=1}^n \eta_i/ \prod_{j=1}^n v_j;p    \Big)   }
\Bigg)^{\delta_{i1}}
\frac{\theta(q \Delta  \eta_i/v_j;p) \prod_{k=1,k \neq i}^n \theta(q \eta_k/v_j;p)  }{\theta(\Delta;p)}
\prod_{k=1}^n
\frac{ \theta(u_k/v_j;p)}{ \theta(qu_k/v_j;p)}
\Bigg), \label{ellipticBSone} \\
&G_{n}^{ell(z)}(\overline{u}|\overline{v}) \nonumber \\
=&\frac{\theta(\Delta;p)
}{\prod_{1 \le i < j \le n} u_j \theta(u_i/u_j;p) \eta_j^{-1} \theta(\eta_j/\eta_i;p) } 
\nn \\
&\times
\det_{1 \le i,j \le n}
\Bigg(
\Bigg(
\frac{  \theta \Big( \Lambda \prod_{i=1}^n u_i/ \prod_{j=1}^n v_j;p    \Big)   }
{\theta \Big( \Delta \prod_{i=1}^n u_i/ \prod_{j=1}^n \eta_j;p    \Big)   }
\Bigg)^{\delta_{i1}}
\frac{\theta(\Delta u_i/\eta_j;p) \prod_{k=1,k \neq j}^n \theta(u_i/\eta_k;p)  }{\theta(\Delta;p)}
\nn \\
&-z
\Bigg(
\frac{  \theta \Big(q \Lambda  \prod_{i=1}^n u_i/ \prod_{j=1}^n v_j;p    \Big)   }
{\theta \Big(q \Delta  \prod_{i=1}^n u_i/ \prod_{j=1}^n \eta_j;p    \Big)   }
\Bigg)^{\delta_{i1}}
\frac{\theta(q \Delta  u_i/\eta_j;p) \prod_{k=1,k \neq j}^n \theta(q u_i/\eta_k;p)  }{\theta(\Delta;p)}
\prod_{k=1}^n
\frac{ \theta(u_i/v_k;p)}{ \theta(qu_i/v_k;p)}
\Bigg). \label{ellipticBStwo}
\end{align}
Here, $\Delta$ and $\eta_i$ $(i=1,\dots,n)$ are additional parameters.
\end{proposition}

\begin{proof}
We show the case for
$G_{n}^{ell(z)}(\overline{u}|\overline{v})$ \eqref{ellipticBStwo}.
\eqref{ellipticBSone} can be proved in the same way.
We start from the expression
\eqref{ellipticrhsrewrite} and first rewrite as follows

\begin{align}
&G_{n}^{ell(z)}(\overline{u}|\overline{v}) \nonumber \\
=&\frac{
\prod_{1 \le i,j \le n} \theta(u_i /v_j;p)
}{\prod_{1 \le i < j \le n} u_j \theta(u_i/u_j;p) \eta_j^{-1} \theta(\eta_j/\eta_i;p) } \nn \\
&\times
\prod_{j=1}^n (1-z T_{q,u_j}) 
\frac{ \theta \Big( \Lambda \prod_{i=1}^n u_i/ \prod_{j=1}^n v_j;p    \Big) \prod_{1 \le i < j \le n} u_j \theta(u_i/u_j;p) \eta_j^{-1} \theta(\eta_j/\eta_i;p)  }
{ \prod_{1 \le i,j \le n} \theta(u_i / v_j;p) } \nn \\
=&
\frac{
\prod_{1 \le i,j \le n} \theta(u_i /v_j;p)
}{\prod_{1 \le i < j \le n} u_j \theta(u_i/u_j;p) \eta_j^{-1} \theta(\eta_j/\eta_i;p) } 
\prod_{j=1}^n (1-z T_{q,u_j}) 
\nn \\
&\times
\prod_{1 \le i,j \le n} \frac{ \theta(u_i / \eta_j;p)}{\theta(u_i / v_j;p)}
\frac{ \theta \Big( \Lambda \prod_{i=1}^n u_i/ \prod_{j=1}^n v_j;p    \Big) \prod_{1 \le i < j \le n} u_j \theta(u_i/u_j;p) \eta_j^{-1} \theta(\eta_j/\eta_i;p)  }
{ \prod_{1 \le i,j \le n} \theta(u_i / \eta_j;p) }.
\end{align}
Using the Frobenius determinant formula \eqref{Frobenius}, we get
\begin{align}
&G_{n}^{ell(z)}(\overline{u}|\overline{v}) \nonumber \\
=&
\frac{
\prod_{1 \le i,j \le n} \theta(u_i /v_j;p)
}{\prod_{1 \le i < j \le n} u_j \theta(u_i/u_j;p) \eta_j^{-1} \theta(\eta_j/\eta_i;p) } \nn \\
&\times
\prod_{j=1}^n (1-z T_{q,u_j}) 
\prod_{1 \le i,j \le n} \frac{ \theta(u_i / \eta_j;p)}{\theta(u_i / v_j;p)}
\frac{ \theta(\Delta;p) \theta \Big( \Lambda \prod_{i=1}^n u_i/ \prod_{j=1}^n v_j;p    \Big)   }
{\theta \Big( \Delta \prod_{i=1}^n u_i/ \prod_{j=1}^n \eta_j;p    \Big)   }
\det_{1 \le i,j \le n} \Bigg(
\frac{\theta(\Delta u_i/\eta_j;p)}{\theta(\Delta;p) \theta(u_i/\eta_j;p)}
\Bigg)
 \nn \\
=&\frac{\theta(\Delta;p)
\prod_{1 \le i,j \le n} \theta(u_i /v_j;p)
}{\prod_{1 \le i < j \le n} u_j \theta(u_i/u_j;p) \eta_j^{-1} \theta(\eta_j/\eta_i;p) } 
\nn \\
&\times
\prod_{j=1}^n (1- z  T_{q,u_j}) 
\det_{1 \le i,j \le n} \Bigg(
\Bigg(
\frac{  \theta \Big( \Lambda \prod_{i=1}^n u_i/ \prod_{j=1}^n v_j;p    \Big)   }
{\theta \Big( \Delta \prod_{i=1}^n u_i/ \prod_{j=1}^n \eta_j;p    \Big)   }
\Bigg)^{\delta_{i1}}
\frac{\theta(\Delta u_i/\eta_j;p) \prod_{k=1,k \neq j}^n \theta(u_i/\eta_k;p)  }{\theta(\Delta;p)\prod_{k=1}^n \theta(u_i/v_k;p)}
\Bigg).
\end{align}
Inserting the difference operators into the determinant, we have
\begin{align}
&G_{n}^{ell(z)}(\overline{u}|\overline{v}) \nonumber \\
=&\frac{\theta(\Delta;p)
\prod_{1 \le i,j \le n} \theta(u_i /v_j;p)
}{\prod_{1 \le i < j \le n} u_j \theta(u_i/u_j;p) \eta_j^{-1} \theta(\eta_j/\eta_i;p) } 
\nn \\
&\times
\det_{1 \le i,j \le n}
\Bigg(
(1- z  T_{q,u_i}) 
\Bigg(
\frac{  \theta \Big( \Lambda \prod_{i=1}^n u_i/ \prod_{j=1}^n v_j;p    \Big)   }
{\theta \Big( \Delta \prod_{i=1}^n u_i/ \prod_{j=1}^n \eta_j;p    \Big)   }
\Bigg)^{\delta_{i1}}
\frac{\theta(\Delta u_i/\eta_j;p) \prod_{k=1,k \neq j}^n \theta(u_i/\eta_k;p)  }{\theta(\Delta;p)\prod_{k=1}^n \theta(u_i/v_k;p)}
\Bigg) \nn \\
=&\frac{\theta(\Delta;p)
}{\prod_{1 \le i < j \le n} u_j \theta(u_i/u_j;p) \eta_j^{-1} \theta(\eta_j/\eta_i;p) } 
\nn \\
&\times
\det_{1 \le i,j \le n}
\Bigg(
\Bigg(
\frac{  \theta \Big( \Lambda \prod_{i=1}^n u_i/ \prod_{j=1}^n v_j;p    \Big)   }
{\theta \Big( \Delta \prod_{i=1}^n u_i/ \prod_{j=1}^n \eta_j;p    \Big)   }
\Bigg)^{\delta_{i1}}
\frac{\theta(\Delta u_i/\eta_j;p) \prod_{k=1,k \neq j}^n \theta(u_i/\eta_k;p)  }{\theta(\Delta;p)}
\nn \\
&-z
\Bigg(
\frac{  \theta \Big(q \Lambda  \prod_{i=1}^n u_i/ \prod_{j=1}^n v_j;p    \Big)   }
{\theta \Big(q \Delta  \prod_{i=1}^n u_i/ \prod_{j=1}^n \eta_j;p    \Big)   }
\Bigg)^{\delta_{i1}}
\frac{\theta(q \Delta  u_i/\eta_j;p) \prod_{k=1,k \neq j}^n \theta(q u_i/\eta_k;p)  }{\theta(\Delta;p)}
\prod_{k=1}^n 
\frac{ \theta(u_i/v_k;p)}{ \theta(qu_i/v_k;p)}
\Bigg).
\end{align}

\end{proof}

\subsection{A complex anaytic proof
of the elliptic source identity}

Although it is easy to see the elliptic source identity
follows from the Frobenius determinant formula as first derived in {\color{black} \cite[Theorem 1.2]{KajiharaNoumi}},
we give a direct complex analytic proof.
Some of the properties of elliptic source functions derived in this section are elliptic version
of the ones in \cite{MPT,PT}.

We prove the following version of the source identity
\begin{align}
&\sum_{K \subset [1,\dots,n]} (-z)^{|K|} q^{|K|(|K|-1)/2}
\theta \Big(q^{|K|} \Lambda \prod_{i=1}^n u_i/ \prod_{j=1}^n v_j;p    \Big) \nn \\
&\times \prod_{\substack{i \in K \\ j \not\in K}} \frac{\theta(qv_j/v_i;p)}{\theta(v_j/v_i;p)}
\prod_{\substack{ j \not\in K  \\ 1 \le k \le n }}  \theta(qu_k/v_j;p)
\prod_{\substack{ i \in K  \\ 1 \le k \le n }} \theta(u_k/v_i;p) \nn \\
=& \sum_{K \subset [1,\dots,n]} (-z)^{|K|} q^{|K|(|K|-1)/2} 
\theta \Big(q^{|K|} \Lambda \prod_{i=1}^n u_i/ \prod_{j=1}^n v_j;p    \Big) \nn \\
&\times
\prod_{\substack{i \in K \\ j \not\in K}} 
\frac{\theta(q u_i/u_j;p)}{\theta(u_i/u_j;p)}
\prod_{\substack{ j \not\in K  \\ 1 \le k \le n }}\theta(qu_j/v_k;p)
\prod_{\substack{ i \in K  \\ 1 \le k \le n }} \theta(u_i/v_k;p).
\label{polyellipticsource}
\end{align}
{\color{black}
Note that \eqref{polyellipticsource} is obtained from
\eqref{ellipticKajiharaNoumi} by multiplying both hand sides by
$\prod_{i=1}^n \prod_{k=1}^n \theta(q u_k/v_i;p)$ and the identities are equivalent.
}
Denote the left and right hand side of
\eqref{polyellipticsource} as
$P_{n}^{ell(z)}(\overline{u}|\overline{v})$
and
$Q_{n}^{ell(z)}(\overline{u}|\overline{v})$ respectively.
The relations with $F_{n}^{ell(z)}(\overline{u}|\overline{v})$ and $G_{n}^{ell(z)}(\overline{u}|\overline{v})$
are
\begin{align}
P_{n}^{ell(z)}(\overline{u}|\overline{v})
&=\prod_{i=1}^n \prod_{k=1}^n \theta(q u_k/v_i;p) F_{n}^{ell(z)}(\overline{u}|\overline{v}), \\
Q_{n}^{ell(z)}(\overline{u}|\overline{v})
&=\prod_{i=1}^n \prod_{k=1}^n \theta(q u_k/v_i;p) G_{n}^{ell(z)}(\overline{u}|\overline{v}).
\end{align}

The complex analytic proof is based on the
following fact for elliptic interpolation which can be seen for example in \cite{Rosengrenlecturenote}.
\begin{proposition} \label{ellipticinterpolation} ({\color{black} \cite[Proposition 4.27]{Rosengrenlecturenote}})
Let $t,y_1,\dots,y_n \in \mathbb{C}^*$ be such that
$t \not\in p^{\mathbb{Z}}$, $y_j/y_k \not\in p^{\mathbb{Z}} \ (j \neq k)$.
Let $V$ be the space of functions that are analytic for $x \neq 0$
and satisfy $f(px)=(-1)^n y_1 \cdots y_n t^{-1} x^{-n} f(x)$.
Then, any $f \in V$ is uniquely determined by the values $f(y_1),\dots,f(y_n)$ as
\begin{align}
f(x)=\sum_{j=1}^n f(y_j) \frac{\theta(tx/y_j;p)}{\theta(t;p)}
\prod_{\substack{k=1 \\ k \neq j}}^n \frac{\theta(x/y_k;p)}{\theta(y_j/y_k;p)}.
\end{align}
\end{proposition}

Using  $\theta(px;p)/\theta(x;p)=-1/x$,
we can see the following holds.
\begin{lemma} \label{quasiperiodicities}
We have
\begin{align}
P_{n}^{ell(z)}(\overline{u}|\overline{v})|_{v_k \to p v_k}
&=(-p^{-1})^{n+1} q^n \Lambda v_k^{-n-1} \prod_{i=1}^n u_i^2 \prod_{\substack{j=1 \\ j \neq k}}^n v_j^{-1}
P_{n}^{ell(z)}(\overline{u}|\overline{v}), \\
Q_{n}^{ell(z)}(\overline{u}|\overline{v})|_{v_k \to p v_k}
&=(-p^{-1})^{n+1} q^n \Lambda v_k^{-n-1} \prod_{i=1}^n u_i^2 \prod_{\substack{j=1 \\ j \neq k}}^n v_j^{-1}
Q_{n}^{ell(z)}(\overline{u}|\overline{v}),
\end{align}
for $k=1,\dots,n$.
\end{lemma}
From Proposition
\ref{ellipticinterpolation}
and Lemma \ref{quasiperiodicities},
one notes that it is enough to check
$P_{n}^{ell(z)}(\overline{u}|\overline{v})=Q_{n}^{ell(z)}(\overline{u}|\overline{v})$
for $(n+1)^n$ distinct points in $(v_1,\dots,v_n)$.
We can in fact check the equality for the following $(2n) \times (2n-1) \times \cdots \times (n+1) \geq (n+1)^n$ specializations
$v_1=q^{\epsilon_1} u_{p_1}, \ v_2=q^{\epsilon_2} u_{p_2}, \ \dots, v_n=q^{\epsilon_n} u_{p_n}$,
$p_1,\dots,p_n \in \{1,\dots,n \}$, $\epsilon_1,\dots,\epsilon_n \in \{ 0,1 \}$
and $v_i \neq v_j$ for $i \neq j$.
First we check the following cases which correspond to $p_i=p_j$ for some $i \neq j$.
\begin{lemma} \label{ellipticvanishinglemma}
If $v_i=u_k$, $v_j=qu_k$ for some $i,j,k \ (i \neq j)$,
we have 
\begin{align}
P_{n}^{ell(z)}(\overline{u}|\overline{v})=Q_{n}^{ell(z)}(\overline{u}|\overline{v})=0.
\end{align}
\end{lemma}
\begin{proof}
We check the case $k=n$. From the expressions \eqref{polyellipticsource},
one can check the following relations hold
when substituting $v_n=q u_n$ in $P_{n}^{ell(z)}(\overline{u}|\overline{v})$
and $Q_{n}^{ell(z)}(\overline{u}|\overline{v})$
\begin{align}
&P_{n}^{ell(z)}(\overline{u}|\overline{v})|_{v_n=q u_n} \nn \\
=&-z \prod_{j=1}^{n} \theta(u_j/q u_n;p) \prod_{j=1}^{n-1} q \theta(u_n/v_j;p)
P_{n-1}^{ell(z)}(u_1,\dots,u_{n-1}|v_1,\dots,v_{n-1}|\Lambda), \label{subellipticone} \\
&Q_{n}^{ell(z)}(\overline{u}|\overline{v})|_{v_n=q u_n} \nn \\
=&-z \prod_{j=1}^{n} \theta(u_j/q u_n;p) \prod_{j=1}^{n-1} q \theta(u_n/v_j;p)
Q_{n-1}^{ell(z)}(u_1,\dots,u_{n-1}|v_1,\dots,v_{n-1}|\Lambda). \label{subelliptictwo}
\end{align}
From the factor $\prod_{j=1}^{n-1} q \theta(u_n/v_j)$ in
\eqref{subellipticone} and \eqref{subelliptictwo},
one notes that $P_{n}^{ell(z)}(\overline{u}|\overline{v})|_{v_n=q u_n}$
and $Q_{n}^{ell(z)}(\overline{u}|\overline{v})|_{v_n=q u_n}$ vanish after
further substituting $v_j=u_n$ for some $j \ (1 \le j \le n-1)$.
The other cases can be checked in the same way or follow by symmetry.
\end{proof}

We next assume $p_i \neq p_j$ for $i \neq j$.
For $I=\{ 1 \le i_1 < i_2 < \cdots < i_{|I|} \le n \}$ and $J=\{1 \le j_1 < j_2 < \cdots < j_{|J|} \le n \}$
such that $I \cup J=[ 1,\dots,n ]$,
let us denote by $\overline{v}=\{ \overline{u}_I, q \overline{u}_J \}$ the 
following type of substitution
\begin{align}
v_{q_1}=u_{i_1}, \dots, v_{q_{|I|}}=u_{i_{|I|}}, v_{r_1}=q u_{j_1}, \dots, v_{r_{|J|}}=q u_{j_{|J|}},
\end{align}
for some $\{ q_1, \dots, q_{|I|}, r_1,\dots, r_{|J|} \}=[1,\dots,n ]$.

We have the following explicit evaluations.
\begin{proposition} \label{ellipticPQspecializations}
We have
\begin{align}
&P_{n}^{ell(z)}(\overline{u}|\{ \overline{u}_I, q \overline{u}_J \}) \nn \\
=&(-z)^{|J|} q^{|J|(|J|-1)/2} \theta(\Lambda;p) \prod_{ \substack{ i \in I \\ j \in J  }  } \theta(u_i/u_j;p)
\prod_{ \substack{ i \in I \\ 1 \le j \le n  }  } \theta(q u_j/u_i;p)
 \prod_{ \substack{ i \in J \\ j \in J  }  } \theta(u_i/q u_j;p), \\
&Q_{n}^{ell(z)}(\overline{u}|\{ \overline{u}_I, q \overline{u}_J \}) \nn \\
=&(-z)^{|J|} q^{|J|(|J|-1)/2} \theta(\Lambda;p) \prod_{ \substack{ i \in I \\ j \in J  }  } \theta(u_i/u_j;p)
\prod_{ \substack{ i \in I \\ 1 \le j \le n  }  } \theta(q u_j/u_i;p)
 \prod_{ \substack{ i \in J \\ j \in J  }  } \theta(u_i/q u_j;p).
\end{align}
\end{proposition}
\begin{proof}
After the substitution
$
v_{q_1}=u_{i_1}, \dots, v_{q_{|I|}}=u_{i_{|I|}}, v_{r_1}=qu_{j_1}, \dots, v_{r_{|J|}}=qu_{j_{|J|}}$,
one notes from the factors
$\displaystyle
\prod_{\substack{ j \not\in K  \\ 1 \le k \le n }}  \theta(qu_k/v_j;p)
$ and
$\displaystyle
\prod_{\substack{ i \in K  \\ 1 \le k \le n }} \theta(u_k/v_i;p)$
that only the summand corresponding to $K=\{ r_1,\dots,r_{|J|} \}$  survives
and we get
\begin{align}
&P_{n}^{ell(z)}(\overline{u}|\{ \overline{u}_I, q\overline{u}_J \}) \nn \\
=&(-z)^{|J|} q^{|J|(|J|-1)/2} \theta(\Lambda;p) \prod_{\substack{i \in J \\ j \in  I}} \frac{\theta(q u_j/q u_i;p)}{\theta(u_j/q u_i;p)}
\prod_{\substack{ j \in I  \\ 1 \le k \le n }} \theta(q u_k/u_j;p)
\prod_{\substack{ j \in J  \\ 1 \le k \le n }} \theta(u_k/q u_j;p) \nn \\
=&(-z)^{|J|} q^{|J|(|J|-1)/2} \theta(\Lambda;p) \prod_{ \substack{ i \in I \\ j \in J  }  } \theta(u_i/u_j;p)
\prod_{ \substack{ i \in I \\ 1 \le j \le n  }  } \theta(q u_j/u_i;p)
 \prod_{ \substack{ i \in J \\ j \in J  }  } \theta(u_i/q u_j;p).
\end{align}

The factorization of
$Q_{n}^{ell(z)}(\overline{u}|\{ \overline{u}_I, q\overline{u}_J \})$ after the substitution
can be checked in the same way.
Substituting
$
v_{q_1}=u_{i_1}, \dots, v_{q_{|I|}}=u_{i_{|I|}}, v_{r_1}=q u_{j_1}, \dots, v_{r_{|J|}}=q u_{j_{|J|}}$,
we note from the factors
$\displaystyle \prod_{\substack{ j \not\in K  \\ 1 \le k \le n }}\theta(qu_j/v_k;p)$ and
$\displaystyle \prod_{\substack{ i \in K  \\ 1 \le k \le n }} \theta(u_i/v_k;p)$
that only the summand corresponding to $K=\{ r_1,\dots,r_{|J|} \}$  survives,
and the factorized expression can be rewritten as
\begin{align}
&Q_{n}^{ell(z)}(\overline{u}|\{ \overline{u}_I, q\overline{u}_J \}) \nn \\
=&(-z)^{|J|} q^{|J|(|J|-1)/2} \theta(\Lambda;p) \prod_{\substack{i \in J \\ j \in  I}} \frac{\theta(q u_i/u_j;p)}{\theta(u_i/u_j;p)}
\nn \\
&\times \prod_{\substack{ j \in J  \\  k \in I }} \theta(u_j/u_k;p) \prod_{\substack{ j \in J  \\  k \in J }} \theta(u_j/qu_k;p)
\prod_{\substack{ i \in I  \\ k \in I }} \theta(q u_i/u_k;p) \prod_{\substack{ i \in I  \\ k \in J }} \theta(u_i/u_k;p) \nn \\
=&(-z)^{|J|} q^{|J|(|J|-1)/2} \theta(\Lambda;p) \prod_{ \substack{ i \in I \\ j \in J  }  } \theta(u_i/u_j;p)
\prod_{ \substack{ i \in I \\ 1 \le j \le n  }  } \theta(q u_j/u_i;p)
 \prod_{ \substack{ i \in J \\ j \in J  }  } \theta(u_i/q u_j;p).
\end{align}

\end{proof}

From Lemma
\ref{ellipticvanishinglemma} and Proposition \ref{ellipticPQspecializations},
we have checked
\eqref{polyellipticsource} holds for enough specializations,
hence we conclude \eqref{polyellipticsource} itself holds.

\section{Trigonometric version}

In this section, we discuss trigonometric source identity as a degeneration from
elliptic source identity.
We also give a geometric derivation and a complex analytic proof
as well as giving a systematic derivation of determinant representations
which can be regarded as a generalization of the ones which appeared in previous literature.

\subsection{From elliptic source identity to trigonometric source identity}
We discuss the degeneration from elliptic source identity \eqref{ellipticKajiharaNoumi}
to trigonometric source identity \eqref{trigonometricKajihara}.

We introduce $q$-integers, $q$-factorials and $q$-binomials $\displaystyle [n]_q:=1+q+q^2+\cdots+q^{n-1}=\frac{1-q^n}{1-q}$, $[n]_q!:=\prod_{j=1}^n [j]_q$,
$\displaystyle \begin{bmatrix}
   n  \\
   \ell
\end{bmatrix}_{q}:=\frac{[n]_q!}{[\ell]_q! [n-\ell]_q !}$ for $n,\ell$ nonnegative integers.
To discuss degeneration, we also
introduce notation for the following rational function
\begin{align}
&F_{n,m}^{trig(z)}(\overline{u}|\overline{v}|\Lambda) \nn \\
:=&
\sum_{K \subset [1,\dots,m]} (-z)^{|K|} 
q^{|K|(|K|-1)/2}
(1-q^{|K|} \Lambda )
\prod_{\substack{i \in K \\ j \not\in K}}
\frac{v_i-qv_j}{v_i-v_j}
\prod_{\substack{ i \in K  \\ 1 \le k \le n }} \frac{v_i-u_k}{v_i-qu_k}.
\end{align}
Also note $\theta(z;0)=1-z$.
By taking the trigonometric limit $p \to 0$
of the elliptic source identity \eqref{ellipticKajiharaNoumi},
we get the following identity.
\begin{theorem} \label{limitofellipticsource}
For $n \geq m$,
we have
\begin{align}
&\sum_{\ell=0}^{n-m} (-z)^\ell q^{\ell(\ell-1)/2}
\displaystyle \begin{bmatrix}
   n-m  \\
   \ell
\end{bmatrix}_{q}
F_{n,m}^{trig(q^{n-m} z)}(\overline{u}|\overline{v}|q^\ell \Lambda) \nn \\
=&\sum_{K \subset [1,\dots,n]} (-z)^{|K|} q^{|K|(|K|-1)/2}
(1-q^{|K|} \Lambda )
\prod_{\substack{i \in K \\ j \not\in K}} \frac{q u_i- u_j}{u_i-u_j}
\prod_{\substack{ i \in K  \\ 1 \le k \le m }} \frac{v_k-u_i}{v_k-qu_i}.
\label{extendedtrigonometricsource}
\end{align}
\end{theorem}

Note that when $\Lambda=0$, using the $q$-binomial theorem
\begin{align}
&\sum_{\ell=0}^{n}
z^{\ell} q^{\ell(\ell+1)/2}
\begin{bmatrix}
   n  \\
   \ell
\end{bmatrix}_{q}
=\prod_{j=1}^{n} (1+q^j z), \label{qbinomial}
\end{align}
\eqref{extendedtrigonometricsource} reduces to
\begin{align}
\prod_{j=1}^{n-m} (1-q^{j-1} z)
F_{n,m}^{trig(q^{n-m} z)}(\overline{u}|\overline{v}|0)
=
\sum_{K \subset [1,\dots,n]} (-z)^{|K|} q^{|K|(|K|-1)/2}
\prod_{\substack{i \in K \\ j \not\in K}} \frac{q u_i- u_j}{u_i-u_j}
\prod_{\substack{ i \in K  \\ 1 \le k \le m }} \frac{v_k-u_i}{v_k-qu_i}
\label{lambdazero}
\end{align}
which becomes the trigonometric source identity 
\eqref{trigonometricKajihara}
after redefining $q^{n-m} z$ as $z$.
Also note that equating both hand sides of the coefficients of
$\Lambda$ in \eqref{extendedtrigonometricsource},
we get \eqref{lambdazero} with $z$ replaced by $qz$.
This means \eqref{extendedtrigonometricsource} is actually a combination of
the trigonometric source identity \eqref{trigonometricKajihara}.

\begin{proof}
We prove by descending induction on $m$, starting from $m=n$.
By taking the limit $p \to 0$ of elliptic source identity \eqref{ellipticKajiharaNoumi}
and redefining $\Lambda \prod_{i=1}^n u_i / \prod_{j=1}^n v_j$ as $\Lambda$,
we get the case $m=n$ of \eqref{extendedtrigonometricsource},
which corresponds to the initial case.

Suppose \eqref{extendedtrigonometricsource} holds,
and derive \eqref{extendedtrigonometricsource}
with $m$ replaced by $m-1$.
We take the limit $v_m \rightarrow \infty$ of
both hand sides of \eqref{extendedtrigonometricsource}.
Dividing the sums in $F_{n,m}^{trig(q^{n-m} z)}(\overline{u}|\overline{v}|q^\ell \Lambda)$
into two parts according to whether $m \in K$ or $m \not\in K$,
one can show
\begin{align}
\lim_{v_m \rightarrow \infty}
F_{n,m}^{trig(q^{n-m} z)}(\overline{u}|\overline{v}|q^\ell \Lambda)
=-q^{n-m} z F_{n,m-1}^{trig(q^{n-m+1} z)}(\overline{u}|\overline{v}|q^{\ell+1} \Lambda)
+F_{n,m-1}^{trig(q^{n-m+1} z)}(\overline{u}|\overline{v}|q^\ell \Lambda).
\end{align}
Then we find the limit $v_m \rightarrow \infty$ of the left hand side of \eqref{extendedtrigonometricsource} can be rewritten 
in the following way
\begin{align}
&\sum_{\ell=0}^{n-m} (-z)^\ell q^{\ell(\ell-1)/2}
\displaystyle \begin{bmatrix}
   n-m  \\
   \ell
\end{bmatrix}_{q}
\Bigg(
-q^{n-m} z F_{n,m-1}^{trig(q^{n-m+1} z)}(\overline{u}|\overline{v}|q^{\ell+1} \Lambda)
+F_{n,m-1}^{trig(q^{n-m+1} z)}(\overline{u}|\overline{v}|q^\ell \Lambda)
\Bigg) \nn \\
=&\sum_{\ell=0}^{n-m+1} (-z)^\ell q^{\ell(\ell-1)/2}
\Bigg(
q^{n-m-\ell+1}
\displaystyle \begin{bmatrix}
   n-m  \\
   \ell-1
\end{bmatrix}_{q}+
\displaystyle \begin{bmatrix}
   n-m  \\
   \ell
\end{bmatrix}_{q}
\Bigg)
F_{n,m-1}^{trig(q^{n-m+1} z)}(\overline{u}|\overline{v}|q^\ell \Lambda)
\nn \\
=&\sum_{\ell=0}^{n-m+1} (-z)^\ell q^{\ell(\ell-1)/2}
\displaystyle \begin{bmatrix}
   n-m+1  \\
   \ell
\end{bmatrix}_{q}
F_{n,m-1}^{trig(q^{n-m+1} z)}(\overline{u}|\overline{v}|q^\ell \Lambda).
\end{align}
It is also easy to see the the limit $v_m \rightarrow \infty$ of the right hand side of \eqref{extendedtrigonometricsource}
is
\begin{align}
\sum_{K \subset [1,\dots,n]} (-z)^{|K|} q^{|K|(|K|-1)/2}
(1-q^{|K|} \Lambda )
\prod_{\substack{i \in K \\ j \not\in K}} \frac{q u_i- u_j}{u_i-u_j}
\prod_{\substack{ i \in K  \\ 1 \le k \le m-1 }} \frac{v_k-u_i}{v_k-qu_i},
\end{align}
hence we get
\begin{align}
&\sum_{\ell=0}^{n-m+1} (-z)^\ell q^{\ell(\ell-1)/2}
\displaystyle \begin{bmatrix}
   n-m+1  \\
   \ell
\end{bmatrix}_{q}
F_{n,m-1}^{trig(q^{n-m+1} z)}(\overline{u}|\overline{v}|q^\ell \Lambda) \nn \\
=&\sum_{K \subset [1,\dots,n]} (-z)^{|K|} q^{|K|(|K|-1)/2}
(1-q^{|K|} \Lambda )
\prod_{\substack{i \in K \\ j \not\in K}} \frac{q u_i- u_j}{u_i-u_j}
\prod_{\substack{ i \in K  \\ 1 \le k \le m-1 }} \frac{v_k-u_i}{v_k-qu_i},
\end{align}
which is nothing but \eqref{extendedtrigonometricsource}
with $m$ replaced by $m-1$.

\end{proof}

\subsection{Determinant representations of trigonometric source functions}

We present determinant representations for trigonometric source functions.
The procedure of deriving determinant representations go pararell with the
elliptic case in most cases.

We introduce notations for the left and right hand side of
\eqref{trigonometricKajihara}
\begin{align}
F_{n,m}^{trig(z)}(\overline{u}|\overline{v})&:=
\sum_{K \subset [1,\dots,m]} (-z)^{|K|} 
q^{|K|(|K|-1)/2}
\prod_{\substack{i \in K \\ j \not\in K}}
\frac{v_i-qv_j}{v_i-v_j}
\prod_{\substack{ i \in K  \\ 1 \le k \le n }} \frac{v_i-u_k}{v_i-qu_k},
\label{trigonometricsourcelhs} \\
G_{n,m}^{trig(z)}(\overline{u}|\overline{v})&:=
{\color{black}
\frac{(z;q)_{\infty}}{(q^{m-n}z;q)_{\infty}}
} \sum_{K \subset [1,\dots,n]}
 (-q^{m-n} z)^{|K|} q^{|K|(|K|-1)/2} \prod_{\substack{i \in K \\ j \not\in K}} \frac{q u_i- u_j}{u_i-u_j}
\prod_{\substack{ i \in K  \\ 1 \le k \le m }} \frac{v_k-u_i}{v_k-qu_i}.
\label{trigonometricsourcerhs}
\end{align}
We call these rational functions as trigonometric source functions.
First we rewrite
\eqref{trigonometricsourcelhs} and \eqref{trigonometricsourcerhs} in the following forms
using multiplicative difference operators
\begin{align}
F_{n,m}^{trig(z)}(\overline{u}|\overline{v})=
&\sum_{K \subset [1,\dots,m]} (-zq^{m-n-1})^{|K|} 
q^{-|K|(|K|-1)/2}
\prod_{\substack{i \in K \\ j \not\in K}}
\frac{q^{-1} v_i-v_j}{v_i-v_j}
\prod_{\substack{ i \in K  \\ 1 \le k \le n }} \frac{v_i-u_k}{q^{-1} v_i-u_k}
\nn \\
=&\frac{\prod_{i=1}^m \prod_{k=1}^n (v_i-u_k)}{\prod_{1 \le i < j \le m} (v_j-v_i)}
\prod_{j=1}^m (1-z q^{m-n-1} T_{q,v_j}^{-1}) \frac{\prod_{1 \le i < j \le m} (v_j-v_i)}{\prod_{i=1}^m \prod_{k=1}^n (v_i-u_k)}.
\label{trigdiffone}  \\
G_{n,m}^{trig(z)}(\overline{u}|\overline{v})
=&
{\color{black} \frac{(z;q)_{\infty}}{(q^{m-n}z;q)_{\infty}}
}
\frac{\prod_{i=1}^m \prod_{k=1}^n (v_i-u_k)}{\prod_{1 \le i < j \le n} (u_j-u_i)}
\prod_{j=1}^n (1- z q^{m-n}  T_{q,u_j}) \frac{\prod_{1 \le i < j \le n} (u_j-u_i)}{\prod_{i=1}^m \prod_{k=1}^n (v_i-u_k)}.
\label{trigdifftwo}
\end{align}
Source functions in these forms essentially appear in {\color{black} \cite[(3.1)]{Kajihara} and \cite[(2.4)]{MN} } for example.

We use the $p=0$ degeneration
of the factorization of the elliptic Vandermonde determinants \eqref{ellipticvandermonde}
\begin{align}
\det_{1 \le j,k \le n} \Big( \psi_j^{A_{n-1}}(u_k;0,r) \Big)
=
\Bigg(1-r \prod_{\ell=1}^n u_\ell \Bigg)
\prod_{1 \le i < j \le n} (u_j-u_i),
\label{trigvandermonde}
\end{align}
where
\begin{align}
  \psi_j^{A_{n-1}}(u;0;r)=
  \begin{cases}
    1-(-1)^{n-1}ru^n & \text{if $j=1$,} \\
    u^{j-1}                 & \text{if $j=2,\dots,n$.} \\
  \end{cases}
\end{align}

The $p=0$ case of the Frobenius determinant formula \eqref{Frobenius} is
\begin{align}
\det_{1 \le i,j \le n} \Bigg( 
\frac{v_j-\Lambda u_i}{(1-\Lambda)(v_j-u_i)}
\Bigg)
=\frac{ \Big( \prod_{j=1}^n v_j- \Lambda \prod_{i=1}^n u_i  \Big) \prod_{1 \le i < j \le n} (u_j -u_i) (v_i-v_j)   }
{(1-\Lambda) \prod_{1 \le i,j \le n} (v_j-u_i) }.
\label{trigFrob}
\end{align}

We can extend the determinant representations
introduced by Minin-Pronko-Tarasov \cite{MPT,PT}
to the case when $z$ is generic and $m$ and $n$ are not necessarily the same.

\begin{proposition} \label{proptrigMPT}
We have the following determinant representations:
\begin{align}
F_{n,m}^{trig(z)}(\overline{u}|\overline{v})=&
\frac{ 1-r \prod_{\ell=1}^m v_\ell }{\det_{1 \le i,j \le m} \Big( \sum_{k=1}^m p_{ik} \psi_k^{A_{m-1}}(v_j;0,r) \Big)}
\nn \\
&\times \det_{1 \le i,j \le m} \Bigg(
\Bigg(
\frac{1}{1-r \prod_{\ell=1}^m v_\ell}
\Bigg)^{\delta_{i1}}
 \sum_{k=1}^m p_{ik} \psi_k^{A_{m-1}}(v_j;0,r)  \nn \\
&-q^{m-1} z
\Bigg(
\frac{1}{1-q^{-1} r \prod_{\ell=1}^m v_\ell}
\Bigg)^{\delta_{i1}}
 \sum_{k=1}^m p_{ik} \psi_k^{A_{m-1}}(q^{-1} v_j;0,r)  
\prod_{\ell=1}^n \frac{v_j-u_\ell }{v_j-q u_\ell}
\Bigg), \label{trigMPT} 
\end{align}
\begin{align}
{\color{black} G_{n,m}^{trig(z)}(\overline{u}|\overline{v}) }
=&{\color{black}
\frac{(z;q)_{\infty}}{(q^{m-n}z;q)_{\infty}} }
\frac{  1-r \prod_{\ell=1}^n u_\ell }{ \det_{1 \le i,j \le n} \Big( \sum_{k=1}^n q_{ik}  \psi_k^{A_{n-1}}(u_j;0,r) \Big) }
\nn \\
&\times \det_{1 \le i,j \le n} \Bigg( 
\Bigg(\frac{1}{1-r \prod_{\ell=1}^n u_\ell }
\Bigg)^{\delta_{i1}}
 \sum_{k=1}^n q_{ik} \psi_k^{A_{n-1}}(u_j;0,r) \nn \\
&-q^{m-n} z 
\Bigg(\frac{1}{1-qr \prod_{\ell=1}^n u_\ell }
\Bigg)^{\delta_{i1}}
\sum_{k=1}^n q_{ik} \psi_k^{A_{n-1}}(qu_j;0,r)
 \prod_{\ell=1}^m \frac{v_\ell-u_j}{v_\ell-qu_j}
\Bigg). \label{trigMPTtwo}
\end{align}
Here, $r$, $p_{ij}$ $(i,j=1,\dots,m)$, $q_{ij}$ $(i,j=1,\dots,n)$
are additional parameters such that $\det_{1 \le i,j \le m}(p_{ij}) \not\equiv 0$, $\det_{1 \le i,j \le n}(q_{ij}) \not\equiv 0$.
\end{proposition}
The derivation is the same with
Proposition \ref{propellipticMPT},
using \eqref{trigvandermonde} and \eqref{trigFrob} instead.

The case
$r=0$, $z=1$, $m=n$ of \eqref{trigMPT}, \eqref{trigMPTtwo}
corresponds to the determinants in Minin-Pronko-Tarasov {\color{black} \cite[(4.9), (5.9)]{MPT}.}
We also remark that the parameter $r$ seems to be related with the parameter $\alpha$ in their recent work \cite{PT}.

Specializing to $r=0$ and $p_{ij}=q_{ij}=\delta_{ij}$, we get
the following type of determinant forms,
which was first introduced
by Kostov \cite{Kostovone,Kostovtwo}
and Foda-Wheeler \cite{FodaWheeler}
for the case $F_{n,m}^{trig(1)}(\overline{u}|\overline{v})$.
{See  \color{black} \cite[(3.39)]{Kostovtwo} and \cite[(4.13), (4.19)]{FodaWheeler}.}
\begin{corollary}
We have the following determinant representations:
\begin{align}
F_{n,m}^{trig(z)}(\overline{u}|\overline{v})
=&
\frac{1}{\prod_{1 \le i < j \le m} (v_j-v_i)}
\det_{1 \le i,j \le m}
\Bigg(v_j^{i-1}-z
q^{m-i} v_j^{i-1}
\prod_{\ell=1}^n   \frac{v_j-u_\ell}{v_j-qu_\ell} 
\Bigg),
\label{trigKostovFWtype} \\
G_{n,m}^{trig(z)}(\overline{u}|\overline{v})
=&{\color{black}
\frac{(z;q)_{\infty}}{(q^{m-n}z;q)_{\infty}}
}
\frac{1}{
\prod_{1 \le i < j \le n} (u_j-u_i)} \nn \\
&\times
\det_{1 \le i,j \le n}
\Bigg(u_j^{i-1}-z
q^{m-n+i-1} u_j^{i-1}
\prod_{\ell=1}^m  \frac{v_\ell-u_j}{v_\ell-q u_j}
\Bigg). \label{trigKostovFWtypetwo} 
\end{align}
\end{corollary}

We call \eqref{trigKostovFWtype}, \eqref{trigKostovFWtypetwo}  as scalar product type representation
in this paper
since this kind of determinant representations was first obtained
in \cite{Kostovone,Kostovtwo,FodaWheeler} as a degeneration from
the determinant formulas for scalar products of the six vertex model by Slavnov \cite{Slavnov}.

On the other hand, another type of determinant representations below
was also obtained in Foda-Wheeler {\color{black} \cite[(3.14), (3.18)]{FodaWheeler} }
for the case $F_{n,m}^{trig(1)}(\overline{u}|\overline{v})$.
\begin{proposition} \label{ProptrigFW}
For $n \ge m$,
we have the following determinant representations:
\begin{align}
F_{n,m}^{trig(z)}(\overline{u}|\overline{v})=&
\frac{\prod_{i=1}^m \prod_{k=1}^n (v_i-u_k)}{\prod_{1 \le i < j \le m} (v_j-v_i) \prod_{1 \le i < j \le n} (u_i-u_j)} \det_{1 \le i,j \le n} (Y),
\label{trigFWtype}
\end{align}
where $Y$ is an $n \times n$ matrix whose $(i,j)$-entry is given by
\begin{align}
Y_{ij}=
\left\{
\begin{array}{ll}
\displaystyle \frac{1}{v_i-u_j}-zq^{m-n}  \frac{1}{v_i-qu_j}, & 1 \le i \le m, \ \ \ 1 \le j \le n \\
u_j^{n-i}, & m+1 \le i \le n, \ \ \ 1 \le j \le n
\end{array}
\right.
.
\end{align}
\begin{align}
G_{n,m}^{trig(z)}(\overline{u}|\overline{v})=&
{\color{black}
\frac{(z;q)_{\infty}}{(q^{m-n}z;q)_{\infty}}
}
\frac{\prod_{i=1}^m \prod_{k=1}^n (v_i-u_k)}{\prod_{1 \le i < j \le m} (v_j-v_i) \prod_{1 \le i < j \le n} (u_i-u_j)} \det_{1 \le i,j \le n}  (Z),
\label{trigFWtypetwo}
\end{align}
where $Z$ is an $n \times n$ matrix whose $(i,j)$-entry is given by
\begin{align}
Z_{ij}=
\left\{
\begin{array}{ll}
\displaystyle \frac{1}{v_i-u_j}-q^{m-n} z \frac{1}{v_i-q u_j}, & 1 \le i \le m, \ \ \ 1 \le j \le n \\
u_j^{n-i}-q^{m-i} z u_j^{n-i}, & m+1 \le i \le n, \ \ \ 1 \le j \le n
\end{array}
\right.
.
\end{align}
{\color{black}
For $n < m$,
we have the following determinant representations:
\begin{align}
F_{n,m}^{trig(z)}(\overline{u}|\overline{v})=&
\frac{\prod_{i=1}^m \prod_{k=1}^n (u_k-v_i)}{\prod_{1 \le i < j \le m} (v_i-v_j) \prod_{1 \le i < j \le n} (u_j-u_i)} \det_{1 \le i,j \le m} (U),
\label{trigFWtypethree}
\end{align}
where $U$ is an $m \times m$ matrix whose $(i,j)$-entry is given by
\begin{align}
U_{ij}=
\left\{
\begin{array}{ll}
\displaystyle \frac{1}{u_i-v_j}-zq^{m-n}  \frac{1}{q u_i-v_j}, & 1 \le i \le n, \ \ \ 1 \le j \le m \\
v_j^{m-i}-zq^{i-n-1}v_j^{m-i}, & n+1 \le i \le m, \ \ \ 1 \le j \le m
\end{array}
\right.
.
\end{align}
\begin{align}
G_{n,m}^{trig(z)}(\overline{u}|\overline{v})=&\frac{(z;q)_{\infty}}{(q^{m-n}z;q)_{\infty}}
\frac{\prod_{i=1}^m \prod_{k=1}^n (u_k-v_i)}{\prod_{1 \le i < j \le m} (v_i-v_j) \prod_{1 \le i < j \le n} (u_j-u_i)} \det_{1 \le i,j \le m}  (V),
\label{trigFWtypefour}
\end{align}
where $V$ is an $m \times m$ matrix whose $(i,j)$-entry is given by
\begin{align}
V_{ij}=
\left\{
\begin{array}{ll}
\displaystyle \frac{1}{u_i-v_j}-q^{m-n} z \frac{1}{q u_i- v_j}, & 1 \le i \le n, \ \ \ 1 \le j \le m \\
v_j^{m-i}, & n+1 \le i \le m, \ \ \ 1 \le j \le m
\end{array}
\right.
.
\end{align}
}

\end{proposition}
We call \eqref{trigFWtype}, \eqref{trigFWtypetwo}, \eqref{trigFWtypethree}, \eqref{trigFWtypefour} as domain wall boundary type determinants
as this type of determinant representations was first obtained in
\cite{FodaWheeler} as a limit of the Gaudin-Izergin-Korepin determinant
for the domain wall boundary partition functions of the six vertex model \cite{Gaudin,Izergin,Korepin}.
Here we give a systematic derivation below.
\begin{proof}
{\color{black} We show the case $n \ge m$. The case $n<m$ can be proved in the same way.}
We use the following factorization of the Cauchy-Vandermonde determinants
which can be proved by induction on $n$
\begin{align}
\det_{1 \le i,j \le n} (X)
=
\frac{\prod_{1 \le i < j \le m} (v_j-v_i) \prod_{1 \le i < j \le n} (u_i-u_j) }{\prod_{i=1}^m \prod_{k=1}^n (v_i-u_k)},
\label{identitytogetanotherdet}
\end{align}
where $X$ is an $n \times n$ matrix whose $(i,j)$-entry is given by
\begin{align}
X_{ij}=
\left\{
\begin{array}{ll}
\displaystyle \frac{1}{v_i-u_j}, & 1 \le i \le m, \ \ \ 1 \le j \le n \\
u_j^{n-i}, & m+1 \le i \le n, \ \ \ 1 \le j \le n
\end{array}
\right.
.
\end{align}

Starting from
\eqref{trigdiffone}
and using
\eqref{identitytogetanotherdet},
we have
\begin{align}
F_{n,m}^{trig(z)}(\overline{u}|\overline{v})
=
\frac{\prod_{i=1}^m \prod_{k=1}^n (v_i-u_k)}{\prod_{1 \le i < j \le m} (v_j-v_i)}
\prod_{i=1}^m (1-zq^{m-n-1} T_{q,v_i}^{-1})
\frac{1}{\prod_{1 \le i < j \le n} (u_i-u_j)}
\det_{1 \le i,j \le n} (X). \label{beforetrigFWtype}
\end{align}

Inserting 
$ (1-zq^{m-n-1}  T_{q,v_i}^{-1}) $
in \eqref{beforetrigFWtype}
into the $i$-th row of $\det_{1 \le i,j \le n} (X)$ for $i=1,\dots,m$ gives
\eqref{trigFWtype}.

Starting from
\eqref{trigdifftwo}
and using
\eqref{identitytogetanotherdet},
one gets
\begin{align}
&G_{n,m}^{trig(z)}(\overline{u}|\overline{v}) \nn \\
=&\frac{(z;q)_{\infty}}{(q^{m-n}z;q)_{\infty}}
\frac{\prod_{i=1}^m \prod_{k=1}^n (v_i-u_k)}{\prod_{1 \le i < j \le n} (u_i-u_j)}
\prod_{j=1}^n (1- z q^{m-n} T_{q,u_j}) \frac{1}{\prod_{1 \le i < j \le m} (v_j-v_i)} \det_{1 \le i,j \le n}  (X).
\label{trigexpfornewdettwo}
\end{align}
Inserting 
$ (1-z q^{m-n} T_{q, u_j}) $
in  \eqref{trigexpfornewdettwo}
into the $j$-th column of $\det_{1 \le i,j \le n}  (X)$ for $j=1,\dots,n$ gives
\eqref{trigFWtypetwo}.

\end{proof}

The following type of determinant representations is
a one-parameter extension of a trigonometric analog of
the rational version by Belliard-Slavnov {\color{black} \cite[(4.5)]{BS}.}
\begin{proposition} \label{ProptrigBS}
We have the following determinant representations:
\begin{align}
&F_{n,m}^{trig(z)}(\overline{u}|\overline{v})
=\frac{1-\Delta}{\prod_{1 \le i < j \le m} (v_j-v_i)(\eta_i-\eta_j)} \nn \\
&\times
\det_{1 \le i,j \le m} \Bigg( 
\Bigg(\frac{1}{\prod_{\ell=1}^m v_\ell-\Delta \prod_{\ell=1}^m \eta_\ell}
\Bigg)^{\delta_{i1}} 
\frac{ (v_i-\Delta \eta_j) \prod_{k=1, k \neq j}^m (v_i-\eta_k)
}{1-\Delta} \nn \\
&-z
\Bigg(\frac{1}{ q^{-1} \prod_{\ell=1}^m v_\ell-\Delta \prod_{\ell=1}^m \eta_\ell}
\Bigg)^{\delta_{i1}}
 \frac{ (q^{-1} v_i-\Delta \eta_j) \prod_{k=1, k \neq j}^m (  v_i-q \eta_k) }{1-\Delta}
\prod_{k=1}^n \frac{v_i-u_k}{v_i-q u_k}
\Bigg),
\label{trigBS}
\end{align} 
\begin{align}
&G_{n,m}^{trig(z)}(\overline{u}|\overline{v})
={\color{black}
\frac{(z;q)_{\infty}}{(q^{m-n}z;q)_{\infty}}
}
\frac{1-\Delta}{\prod_{1 \le i < j \le n} (u_j-u_i)(\eta_i-\eta_j)} \nn \\
&\times
\det_{1 \le i,j \le n} \Bigg( 
\Bigg(\frac{1}{\prod_{\ell=1}^n u_\ell-\Delta \prod_{\ell=1}^n \eta_\ell}
\Bigg)^{\delta_{i1}}
 \frac{  (u_i-\Delta \eta_j)  \prod_{k=1,k \neq j}^n (u_i-\eta_k) }{1-\Delta}
\nn \\
&-q^{m-n} z \Bigg(\frac{1}{q \prod_{\ell=1}^n u_\ell-\Delta \prod_{\ell=1}^n \eta_\ell}
\Bigg)^{\delta_{i1}}
\frac{(qu_i-\Delta \eta_j)
\prod_{k=1,k \neq j}^n (q u_i-\eta_k)}{1-\Delta}
\prod_{k=1}^m \frac{u_i-v_k}{q u_i-v_k}
\Bigg). \label{trigBStwo}
\end{align}
Here, $\Delta$, $\eta_i$ $(i=1,\dots,m)$
are additional parameters for
$F_{n,m}^{trig(z)}(\overline{u}|\overline{v})$,
and $\Delta$, $\eta_i$ $(i=1,\dots,n)$
are additional parameters for
$G_{n,m}^{trig(z)}(\overline{u}|\overline{v})$.

\end{proposition}

The derivation of  \eqref{trigBS}, \eqref{trigBStwo} is the same with
Proposition \ref{propellipticBS}, using \eqref{trigFrob} instead.
Taking the limit $\Delta \to \infty$, we get the following forms
which are direct trigonometric analogue of the representations in Belliard-Slavnov \cite{BS}.
\begin{corollary}
We have the following determinant representations:
\begin{align}
F_{n,m}^{trig(z)}(\overline{u}|\overline{v})
=&\frac{1}{\prod_{1 \le i < j \le m} (v_j-v_i)(\eta_i-\eta_j)} \nn \\
&\times
\det_{1 \le i,j \le m} \Bigg( 
\Bigg(
 \prod_{k=1, k \neq j}^m (v_i-\eta_k)
-z
\prod_{k=1, k \neq j}^m (  v_i-q \eta_k)
\prod_{k=1}^n \frac{v_i-u_k}{v_i-q u_k}
\Bigg), \\
G_{n,m}^{trig(z)}(\overline{u}|\overline{v})
=&{\color{black}
\frac{(z;q)_{\infty}}{(q^{m-n}z;q)_{\infty}}
}
\frac{1}{\prod_{1 \le i < j \le n} (u_j-u_i)(\eta_i-\eta_j)} \nn \\
&\times
\det_{1 \le i,j \le n} \Bigg( 
  \prod_{k=1,k \neq j}^n (u_i-\eta_k) 
-q^{m-n} z 
\prod_{k=1,k \neq j}^n (q u_i-\eta_k)
\prod_{k=1}^m \frac{u_i-v_k}{q u_i-v_k}
\Bigg).
\end{align}
\end{corollary}

\subsection{A complex analytic proof of trigonometric source identity}
We give a complex analytic proof of trigonometric source identity.
See also {\color{black} \cite[Lemma 3]{MN} } for example for the original different proof of this identity.
We prove the following polynomial version

\begin{align}
&\sum_{K \subset [1,\dots,m]} (-z)^{|K|} q^{|K|(|K|-1)/2} \prod_{\substack{i \in K \\ j \not\in K}} \frac{v_i-q v_j}{v_i-v_j}
\prod_{\substack{ j \not\in K  \\ 1 \le k \le n }} (v_j-q u_k)
\prod_{\substack{ i \in K  \\ 1 \le k \le n }} (v_i-u_k)
\nonumber \\
=&
{\color{black}
\frac{(z;q)_{\infty}}{(q^{m-n}z;q)_{\infty}} 
}
\nn \\
&\times \sum_{K \subset [1,\dots,n]} (-q^{m-n} z)^{|K|} q^{|K|(|K|-1)/2} \prod_{\substack{i \in K \\ j \not\in K}} \frac{q u_i- u_j}{u_i-u_j}
\prod_{\substack{ j \not\in K  \\ 1 \le k \le m }} (v_k-q u_j)
\prod_{\substack{ i \in K  \\ 1 \le k \le m }} (v_k-u_i).
\label{polytrigKajiharaone}
\end{align}
Let us denote the left hand side and the right hand side of  \eqref{polytrigKajiharaone}
as $P_{n,m}^{trig(z)}(\overline{u}|\overline{v})$ and $Q_{n,m}^{trig(z)}(\overline{u}|\overline{v})$
respectively.
{\color{black}
Note that \eqref{polytrigKajiharaone} is obtained from \eqref{trigonometricKajihara}
by multiplying both hand sides by $\prod_{i=1}^m \prod_{k=1}^n (v_i-qu_k)$
and the identities are equivalent.
Here we rewrite as \eqref{polytrigKajiharaone} and prove this version in this section
since the proof given here is based on the property of polynomial interpolation
and check the equality holds for enough specializations,
and we need to view both hand sides of the identity as polynomials rather than rational functions.

We give the details of the proof when
$m \le n$.
The other case $m>n$ can be proved in the same way,
and we list the corresponding lemmas and proposition
at the end of this subsection.
}

Relations with $F_{n,m}^{trig(z)}(\overline{u}|\overline{v})$ and $G_{n,m}^{trig(z)}(\overline{u}|\overline{v})$
are
\begin{align}
P_{n,m}^{trig(z)}(\overline{u}|\overline{v})&=\prod_{i=1}^m \prod_{k=1}^n (v_i-qu_k) F_{n,m}^{trig(z)}(\overline{u}|\overline{v}), \\
Q_{n,m}^{trig(z)}(\overline{u}|\overline{v})&=\prod_{i=1}^m \prod_{k=1}^n (v_i-qu_k) G_{n,m}^{trig(z)}(\overline{u}|\overline{v}).
\end{align}

We first show the following polynomial properties.
\begin{lemma} \label{trigdegreesymmetrylemma}
$P_{n,m}^{trig(z)}(\overline{u}|\overline{v})$ and $Q_{n,m}^{trig(z)}(\overline{u}|\overline{v})$
are symmetric polynomials in $v_1, \dots, v_m$.
The degree of $v_i \ (i=1,2,\dots,m)$ is at most  $n$
for both polynomials.
\end{lemma}

\begin{proof}
From \eqref{trigKostovFWtype} and \eqref{trigKostovFWtypetwo},
we have the following determinant forms
of $P_{n,m}^{trig(z)}(\overline{u}|\overline{v})$ and $Q_{n,m}^{trig(z)}(\overline{u}|\overline{v})$
\begin{align}
P_{n,m}^{trig(z)}(\overline{u}|\overline{v})
=&
\frac{1}{\prod_{1 \le i < j \le m} (v_j-v_i)}
\det_{1 \le i,j \le m}
\Bigg(v_i^{j-1}
\prod_{k=1}^n   (v_i-qu_k)
-z
q^{m-j} v_i^{j-1}
\prod_{k=1}^n   (v_i-u_k)
\Bigg), \label{polydetformPtrig}
 \\
Q_{n,m}^{trig(z)}(\overline{u}|\overline{v})
=&
{\color{black}
\frac{(z;q)_{\infty}}{(q^{m-n}z;q)_{\infty}}
}
\frac{1}{
\prod_{1 \le i < j \le n} (u_j-u_i)} \nn \\
&\times
\det_{1 \le i,j \le n}
\Bigg(u_i^{j-1} \prod_{k=1}^m  (v_k-q u_i)
-z
q^{m-n+j-1} u_i^{j-1}
\prod_{k=1}^m  (v_k-u_i)
\Bigg), \label{polydetformQtrig}
\end{align}

As for $P_{n,m}^{trig(z)}(\overline{u}|\overline{v})$,
one can see
from \eqref{polydetformPtrig} that $v_i=v_j$ coming from $\prod_{1 \le i < j \le m} (v_j-v_i)^{-1}$
is an apparent pole since two rows of the determinant coincide in that case.
As a polynomial of $v_i$, one can see the contribution to the degree coming from the 
determinant is at most $m+n-1$, and $\prod_{1 \le i < j \le m} (v_j-v_i)^{-1}$
subtracts the degree by $m-1$, hence we find the degree is at most $n$.
Symmetry property is also easy to check.
It is easier to check the properties for $Q_{n,m}^{trig(z)}(\overline{u}|\overline{v})$
from \eqref{polydetformQtrig}.
\end{proof}

Due to Lemma \ref{trigdegreesymmetrylemma},
one notes that it is enough to check $P_{n,m}^{trig(z)}(\overline{u}|\overline{v})=Q_{n,m}^{trig(z)}(\overline{u}|\overline{v})$
for $(n+1)^m$ distinct points in $(v_1,\dots,v_m)$.
One can check the equality for the following $2n \times (2n-1) \times \cdots \times (2n-m+1) (\ge (n+1)^m)$
distinct points:
$v_1=q^{\epsilon_1} u_{p_1}, \ v_2=q^{\epsilon_2} u_{p_2}, \ \dots, v_m=q^{\epsilon_m} u_{p_m}$,
$p_1,\dots,p_m \in \{1,\dots,n \}$, $\epsilon_1,\dots,\epsilon_m \in \{ 0,1 \}$
and $v_i \neq v_j$ for $i \neq j$.
These specializations can be evaluated explicitly.
It is easy to check the following cases which correspond to $p_i=p_j$ for some $i \neq j$.

\begin{lemma} \label{trigvanishinglemma}
If $v_i=u_k$, $v_j=qu_k$ for some $i,j,k \ (i \neq j)$,
we have 
\begin{align}
P_{n,m}^{trig(z)}(\overline{u}|\overline{v})=Q_{n,m}^{trig(z)}(\overline{u}|\overline{v})=0.
\end{align}
\end{lemma}
\begin{proof}
We check the case $k=n$. From the expressions \eqref{polytrigKajiharaone},
one can check the following relations hold
when substituting $v_m=q u_n$ in $P_{n,m}^{trig(z)}(\overline{u}|\overline{v})$
and $Q_{n,m}^{trig(z)}(\overline{u}|\overline{v})$
\begin{align}
&P_{n,m}^{trig(z)}(\overline{u}|\overline{v})|_{v_m=q u_n} \nn \\
=&(1-q) u_n z \prod_{j=1}^{n-1} (q u_n-u_j) \prod_{j=1}^{m-1} q(v_j-u_n)
P_{n-1,m-1}^{trig(z)}(u_1,\dots,u_{n-1}|v_1,\dots,v_{m-1}), \label{subtrigone} \\
&Q_{n,m}^{trig(z)}(\overline{u}|\overline{v})|_{v_m=q u_n} \nn \\
=&(1-q) u_n z \prod_{j=1}^{n-1} (q u_n-u_j) \prod_{j=1}^{m-1} q(v_j-u_n)
Q_{n-1,m-1}^{trig(z)}(u_1,\dots,u_{n-1}|v_1,\dots,v_{m-1}). \label{subtrigtwo}
\end{align}
From the factor $\prod_{j=1}^{m-1} q(v_j-u_n)$ in
\eqref{subtrigone} and \eqref{subtrigtwo},
one notes that $P_{n,m}^{trig(z)}(\overline{u}|\overline{v})|_{v_m=q u_n}$
and $Q_{n,m}^{trig(z)}(\overline{u}|\overline{v})|_{v_m=q u_n}$ vanish after
further substituting $v_j=u_n$ for some $j \ (1 \le j \le m-1)$ gives 0.
The other cases can be checked in the same way or follow by symmetry.
\end{proof}

We next assume $p_i \neq p_j$ for $i \neq j$.
For $I=\{ 1 \le i_1 < i_2 < \cdots < i_{|I|} \le n \}$ and $J=\{1 \le j_1 < j_2 < \cdots < j_{|J|} \le n \}$
such that $I \cup J$ is a set of $m$ distinct integers in $[ 1,\dots,n ]$,
let us denote by $\overline{v}=\{ \overline{u}_I, q \overline{u}_J \}$ the 
following type of substitution
\begin{align}
v_{q_1}=u_{i_1}, \dots, v_{q_{|I|}}=u_{i_{|I|}}, v_{r_1}=q u_{j_1}, \dots, v_{r_{|J|}}=q u_{j_{|J|}},
\end{align}
for some $\{ q_1, \dots, q_{|I|}, r_1,\dots, r_{|J|} \}=[1,\dots,m ]$.
Note $|I|+|J|=m$.

We have the following explicit evaluations.
\begin{proposition} \label{trigPQspecializations}
We have
\begin{align}
P_{n,m}^{trig(z)}(\overline{u}|\{ \overline{u}_I, q \overline{u}_J \})
=(-z)^{|J|} q^{|I||J|+|J|(|J|-1)/2} \prod_{ \substack{ i \in I \\ j \in J  }  } (u_j-u_i)
\prod_{ \substack{ i \in I \\ 1 \le j \le n  }  } (u_i-q u_j)
 \prod_{ \substack{ j \in J \\ k \in [1,\dots,n] \backslash I  }  } (q u_j-u_k), \\
Q_{n,m}^{trig(z)}(\overline{u}|\{ \overline{u}_I, q \overline{u}_J \})
=(-z)^{|J|} q^{|I||J|+|J|(|J|-1)/2} \prod_{ \substack{ i \in I \\ j \in J  }  } (u_j-u_i)
\prod_{ \substack{ i \in I \\ 1 \le j \le n  }  } (u_i-q u_j)
 \prod_{ \substack{ j \in J \\ k \in [1,\dots,n] \backslash I  }  } (q u_j-u_k).
\end{align}
\end{proposition}
\begin{proof}
After the substitution
$
v_{q_1}=u_{i_1}, \dots, v_{q_{|I|}}=u_{i_{|I|}}, v_{r_1}=qu_{j_1}, \dots, v_{r_{|J|}}=qu_{j_{|J|}}$,
one notes from the factors
$\displaystyle \prod_{\substack{ j \not\in K  \\ 1 \le k \le n }} (v_j-q u_k)$
and
$\displaystyle \prod_{\substack{ i \in K  \\ 1 \le k \le n }} (v_i-u_k)$
that only the summand corresponding to $K=\{ r_1,\dots,r_{|J|} \}$  survives
and we get
\begin{align}
&P_{n,m}^{trig(z)}(\overline{u}|\{ \overline{u}_I, q \overline{u}_J \}) \nn \\
=&(-z)^{|J|} q^{|J|(|J|-1)/2} \prod_{\substack{i \in J \\ j \in  I}} \frac{q(u_i-u_j)}{q u_i-u_j}
\prod_{\substack{ j \in I  \\ 1 \le k \le n }} (u_j-q u_k)
\prod_{\substack{ j \in J  \\ 1 \le k \le n }} (q u_j-u_k) \nn \\
=&(-z)^{|J|} q^{|I||J|+|J|(|J|-1)/2} \prod_{ \substack{ i \in I \\ j \in J  }  } (u_j-u_i)
\prod_{ \substack{ i \in I \\ 1 \le j \le n  }  } (u_i-q u_j)
 \prod_{ \substack{ j \in J \\ k \in [1,\dots,n] \backslash I  }  } (q u_j-u_k).
\end{align}

Next, we check the case for 
$Q_{n,m}^{trig(z)}(\overline{u}|\{ \overline{u}_I, q \overline{u}_J \})$.
Substituting
$
v_{q_1}=u_{i_1}, \dots, v_{q_{|I|}}=u_{i_{|I|}}, v_{r_1}=q u_{j_1}, \dots, v_{r_{|J|}}=q u_{j_{|J|}}$,
we note from the factor
$\displaystyle \prod_{\substack{ j \in K  \\ 1 \le k \le m }} (v_k-u_j)$
that summands labeled by $K$ satisfying $K \cap I \neq \phi$ vanish,
and from the factor
$\displaystyle \prod_{\substack{ j \not\in K  \\ 1 \le k \le m }} (v_k-q u_j)$
we note the summands vanish unless $J \subset K$.
Hence the sum can be restricted to
$\displaystyle \sum_{J \subset K \subset [1,\dots,n] \backslash I}$
and after simplifications, we find 
$Q_{n,m}^{trig(z)}(\overline{u}|\{ \overline{u}_I, q \overline{u}_J \})$ can be rewritten as
\begin{align}
&Q_{n,m}^{trig(z)}(\overline{u}|\{ \overline{u}_I, q \overline{u}_J \}) \nn \\
=&q^{|I||J|+|J|(|J|-1)/2}
{\color{black}
\frac{(z;q)_{\infty}}{(q^{m-n}z;q)_{\infty}}
}
\prod_{ \substack{ i \in I \\ j \in J  }  } (u_j-u_i)
\prod_{ \substack{ i \in I \\ 1 \le j \le n  }  } (u_i- q u_j)
 \prod_{ \substack{ j \in J \\ k \in [1,\dots,n] \backslash I  }  } (q u_j-u_k)
(-z)^{|J|} \nn \\
&\times \sum_{J \subset K \subset [1,\dots,n] \backslash I} 
(-z)^{|K|-|J|} q^{-(|K|-|J|)(|K|-|J|+1)/2}
\prod_{\substack{i \in K \backslash J \\ j \in [1,\dots,n] \backslash \{I,K \} }}
\frac{u_i-q^{-1} u_j}{u_i-u_j}. \label{trigbeforesimplyfyingq}
\end{align}
Finally, using Cauchy's $q$-binomial identity \eqref{qbinomial}
we note the following factorization
\begin{align}
&\sum_{J \subset K \subset [1,\dots,n] \backslash I} 
(-z)^{|K|-|J|} q^{-(|K|-|J|)(|K|-|J|+1)/2}
\prod_{\substack{i \in K \backslash J \\ j \in [1,\dots,n] \backslash \{I,K \} }}
\frac{u_i-q^{-1} u_j}{u_i-u_j} \nn \\
=&\sum_{\ell=|J|}^{n-|I|}
\sum_{\substack{ K: |K|=\ell   \\ J \subset K \subset [1,\dots,n] \backslash I }}
(-z)^{\ell-|J|} q^{-(\ell-|J|)(\ell-|J|+1)/2}
\prod_{\substack{i \in K \backslash J \\ j \in [1,\dots,n] \backslash \{I,K \} }}
\frac{u_i-q^{-1} u_j}{u_i-u_j} \nn \\
=&\sum_{\ell=|J|}^{n-|I|}
(-z)^{\ell-|J|} q^{-(\ell-|J|)(\ell-|J|+1)/2}
\begin{bmatrix}
   n-|I|-|J|  \\
   \ell- |J |
\end{bmatrix}_{q^{-1}}
=\prod_{j=1}^{n-|I|-|J|} (1-q^{-j} z) =\prod_{j=1}^{n-m} (1-q^{-j} z)
. \label{trigbeforesimplyfyingqtwo}
\end{align}
Note that we also used the following $q$-identity
\begin{align}
\sum_{K \subset [1,\dots,n]: |K|=\ell }
\prod_{\substack{i \in K \\ j \in [1,\dots,n] \backslash K  }}
\frac{u_i-q^{-1} u_j}{u_i-u_j}=\begin{bmatrix}
   n  \\
   \ell
\end{bmatrix}_{q^{-1}},
\label{oneqidentity}
\end{align}
in the second equality of \eqref{trigbeforesimplyfyingqtwo}.
\eqref{oneqidentity} can be checked by noting that the left hand side is a bounded function and
the possible poles $u_i=u_j \ (i \neq j)$ are apparent singularities, hence using Liouville's theorem
and taking the limit $u_1 \to \infty, u_2 \to \infty,\dots,u_n \to \infty$ in this order,
we get
\begin{align}
\sum_{K \subset [1,\dots,n]: |K|=\ell }
\prod_{\substack{i \in K \\ j \in [1,\dots,n] \backslash K  }}
\frac{u_i-q^{-1} u_j}{u_i-u_j}
=
\sum_{K \subset [1,\dots,n]: |K|=\ell} q^{- \# \{(i,j) \ | \ i \in K, \ j \in [1,\dots,n] \backslash K, \ i>j  \} }.
\end{align}
Hence it remains to show
\begin{align}
\sum_{K \subset [1,\dots,n]: |K|=\ell} q^{- \# \{(i,j) \ | \ i \in K, \ j \in [1,\dots,n] \backslash K, \ i>j  \} }
=\displaystyle \begin{bmatrix}
   n  \\
   \ell
\end{bmatrix}_{q^{-1}},
\label{qbinomialidentity}
\end{align}
which can be proved by induction on $n$.
Denote the left hand side and right hand side of
\eqref{qbinomialidentity}
as $f(n,\ell)$ and $g(n,\ell)$ respectively.
It is easy to show they satisfy the recursive relations
\begin{align}
f(n,\ell)&=f(n-1,\ell)+f(n-1,\ell-1)q^{\ell-n}, \\
g(n,\ell)&=f(n-1,\ell)+f(n-1,\ell-1)q^{\ell-n},
\end{align}
and $f(n,0)=g(n,0)=1$ and $f(1,1)=g(1,1)=1$ which correspond to the initial conditions.

Combining
\eqref{trigbeforesimplyfyingq} and \eqref{trigbeforesimplyfyingqtwo} gives
\begin{align}
&Q_{n,m}^{trig(z)}(\overline{u}|\{ \overline{u}_I, \overline{u}_J+c \}) \nn \\
=&(-z)^{|J|} q^{|I||J|+|J|(|J|-1)/2} \prod_{ \substack{ i \in I \\ j \in J  }  } (u_j-u_i)
\prod_{ \substack{ i \in I \\ 1 \le j \le n  }  } (u_i-q u_j)
 \prod_{ \substack{ j \in J \\ k \in [1,\dots,n] \backslash I  }  } (q u_j-u_k).
\end{align}

\end{proof}

By showing
Lemma \ref{trigvanishinglemma} and Proposition \ref{trigPQspecializations},
we checked \eqref{polytrigKajiharaone} for enough specializations of $(v_1,\dots,v_m)$ and
hence the identity \eqref{polytrigKajiharaone} itself holds.

{\color{black}
At the end of this subsection, we briefly explain the case $m>n$.
The proof can be done in the same way with the case $m \leq n$,
and we check \eqref{polytrigKajiharaone} for enough specializations of $(u_1,\dots,u_n)$.
We write down lemmas and proposition corresponding to
Lemma \ref{trigdegreesymmetrylemma},
Lemma \ref{trigvanishinglemma} and Proposition \ref{trigPQspecializations}.
\begin{lemma} 
$P_{n,m}^{trig(z)}(\overline{u}|\overline{v})$ and $Q_{n,m}^{trig(z)}(\overline{u}|\overline{v})$
are symmetric polynomials in $u_1, \dots, u_n$.
The degree of $u_i \ (i=1,2,\dots,n)$ is at most  $m$
for both polynomials.
\end{lemma}
\begin{lemma} 
If $u_i=v_k$, $u_j=q^{-1}v_k$ for some $i,j,k \ (i \neq j)$,
we have 
\begin{align}
P_{n,m}^{trig(z)}(\overline{u}|\overline{v})=Q_{n,m}^{trig(z)}(\overline{u}|\overline{v})=0.
\end{align}
\end{lemma}
For $I=\{ 1 \le i_1 < i_2 < \cdots < i_{|I|} \le m \}$ and $J=\{1 \le j_1 < j_2 < \cdots < j_{|J|} \le m \}$
such that $I \cup J$ is a set of $n$ distinct integers in $[ 1,\dots,m ]$,
let us denote by $\overline{u}=\{ \overline{v}_I, q^{-1} \overline{v}_J \}$ the 
following type of substitution
\begin{align}
u_{q_1}=v_{i_1}, \dots, u_{q_{|I|}}=v_{i_{|I|}}, u_{r_1}=q^{-1} v_{j_1}, \dots, u_{r_{|J|}}=q^{-1} v_{j_{|J|}},
\end{align}
for some $\{ q_1, \dots, q_{|I|}, r_1,\dots, r_{|J|} \}=[1,\dots,n ]$.
Note $|I|+|J|=n$.
\begin{proposition}
We have
\begin{align}
P_{n,m}^{trig(z)}(\{ \overline{v}_I, q^{-1} \overline{v}_J \}|\overline{v})
=&(-z)^{|J|} q^{-|J|(|J|+1)/2} \frac{(z;q)_{\infty}}{(q^{m-n}z;q)_{\infty}} \nn \\
&\times \prod_{ \substack{ i \in I \\ j \in J  }  } (v_i-v_j)
\prod_{ \substack{ i \in I \\ 1 \le j \le m  }  } (v_j-q v_i)
 \prod_{ \substack{ j \in J \\ k \in [1,\dots,m] \backslash I  }  } (q v_k-v_j), \\
Q_{n,m}^{trig(z)}(\{ \overline{v}_I, q^{-1} \overline{v}_J \}|\overline{v})
=&(-z)^{|J|} q^{-|J|(|J|+1)/2} \frac{(z;q)_{\infty}}{(q^{m-n}z;q)_{\infty}} \nn \\
&\times \prod_{ \substack{ i \in I \\ j \in J  }  } (v_i-v_j)
\prod_{ \substack{ i \in I \\ 1 \le j \le m  }  } (v_j-q v_i)
 \prod_{ \substack{ j \in J \\ k \in [1,\dots,m] \backslash I  }  } (q v_k-v_j).
\end{align}
\end{proposition}
}

\subsection{Geometric derivation of trigonometric source identity}
\label{subsec:geom}
We give a geometric derivation of the trigonometric source identity.
We assume $n \le m$ and set $t:=q^{-1}$ and redefine $t^{-1} v_i$ $(i=1,\dots,m)$ and $t^{m-1}z$ as $v_i$ and $z$.
Then \eqref{trigonometricKajihara} is rewritten as

\begin{align}
&\sum_{K \subset [1,\dots,m]} (-z)^{|K|} t^{|K|(|K|-1)/2} \prod_{\substack{i \in K \\ j \not\in K}} \frac{1-t v_i/v_j}{1-v_i/v_j}
\prod_{\substack{ i \in K  \\ 1 \le k \le n }} \frac{1-t u_k/v_i}{1-u_k/v_i}
\nonumber \\
=&\prod_{j=1}^{m-n} (1-t^{m-j} z) \sum_{K \subset [1,\dots,n]} (-z)^{|K|} t^{|K|(|K|-1)/2}
\prod_{\substack{i \in K \\ j \not\in K}} \frac{1- t u_j/u_i}{1-u_j/u_i}
\prod_{\substack{ i \in K  \\ 1 \le k \le m }} \frac{1-t u_i/v_k}{1-u_i/v_k}.
\label{geometrictrigonometricKajihara}
\end{align}
We show the version \eqref{geometrictrigonometricKajihara} of the trigonometric source identity.
Applying the $q$-binomial identity \eqref{qbinomial},
we have
\begin{align}
\prod_{j=1}^{m-n} (1-t^{m-j} z)=\sum_{k=0}^{m-n} t^{nk} t^{k(k-1)/2} (-z)^k 
\displaystyle \begin{bmatrix}
   m-n  \\
   k
\end{bmatrix}_{t}.
\label{qbinomialforgeometry}
\end{align}
Using
\eqref{qbinomialforgeometry}
and expanding in terms of $-z$ and taking the 
coefficients of $(-z)^{\ell}$ of both hand sides of 
\eqref{geometrictrigonometricKajihara}, we have
\begin{align}
&t^{\ell(\ell -1)/2} \sum_{\substack{K \subset 
[1,\dots, m] \\ \lvert K \rvert=\ell }} 
\prod_{\substack{i \in K \\ j \in [1,\dots, m] \setminus 
K}} 
\frac{1-  t v_i/v_j }{1-v_{i}/v_{j} }
\prod_{\substack{ i \in K  \\ 1 \le k \le n }} 
\frac{1- tu_{k}/v_{i}}{ 1 - u_{k}/v_{i}} \nn \\
=&\sum_{k=0}^{\ell} t^{nk} t^{k(k-1)/2}
\displaystyle \begin{bmatrix}
   m-n  \\
   k
\end{bmatrix}_{t}
\times  t^{(\ell -k)(\ell -k-1)/2}
\sum_{\substack{K \subset [1,\dots,n] \\ \lvert K \rvert
=\ell -k}} 
\prod_{\substack{i \in K \\ j \in [1, \ldots, n] 
\setminus  K}} 
\frac{ 1 - tu_{j}/u_{i} }{1- u_{j}/u_{i} }
\prod_{\substack{ i \in K  \\ 1 \le k^\prime \le m }} 
\frac{1 - t u_{i}/v_{k^\prime} }{ 1 - u_{i}/v_{k^\prime} }.
\label{coeffgeomsource}
\end{align}
We give a geometric derivation of \eqref{coeffgeomsource}.
For the description,
it is useful to introduce another version of
$q$-integers, $q$-factorials and $q$-binomials
$\displaystyle (n)_q:={q^{n/2}-q^{-n/2} \over 
q^{1/2}-q^{-1/2}}$, $(n)_q!:=\prod_{j=1}^n (j)_q$
{\color{black}
where $n$ is a nonnegative integer.}
Note the following relations $(n)_q=q^{(1-n)/2}[n]_q$, 
$(n)_q!= q^{(n-n^2)/4}[n]_q!$.

We apply the wall-crossing formula 
\cite[Theorem 3.6]{OS}
to the following framed quiver
$Q \colon$
\begin{center}
\includegraphics[scale=1]{source}
\end{center}
We consider $Q$-representations on $V=\C^{\ell}$ with the
framing $W_{1}=\mathbb{C}^{n}
\to V$ and 
the co-framing $V \to W_{2}=\mathbb{C}^{m}$.
Such $Q$-representations are parametrized by
\[
\mathbb{M}=\Hom(W_{1}, \C^{\ell}) \oplus \Hom(\C^{\ell}, 
W_{2}).
\]
The general linear group $\GL(\C^{\ell})$ acts on 
$\mathbb{M}$ by $g \cdot (z, w)=(gz, wg^{-1})$ for 
$g \in \GL(\C^{\ell})$ and $(z, w) \in \mathbb{M}$, and 
$\GL(\C^{\ell})$-orbits are regarded as isomorphism classes
of $Q$-representations.

We consider the character $\GL(\C^{\ell}) \to 
\C^{\ast}, g \mapsto \text{det}(g)^{\pm 1}$
{\color{black} which corresponds to $\pm 1$-stability 
condition
lying the chamber with the 
bottom (resp. top) boundary wall $0$}.
These define semistable loci 
\begin{align*}
\mathbb{M}^{+1} &= \lbrace (z, w)
\in \mathbb{M} \mid w \text{ is injective} \rbrace \\
\mathbb{M}^{-1} &= \lbrace (z, w)
\in \mathbb{M} \mid z \text{ is surjective} \rbrace. 
\end{align*}
{\color{black}
Replacing $n$ and $m$, we have a bijection 
$\mathbb{M}^{+1} \to \mathbb{M}^{-1}$
by sending $(z, w)$ to the dual $({}^t w, {}^t z)$.} 

The framed quiver moduli $M^{\pm}_{Q} (\ell)$
of $\pm 1$-stable
$Q$-representations on $\C^{\ell}$ are given by
$M^{\pm }_{Q} (\ell)=\mathbb{M}^{\pm 1} / 
\text{GL}(\C^{\ell})$.
We consider the diagonal torus 
$\mathbb{T}=(\C^{\ast} )^{n+m}$
of $\text{GL}(W_{1}) \times \text{GL}(W_{2})$, and 
$\mathbb{T}$-action on $M^{\pm}_{Q}(\ell)$ induced by the
natural $\text{GL}(W_{1}) \times \text{GL}(W_{2})$-action
on $\mathbb{M}$ by the conjugation.
Concretely, we have $(u, v) \cdot (z,w)=(zu^{-1}, vw)$
for $(u,v) \in \mathbb{T}$ and $(z,w) \in \mathbb{M}$ 
where
$u \in (\C^{\ast})^{n}$ and $v \in (\C^{\ast})^{m}$.

For a vector bundle $\mathcal{E}$, we set
$\wedge_{-t} \mathcal{E}=
\sum_{i} {-t}^{i} \wedge^{i} \mathcal{E}$, and
define the $K$-theoretic Euler class 
\[
\eukt(\mathcal{E})= 
\ch( \wedge_{-t} \mathcal{E}).
\]
We take $\mathbb{T}$-equivariant $K$-theoretic integrals 
\[
\intk_{M^{\pm}_{Q}(\ell)} 
\eukt(TM^{\pm}_{Q}(\ell))
=
\int_{M^{\pm}_{Q}(\ell)} 
\eukt(TM^{\pm}_{Q}(\ell)) \cdot
\td_{M^{\pm}_{Q}(\ell)}
\]
where $\td_{M^{\pm}_{Q}(\ell)}=
{\eu(TM^{\pm}_{Q}(\ell)) 
\over \euk(TM^{\pm}_{Q}(\ell))}$ is the Todd class
and $\eu(TM^{\pm}_{Q}(\ell))$ is the usual Euler class. 
This equivariant integral computes the equivariant 
$\chi_{t}$-genus of the moduli space.
We see that $\intk_{M^{+}_{Q}(\ell)} 
\eukt(TM^{+}_{Q}(\ell))$ is equal to
\begin{align*}
&\sum_{\substack{K \subset [1,\dots, m] \\ \lvert K 
\rvert=\ell }} 
\prod_{\substack{i \in K \\ j \in [1,\dots, m] 
\setminus K}} 
\frac{1-  t v_i/v_j }{1-v_{i}/v_{j} }
\prod_{\substack{ i \in K  \\ 1 \le k \le n }} 
\frac{1- tu_{k}/v_{i}}{ 1 - u_{k}/v_{i}},
\end{align*}
and $\intk_{M^{-}_{Q}(\ell)} 
\eukt(M^{-}_{Q}(\ell))$ is 
equal to
\begin{align*}
& \sum_{\substack{K \subset [1,\dots,n] \\ \lvert K \rvert=\ell }} 
\prod_{\substack{i \in K \\ j \in [1, \ldots, n] \setminus  K}} 
\frac{ 1 - tu_{j}/u_{i} }{1- u_{j}/u_{i} }
\prod_{\substack{ i \in K  \\ 1 \le k \le m }} 
\frac{1 - t u_{i}/v_{k} }{ 1 - u_{i}/v_{k} }.
\end{align*}
By \cite[Theorem 3.6]{OS}, we have
\begin{align*}
&
\intk_{M^{+}_{Q}(\ell)} 
\eukt(M^{+}_{Q}(\ell)) 
- \intk_{M^{-}_{Q}(\ell)} 
\eukt(M^{-}_{Q}(\ell)) \\
&=
\sum_{k=1}^{\ell} 
\sum_{{\mbi{\mk I} \in \Dec(\ell) \atop |\mbi{\mk I}|=k}} 
{ [\ell -k]_{t}!  \over  [\ell]_{t} ! }
\prod_{i=1}^{j} {[ d_{i} -1]_{t}! \over t-1} 
\gamma^{K}_{d_{i}} (t) 
t^{-(\ell -i)d_{i}}( t^{s^K( \mk I_{i}, \mbi{\mk I}_{> i}) +  md_{i}} - 
t^{s^K( \mbi{\mk I}_{> i}, \mk I_{i}) + nd_{i} } ) \nn \\
&\times
\intk_{ M^{-}_{Q}(\ell  -k)} 
\eukt(M^{-}_{Q}(\ell  -k)).
\end{align*}
where $\Dec(\ell)$ is the set of collections 
$\mbi{\mk I} = ( \mk I_{1}, \ldots, \mk I_{j} )$ 
satisfying $\mk I_{1} \sqcup \cdots \sqcup \mk I_{j} 
\subset [1,\dots, \ell]$ and
\[
\min(\mk I_{1}) > \cdots > \min(\mk I_{j}).
\]
Here we set 
$d_{i} = |\mk I_{i}|$, and 
$|\mbi{\mk I}|=d_{1} + \cdots + d_{j}$, 
\[
s^K(\mk I_{1}, \mk I_{2}) = \left| \lbrace (i, j) 
\in \mk I_{1} \times \mk I_{2} \mid i < j \rbrace 
\right|.
\]
We have
\begin{align}
\label{gamma}
\gamma^{K}_{d}(t) = \intk_{M_{Q^{\sharp}}(d)} \eukt(M_{Q^{\sharp}}(d))  = 
\begin{cases} 1 & d=1\\ 0 & d \neq 1 \end{cases}
\end{align}
for the following framed quiver $Q^{\sharp} \colon$
\begin{center}
\includegraphics[scale=1]{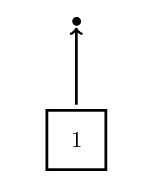}
\end{center}
{
\color{black}
Here we do not use the equivariant parameter
induced from the framing of $Q^{\sharp}$.
}

Hence we have
\begin{align*}
&
\intk_{M^{+}_{Q}(\ell)} 
\eukt(M^{+}_{Q}(\ell)) 
- \intk_{M^{-}_{Q}(\ell)} 
\eukt(M^{-}_{Q}(\ell)) 
\\
&=
\sum_{k=1}^{\ell} 
\sum_{\substack{\mbi{\mk I} \in \Dec(\ell) \\ |\mbi{\mk I}|=k \\ |\mk I_{1}|=\cdots = |\mk I_{k}|=1}} 
\frac{ [\ell -k]_{t}!  }{ [\ell]_{t} ! }
\prod_{i=1}^{k} \frac{t^{ s^K( \mk I_{i},  \mbi{\mk I}_{> i}) + m} - 
t^{ s^K( \mbi{\mk I}_{> i}, \mk I_{i}) + n} }{t-1} \cdot t^{-\ell +i} 
\intk_{ M^{-}_{Q}(\ell  -k)} \eukt(M^{-}_{Q}(\ell  -k)).
\end{align*}
Since $\sum_{i=1}^{k} i-\ell  = \frac{ k(k+1)-2kl}{2}=
\frac{(\ell -k)(\ell -k-1) -\ell  (\ell -1)}{2}$, 
this is equivalent to
\begin{align}
&
t^{\ell (\ell -1)/2} \intk_{M^{+}_{Q}(\ell)} 
\eukt(M^{+}_{Q}(\ell)) 
- t^{\ell (\ell -1)/2} \intk_{M^{-}_{Q}(\ell)} 
\eukt(M^{-}_{Q}(\ell)) 
\nn \\
&=
\sum_{k=1}^{\ell} 
\sum_{\mbi{\mk I} \in \Dec(\ell) \atop 
{|\mbi{\mk I}|=k \atop 
|\mk I_{1}|=\cdots = |\mk I_{k}|=1} }
{ [\ell -k]_{t}! \over [\ell]_{t} ! }
\prod_{i=1}^{k} 
\frac{t^{ s^K( \mk I_{i}, \mbi{\mk I}_{> i}) + m} - 
t^{  s^K( \mbi{\mk I}_{> i}, \mk I_{i}) + n} }{t-1} \nn \\
&\times
t^{\frac{(\ell -k)(\ell -k-1)}{2}}
\intk_{ M^{-}_{Q}(\ell  -k)} 
\eukt(M^{-}_{Q}(\ell  -k)). 
\label{equivalentcoeffgeomsource}
\end{align}
We now transform
\eqref{equivalentcoeffgeomsource} to
\eqref{coeffgeomsource}.
To this end, we investigate the sum in the right hand side of 
\eqref{equivalentcoeffgeomsource}.
Using $ s^K( \mbi{\mk I}_{> i}, \mk I_{i})+s^K( \mk I_{i}, \mbi{\mk I}_{> i})=\ell -i$,
we note the factors in \eqref{equivalentcoeffgeomsource} can be rewritten as
\begin{align}
&\frac{ [\ell -k]_{t}!  }{ [\ell]_{t} ! }
\prod_{i=1}^{k} 
\frac{t^{ s^K( \mk I_{i}, \mbi{\mk I}_{> i}) + m} - 
t^{  s^K( \mbi{\mk I}_{> i}, \mk I_{i}) + n} }{t-1}
\nn \\
=&t^{(\ell -k)^2/4-(\ell -k)/4-\ell ^2/4+l/4} \frac{ (\ell -k)_{t}!  }{ (\ell)_{t} ! } \nn \\
&\times \prod_{i=1}^{k} 
t^{(m+n-1+s^K( \mbi{\mk I}_{> i}, \mk I_{i})+s^K( \mk I_{i}, \mbi{\mk I}_{> i}))/2}
(m-n-s^K( \mbi{\mk I}_{> i}, \mk I_{i})+s^K( \mk I_{i}, \mbi{\mk I}_{> i}))_t
\nn \\
=&t^{(k^2+k)/4-\ell k/2} { (\ell -k)_{t}! \over (\ell)_{t} ! } 
\prod_{i=1}^{k} 
t^{(m+n-1+\ell-i)/2}
(m-n-s^K( \mbi{\mk I}_{> i}, \mk I_{i})+s^K( \mk I_{i}, \mbi{\mk I}_{> i}))_t \nn \\
=& t^{(m+n-1)k/2} \frac{ (\ell -k)_{t}!  }{ (\ell)_{t} ! }
\prod_{i=1}^{k} 
(m-n-s^K( \mbi{\mk I}_{> i}, \mk I_{i})+s^K( \mk I_{i}, \mbi{\mk I}_{> i}))_t.
\end{align}
Hence the right hand side of 
\eqref{equivalentcoeffgeomsource} can be rewritten as
\begin{align}
&\sum_{k=1}^{\ell} 
t^{(m+n-1)k/2} {(\ell -k)_{t}! \over (\ell)_{t} ! }
\sum_{\mbi{\mk I} \in \Dec(\ell) \atop
{|\mbi{\mk I}|=k \atop |\mk I_{1}|=\cdots = 
|\mk I_{k}|=1}} 
\prod_{i=1}^{k} 
(m-n-s^K( \mbi{\mk I}_{> i}, \mk I_{i})+s^K( \mk I_{i}, \mbi{\mk I}_{> i}))_t
\nn \\
&\times
t^{\frac{(\ell -k)(\ell -k-1)}{2}}
\intk_{ M^{-}_{Q}(\ell  -k)} \eukt(M^{-}_{Q}(\ell  -k)).
\label{equivalentcoeffgeomsourcerewrite}
\end{align}
Writing the subsets as $\mk I_{i}=\{ h_i \}$ $(\ell  \ge h_1 > \cdots > h_k \ge 1)$,
we have
\begin{align}
s^K( \mbi{\mk I}_{> i}, \mk I_{i})&=s^K([\ell] \backslash \{h_1,\dots,h_i \},\{ h_i \})=h_i-1, \\
s^K( \mk I_{i}, \mbi{\mk I}_{> i})&=s^K(\{ h_i \},[\ell] \backslash \{h_1,\dots,h_i \})=\ell -h_i-i+1,
\end{align}
and hence the second sum in
\eqref{equivalentcoeffgeomsourcerewrite}
can be simplified using 
\eqref{qidentityforgeometricderivation} 
in the the proposition below as
\begin{align}
&\sum_{\substack{\mbi{\mk I} \in \Dec(\ell) \\ |\mbi{\mk I}|=k \\ |\mk I_{1}|=\cdots = |\mk I_{k}|=1}} 
\prod_{i=1}^{k} 
(m-n-s^K( \mbi{\mk I}_{> i}, \mk I_{i})+s^K( \mk I_{i}, \mbi{\mk I}_{> i}))_t \nn 
\\
=&
\sum_{\ell \ge h_1 > \cdots > h_k \ge 1} \prod_{i=1}^k (\ell -2h_i-i+2+m-n)_t
=\frac{(m-n)_t! (\ell)_t!}{(k)_t!(m-n-k)_t!(\ell -k)_t!}. 
\label{appliedqidentityforgeometricderivation}
\end{align}
Therefore,
\eqref{equivalentcoeffgeomsourcerewrite}
can be simplified as
\begin{align}
&\sum_{k=1}^{\ell} 
t^{(m+n-1)k/2} {(m-n)_t! \over (k)_t!(m-n-k)_t!}
\times
t^{{(\ell -k)(\ell -k-1) \over 2}}
\intk_{ M^{-}_{Q}(\ell  -k)} 
\eukt(M^{-}_{Q}(\ell  -k)) \nn \\
=&\sum_{k=1}^{\ell} 
t^{nk} t^{k(k-1)/2} 
\displaystyle \begin{bmatrix}
   m-n  \\
   k
\end{bmatrix}_{t}
\times
t^{\frac{(\ell -k)(\ell -k-1)}{2}}
\intk_{ M^{-}_{Q}(\ell  -k)} \eukt(M^{-}_{Q}(\ell  -k)). \label{equivalentcoeffgeomsourcerewritetwo}
\end{align}
The right hand side of
\eqref{equivalentcoeffgeomsource} is simplified as
\eqref{equivalentcoeffgeomsourcerewritetwo},
and moving the second term in the left hand side to the right hand side,
\eqref{equivalentcoeffgeomsource} can be rewritten as
\begin{align}
&t^{\ell (\ell -1)/2} 
\intk_{M^{+}_{Q}(\ell)} 
\eukt(M^{+}_{Q}(\ell)) \nn \\
=&\sum_{k=0}^{\ell} 
t^{nk} t^{k(k-1)/2} 
\displaystyle \begin{bmatrix}
   m-n  \\
   k
\end{bmatrix}_{t}
\times
t^{\frac{(\ell -k)(\ell -k-1)}{2}}
\intk_{ M^{-}_{Q}(\ell  -k)} \eukt(M^{-}_{Q}(\ell  -k)).
\end{align}
This is nothing but the identity
\eqref{coeffgeomsource}.

Finally we show the following proposition used in 
\eqref{appliedqidentityforgeometricderivation}.
\begin{proposition}
We have
\begin{align}
\sum_{\ell \ge h_1 > \cdots > h_k \ge 1} \prod_{i=1}^k (\ell -2h_i-i+2+m-n)_t
=\frac{(m-n)_t! (\ell)_t!}{(k)_t!(m-n-k)_t!(\ell -k)_t!}. 
\label{qidentityforgeometricderivation}
\end{align}
\end{proposition}
\begin{proof}
We show by induction on $\ell$. 
For $\ell=k$, it is easy to see that both
hand sides of \eqref{qidentityforgeometricderivation}
are $(m-n)_t (m-n-1)_t \cdots (m-n-k+1)_t$.
Next we assume that 
\eqref{qidentityforgeometricderivation} hold for 
$\ell \ge k$ and show 
\eqref{qidentityforgeometricderivation} for $\ell+1$. 
We divide the summation in left hand side of 
\eqref{qidentityforgeometricderivation} as
\[
\sum_{\ell +1 \ge h_1>\cdots>h_k \ge 1} \ =\sum_{\ell +1=h_1>h_2>\cdots>h_k \ge 1} \ +\sum_{\ell \ge h_1>h_2>\cdots>h_k\ge 1} \ ,
\]
and apply the induction hypothesis to each sum.
The first sum becomes
\begin{align}
&\sum_{\ell +1=h_1>h_2>\cdots>h_k \ge 1} \prod_{i=1}^k (\ell -2h_i-i+3+m-n)_t \nn \\
=&(-\ell +m-n)_t
\sum_{\ell \ge h_2>\cdots>h_k \ge 1} \prod_{i=2}^k (\ell -2h_i-i+3+m-n)_t \nn \\
=&(-\ell +m-n)_t
\sum_{\ell \ge h_1>\cdots>h_{k-1} \ge 1} \prod_{i=1}^{k-1} (\ell -2h_i-i+2+m-n)_t \nn \\
=&(-\ell +m-n)_t \frac{(m-n)_t! (\ell)_t!}{(k-1)_t!(m-n-k+1)_t!(\ell -k+1)_t!}. \label{inductionsumone}
\end{align}

The second sum can be rewritten as
\begin{align}
\sum_{\ell \ge h_1>h_2>\cdots>h_k\ge 1} \prod_{i=1}^k (\ell -2h_i-i+3+m-n)_t
=\frac{(m-n+1)_t! (\ell)_t!}{(k)_t!(m-n+1-k)_t!(\ell -k)_t!}. \label{inductionsumtwo}
\end{align}
Combining \eqref{inductionsumone} and
\eqref{inductionsumtwo} and using
\begin{align}
(x-u)_t (y-v)_t
-(x-v)_t (y-u)_t=(x-y)_t (u-v)_t,
\end{align}
for $x=m-n, y=k-1, u=\ell , v=-1$, 
we get
\begin{align}
&\sum_{\ell +1 \ge h_1 > \cdots > h_k \ge 1} \prod_{i=1}^k (\ell -2h_i-i+3+m-n)_t \nn \\
=&(-\ell +m-n)_t \frac{(m-n)_t! (\ell)_t!}{(k-1)_t!(m-n-k+1)_t!(\ell -k+1)_t!}
+\frac{(m-n+1)_t! (\ell)_t!}{(k)_t!(m-n+1-k)_t!(\ell -k)_t!} \nn \\
=&\frac{(m-n)_t! (\ell)_t!}{(k)_t!(m-n-k+1)_t!(\ell -k+1)_t!}
(
(-\ell +m-n)_t (k)_t+
(m-n+1)_t (\ell -k+1)_t
) \nn \\
=&\frac{(m-n)_t! (\ell +1)_t!}{(k)_t!(m-n-k)_t!(\ell -k+1)_t!},
\end{align}
which completes the proof.
\end{proof}
{
\color{black}The elliptic uplift of the 
$\chi_{t}$-genus should be similarly defined in the 
context of the elliptic cohomology.
In the near future, we will try to formulate it.}

\section{Rational version}

In this section, we discuss the rational version of the source identity.
In addition to presenting the results which are rational version
of the ones in previous section for the trigonometric case, we also prove a rational version
of symmetrization formulas by Lascoux.

\subsection{Reduction from trigonometric to rational source identity}
Here we briefly discuss the degeneration from trigonometric to rational source identity.

We start from the $m=n$ case of the trigonometric source identity \eqref{trigonometricKajihara}.
Replacing $zq^{-1/2}$ by $z$, $m=n$ case of
\eqref{trigonometricKajihara} can be rewritten as
\begin{align}
&\sum_{K \subset [1,\dots,n]} (-z)^{|K|}  \prod_{\substack{i \in K \\ j \not\in K}} \frac{q^{-1/2} v_i^{1/2} v_j^{-1/2}-q^{1/2} v_j^{1/2} v_i^{-1/2} }
{v_i^{1/2} v_j^{-1/2}-v_j^{1/2} v_i^{-1/2}}
\prod_{\substack{ i \in K  \\ 1 \le k \le n }} \frac{v_i^{1/2} u_k^{-1/2}-u_k^{1/2} v_i^{-1/2}}{q^{-1/2}v_i^{1/2} u_k^{-1/2}-q^{1/2}u_k^{1/2} v_i^{-1/2}}
\nonumber \\
=& \sum_{K \subset [1,\dots,n]} (-z)^{|K|}  \prod_{\substack{i \in K \\ j \not\in K}} \frac{q^{1/2} u_i^{1/2} u_j^{-1/2}-q^{-1/2} u_j^{1/2} u_i^{-1/2}}{u_i^{1/2} u_j^{-1/2}-u_j^{1/2} u_i^{-1/2}}
\prod_{\substack{ i \in K  \\ 1 \le k \le m }} \frac{v_k^{1/2} u_i^{-1/2}-u_i^{1/2} v_k^{-1/2} }{q^{-1/2} v_k^{1/2} u_i^{-1/2}-q^{1/2} u_i^{1/2} v_k^{-1/2}}.
\label{beforetakingtrigtorat}
\end{align}

Introduce $x_i$, $y_j$, $c$
through $u_i=e^{x_i}$, $v_j=e^{y_j}$, $q^{-1/2}=e^{-c}$.
Then \eqref{beforetakingtrigtorat} is rewritten using trigonometric functions as
\begin{align}
&\sum_{K \subset [1,\dots,n]} (-z)^{|K|}  \prod_{\substack{i \in K \\ j \not\in K}} \frac{\mathrm{sinh}(y_i-y_j-c) }{\mathrm{sinh}(y_i-y_j)}
\prod_{\substack{ i \in K  \\ 1 \le k \le n }} \frac{\mathrm{sinh}(y_i-x_k) }{\mathrm{sinh}(y_i-x_k-c)}
\nonumber \\
=& \sum_{K \subset [1,\dots,n]} (-z)^{|K|}  \prod_{\substack{i \in K \\ j \not\in K}} \frac{\mathrm{sinh}(x_i-x_j+c)}{\mathrm{sinh}(x_i-x_j)}
\prod_{\substack{ i \in K  \\ 1 \le k \le m }} \frac{\mathrm{sinh}(x_i-y_k) }{\mathrm{sinh}(x_i-y_k+c)}.
\end{align}
Taking the rational limit $\sinh w \to w$ $(w \to 0)$ and redefining $x_i$ and $y_j$ as $u_i$ and $v_j$ gives
the $m=n$ case of the rational source identity \eqref{trigonometricKajihara}.
Next, taking the limiting procedure in the same way as described in the proof in
Theorem \ref{limitofellipticsource}, we get \eqref{trigonometricKajihara}
for $n>m$.

\subsection{Determinant representations of rational source functions}

The systematic derivation of determinant forms
of rational functions appearing in the source identity
\eqref{rationalKajihara} can be applied to the rational case as well.
We introduce notations for both hand sides of \eqref{rationalKajihara}
\begin{align}
F_{n,m}^{(z)}(\overline{u}|\overline{v})&:=\sum_{K \subset [1,\dots,m]} (-z)^{|K|} \prod_{\substack{i \in K \\ j \not\in K}} \frac{v_i-v_j-c}{v_i-v_j}
\prod_{\substack{ i \in K  \\ 1 \le k \le n }} \frac{v_i-u_k}{v_i-u_k-c}, \label{rationalsourcefunctionone}  \\
G_{n,m}^{(z)}(\overline{u}|\overline{v})&:=(1-z)^{m-n} \sum_{K \subset [1,\dots,n]} (-z)^{|K|} \prod_{\substack{i \in K \\ j \not\in K}} \frac{u_i-u_j+c}{u_i-u_j}
\prod_{\substack{ i \in K  \\ 1 \le k \le m }} \frac{u_i-v_k}{u_i-v_k+c}. \label{rationalsourcefunctiontwo}
\end{align}
We call these functions as rational source functions.
{\color{black}
We note a special case of the form \eqref{rationalsourcefunctionone} for
$F_{n,m}^{(1)}(\overline{u}|\overline{v})$
appeared in the context of partition functions first in
Gromov-Sever-Vieira \cite[(1.3)]{GSV}. See also \cite[(4.17)]{FodaWheeler}.}
For the case of generic $z$, $F_{n,m}^{(z)}(\overline{u}|\overline{v})$
and $G_{n,m}^{(z)}(\overline{u}|\overline{v})$ appear as representations
of the most generalized domain wall boundary partition functions in Belliard-Pimenta-Slavnov
{\color{black}
 \cite[(92), (93)]{BPS}.}

To derive determinant forms, we rewrite
rational source functions in the following forms
using additive difference operators
\begin{align}
F_{n,m}^{(z)}(\overline{u}|\overline{v})
&=\frac{\prod_{i=1}^m \prod_{k=1}^n (v_i-u_k)}{\prod_{1 \le i < j \le m} (v_j-v_i)}
\prod_{j=1}^m (1-z T_{v_j}^{-1}) \frac{\prod_{1 \le i < j \le m} (v_j-v_i)}{\prod_{i=1}^m \prod_{k=1}^n (v_i-u_k)},
\label{rationalsourcedifferenceone} \\
G_{n,m}^{(z)}(\overline{u}|\overline{v})
&=(1-z)^{m-n} \frac{  \prod_{i=1}^n \prod_{k=1}^m (u_i-v_k)   }{\prod_{1 \le i < j \le n} (u_j-u_i)}
\prod_{j=1}^n (1-z T_{u_j}) \frac{\prod_{1 \le i < j \le n} (u_j-u_i)}{  \prod_{i=1}^n \prod_{k=1}^m (u_i-v_k) }.
\label{rationalsourcedifferencetwo}
\end{align}

In the same way with deriving
Proposition \ref{proptrigMPT} starting from \eqref{rationalsourcedifferenceone} and
\eqref{rationalsourcedifferencetwo} instead, we get the following Minin-Pronko-Tarasov \cite{MPT,PT} type determinant represenations.
\begin{proposition}
We have the following determinant representations:
\begin{align}
F_{n,m}^{(z)}(\overline{u}|\overline{v})=&
\frac{ 1-r \prod_{\ell=1}^m v_\ell }{\det_{1 \le i,j \le m} \Big( \sum_{k=1}^m p_{ik} \psi_k^{A_{m-1}}(v_j;0,r) \Big)}
\nn \\
&\times \det_{1 \le i,j \le m} \Bigg(
\Bigg(
\frac{1}{1-r \prod_{\ell=1}^m v_\ell}
\Bigg)^{\delta_{i1}}
 \sum_{k=1}^m p_{ik} \psi_k^{A_{m-1}}(v_j;0,r) \nn \\
&-z
\Bigg(
\frac{1}{1-r (v_j-c) \prod_{\ell=1, \ell \neq j}^m v_\ell}
\Bigg)^{\delta_{i1}}
 \sum_{k=1}^m p_{ik} \psi_k^{A_{m-1}}(v_j-c;0,r)  
\prod_{\ell=1}^n \frac{v_j-u_\ell }{v_j-u_\ell-c}
\Bigg), \label{MPTone}
\end{align}
\begin{align}
G_{n,m}^{(z)}(\overline{u}|\overline{v}) 
=&(1-z)^{m-n} \frac{  1-r \prod_{\ell=1}^n u_\ell }{ \det_{1 \le i,j \le n} \Big( \sum_{k=1}^n q_{ik}  \psi_k^{A_{n-1}}(u_j;0,r) \Big) }
\nn \\
&\times \det_{1 \le i,j \le n} \Bigg( 
\Bigg(\frac{1}{1-r \prod_{\ell=1}^n u_\ell }
\Bigg)^{\delta_{i1}}
 \sum_{k=1}^n q_{ik} \psi_k^{A_{n-1}}(u_j;0,r) \nn \\
&-z 
\Bigg(\frac{1}{1-r (u_j+c) \prod_{\ell=1, \ell \neq j}^n u_\ell }
\Bigg)^{\delta_{i1}}
\sum_{k=1}^n q_{ik} \psi_k^{A_{n-1}}(u_j+c;0,r)
 \prod_{\ell=1}^m \frac{u_j-v_\ell }{u_j-v_\ell+c}
\Bigg). \label{MPTtwo}
\end{align}
Here, $r$, $p_{ij}$ $(i,j=1,\dots,m)$, $q_{ij}$ $(i,j=1,\dots,n)$ are additional parameters
such that $\det_{1 \le i,j \le m}(p_{ij}) \not\equiv 0$, $\det_{1 \le i,j \le n}(q_{ij}) \not\equiv 0$.
\end{proposition}

The case $r=0$, $p_{ij}=q_{ij}=\delta_{ij}$ of \eqref{MPTone}, \eqref{MPTtwo}
are the scalar product type representations first considered by Kostov \cite{Kostovone,Kostovtwo}.

\begin{corollary}
We have the following determinant representations:
\begin{align}
F_{n,m}^{(z)}(\overline{u}|\overline{v})
&=\prod_{1 \le i < j \le m} \frac{1}{v_j-v_i}
\det_{1 \le i,j\le m} \Bigg( v_j^{i-1} 
-z (v_j-c)^{i-1} \prod_{\ell=1}^n  \frac{v_j-u_\ell}{v_j-u_\ell-c} 
\Bigg), \label{oneparameterKostov}
\\
G_{n,m}^{(z)}(\overline{u}|\overline{v})
&=(1-z)^{m-n} \prod_{1 \le i < j \le n} \frac{1}{ u_j-u_i}
\det_{1 \le i , j \le n} \Bigg(
u_j^{i-1} 
-z (u_j+c)^{i-1} \prod_{\ell=1}^m  \frac{u_j-v_\ell}{u_j-v_\ell+c} 
\Bigg). \label{oneparameterKostovvertwo}
\end{align}
\end{corollary}

The following domain wall boundary type determinant representations first considered by Foda-Wheeler
\cite{FodaWheeler}
can also be derived in  the same way with deriving Proposition \ref{ProptrigFW}.
\begin{proposition}
For $n \ge m$,
we have the following determinant representations:
\begin{align}
F_{n,m}^{(z)}(\overline{u}|\overline{v})=&
\frac{\prod_{i=1}^m \prod_{k=1}^n (v_i-u_k)}{\prod_{1 \le i < j \le m} (v_j-v_i) \prod_{1 \le i < j \le n} (u_i-u_j)} 
\det_{1 \le i,j \le n} (Y),
\label{oneparameterdefFodaWheeler}
\end{align}
where $Y$ is an $n \times n$ matrix whose $(i,j)$-entry is given by
\begin{align}
Y_{ij}=
\left\{
\begin{array}{ll}
\displaystyle \frac{1}{v_i-u_j}-z \frac{1}{v_i-u_j-c}, & 1 \le i \le m, \ \ \ 1 \le j \le n \\
u_j^{n-i}, & m+1 \le i \le n, \ \ \ 1 \le j \le n
\end{array}
\right.
.
\end{align}

\begin{align}
G_{n,m}^{(z)}(\overline{u}|\overline{v})=&(1-z)^{m-n}
\frac{\prod_{i=1}^m \prod_{k=1}^n (v_i-u_k)}{\prod_{1 \le i < j \le m} (v_j-v_i) \prod_{1 \le i < j \le n} (u_i-u_j)} \det_{1 \le i,j \le n} (Z),
\label{oneparameterdefFodaWheelertwo}
\end{align}
where $Z$ is an $n \times n$ matrix whose $(i,j)$-entry is given by
\begin{align}
Z_{ij}=
\left\{
\begin{array}{ll}
\displaystyle \frac{1}{v_i-u_j}-z \frac{1}{v_i-u_j-c}, & 1 \le i \le m, \ \ \ 1 \le j \le n \\
u_j^{n-i}-z (u_j+c)^{n-i}, & m+1 \le i \le n, \ \ \ 1 \le j \le n
\end{array}
\right.
.
\end{align}
{\color{black}
For $n < m$,
we have the following determinant representations:
\begin{align}
F_{n,m}^{(z)}(\overline{u}|\overline{v})=&
\frac{\prod_{i=1}^m \prod_{k=1}^n (u_k-v_i)}{\prod_{1 \le i < j \le m} (v_i-v_j) \prod_{1 \le i < j \le n} (u_j-u_i)} \det_{1 \le i,j \le m} (U),
\end{align}
where $U$ is an $m \times m$ matrix whose $(i,j)$-entry is given by
\begin{align}
U_{ij}=
\left\{
\begin{array}{ll}
\displaystyle \frac{1}{u_i-v_j}-z \frac{1}{u_i-v_j+c}, & 1 \le i \le n, \ \ \ 1 \le j \le m \\
v_j^{m-i}-z(v_j-c)^{m-i}, & n+1 \le i \le m, \ \ \ 1 \le j \le m
\end{array}
\right.
.
\end{align}
\begin{align}
G_{n,m}^{(z)}(\overline{u}|\overline{v})=&(1-z)^{m-n}
\frac{\prod_{i=1}^m \prod_{k=1}^n (u_k-v_i)}{\prod_{1 \le i < j \le m} (v_i-v_j) \prod_{1 \le i < j \le n} (u_j-u_i)} \det_{1 \le i,j \le m}  (V),
\end{align}
where $V$ is an $m \times m$ matrix whose $(i,j)$-entry is given by
\begin{align}
V_{ij}=
\left\{
\begin{array}{ll}
\displaystyle \frac{1}{u_i-v_j}-z \frac{1}{u_i- v_j+c}, & 1 \le i \le n, \ \ \ 1 \le j \le m \\
v_j^{m-i}, & n+1 \le i \le m, \ \ \ 1 \le j \le m
\end{array}
\right.
.
\end{align}
}

\end{proposition}

Finally, the following is a one-parameter extension of the determinant representations by
Belliard-Slavnov \cite{BS}. These representations can be derived in the same way with
Proposition \ref{ProptrigBS}.
\begin{proposition}
We have the following determinant representations:
\begin{align}
&F_{n,m}^{(z)}(\overline{u}|\overline{v}) 
=\frac{1-\Delta}{\prod_{1 \le i < j \le m} (v_j-v_i)(\eta_i-\eta_j)} \nn \\
&\times \det_{1 \le i,j \le m} \Bigg(
\Bigg(\frac{1}{\prod_{\ell=1}^m v_\ell-\Delta \prod_{\ell=1}^m \eta_\ell}
\Bigg)^{\delta_{i1}} 
\frac{ (v_i-\Delta \eta_j)
\prod_{k=1, k \neq j}^m (v_i-\eta_k) 
}{1-\Delta} \nn \\
&-z
\Bigg(\frac{1}{(v_i-c) \prod_{\ell=1, \ell \neq i}^m v_\ell-\Delta \prod_{\ell=1}^m \eta_\ell}
\Bigg)^{\delta_{i1}}
\frac{ (v_i-\Delta \eta_j-c)
\prod_{k=1, k \neq j}^m (v_i-\eta_k-c)
}{1-\Delta}
\prod_{k=1}^n \frac{v_i-u_k}{v_i-u_k-c}
\Bigg),
\label{BSrepone} 
\end{align}
\begin{align}
&G_{n,m}^{(z)}(\overline{u}|\overline{v})=
(1-z)^{m-n}
\frac{1-\Delta}{\prod_{1 \le i < j \le n} (u_j-u_i)(\eta_i-\eta_j)} 
\nn \\
&\times
\det_{1 \le i,j \le n} \Bigg( 
\Bigg(\frac{1}{\prod_{\ell=1}^n u_\ell-\Delta \prod_{\ell=1}^n \eta_\ell}
\Bigg)^{\delta_{i1}} 
  \frac{(u_i- \Delta \eta_j)  \prod_{k=1,k \neq j}^n (u_i-\eta_k) }{1-\Delta}
\nn \\
&-z
\Bigg(\frac{1}{(u_i+c) \prod_{\ell=1, \ell \neq i}^n u_\ell-\Delta \prod_{\ell=1}^n \eta_\ell}
\Bigg)^{\delta_{i1}}
\frac{(u_i-\Delta \eta_j+c)
\prod_{k=1,k \neq j}^n (u_i-\eta_k+c) }
{ 1-\Delta }
\prod_{k=1}^m \frac{u_i-v_k}{u_i-v_k+c}
\Bigg).
\label{BSreptwo}
\end{align}
\end{proposition}

Taking the limit $\Delta \to \infty$ gives the determinants in \cite{BS}.

\begin{corollary}
We have the following determinant representations:
\begin{align}
F_{n,m}^{(z)}(\overline{u}|\overline{v}) 
=&\frac{1}{\prod_{1 \le i < j \le m} (v_j-v_i)(\eta_i-\eta_j)} \nn \\
&\times \det_{1 \le i,j \le m} \Bigg(
\prod_{k=1, k \neq j}^m (v_i-\eta_k) 
-z
\prod_{k=1, k \neq j}^m (v_i-\eta_k-c)
\prod_{k=1}^n \frac{v_i-u_k}{v_i-u_k-c}
\Bigg),  \\
G_{n,m}^{(z)}(\overline{u}|\overline{v})=&
(1-z)^{m-n}
\frac{1}{\prod_{1 \le i < j \le n} (u_j-u_i)(\eta_i-\eta_j)} 
\nn \\
&\times
\det_{1 \le i,j \le n} \Bigg( 
  \prod_{k=1,k \neq j}^n (u_i-\eta_k) 
-z
\prod_{k=1,k \neq j}^n (u_i-\eta_k+c) 
\prod_{k=1}^m \frac{u_i-v_k}{u_i-v_k+c}
\Bigg).
\end{align}
\end{corollary}

\subsection{A complex analytic proof of rational source identity}
We give a complex analytic proof of the rational source identity
which goes parallel with the trigonometric case.
A different proof showing as an identity between determinants can be found in {\color{black} \cite[Appendix A]{GZZ}.}
We prove the following polynomial version

\begin{align}
&\sum_{K \subset [1,\dots,m]} (-z)^{|K|} \prod_{\substack{i \in K \\ j \not\in K}} \frac{v_i-v_j-c}{v_i-v_j}
\prod_{\substack{ j \not\in K  \\ 1 \le k \le n }} (v_j-u_k-c)
\prod_{\substack{ i \in K  \\ 1 \le k \le n }} (v_i-u_k)
\nonumber \\
=&(1-z)^{m-n} \sum_{K \subset [1,\dots,n]} (-z)^{|K|} \prod_{\substack{i \in K \\ j \not\in K}} \frac{u_i-u_j+c}{u_i-u_j}
\prod_{\substack{ j \not\in K  \\ 1 \le k \le m }} (v_k-u_j-c)
\prod_{\substack{ i \in K  \\ 1 \le k \le m }} (v_k-u_i).
\label{polyrationalKajiharaone}
\end{align}
{\color{black}
Note \eqref{polyrationalKajiharaone}
can be obtained from
\eqref{rationalKajihara} by multiplying both hand sides by $\prod_{i=1}^m \prod_{k=1}^n (v_i-u_k-c)$
and the identities are equivalent.}
Let us denote the left hand side and the right hand side of  \eqref{polyrationalKajiharaone}
as $P_{n,m}^{(z)}(\overline{u}|\overline{v})$ and $Q_{n,m}^{(z)}(\overline{u}|\overline{v})$
respectively.
Relations with $F_{n,m}^{(z)}(\overline{u}|\overline{v})$ and $G_{n,m}^{(z)}(\overline{u}|\overline{v})$
are
\begin{align}
P_{n,m}^{(z)}(\overline{u}|\overline{v})&=\prod_{i=1}^m \prod_{k=1}^n (v_i-u_k-c) F_{n,m}^{(z)}(\overline{u}|\overline{v}), \\
Q_{n,m}^{(z)}(\overline{u}|\overline{v})&=\prod_{i=1}^m \prod_{k=1}^n (v_i-u_k-c) G_{n,m}^{(z)}(\overline{u}|\overline{v}).
\end{align}

We assume $m \le n$ for the proof.
First,
we find the following properties hold.
\begin{lemma} \label{degreesymmetrylemma}
$P_{n,m}^{(z)}(\overline{u}|\overline{v})$ and $Q_{n,m}^{(z)}(\overline{u}|\overline{v})$
are symmetric polynomials in $v_1, \dots, v_m$.
The degree of $v_i \ (i=1,2,\dots,m)$ is at most  $n$ for both polynomials.
\end{lemma}
\begin{proof}
From the determinant forms of $F_{n,m}^{(z)}(\overline{u}|\overline{v})$ \eqref{oneparameterKostov}
and $G_{n,m}^{(z)}(\overline{u}|\overline{v})$
\eqref{oneparameterKostovvertwo}, we have
the following determinant forms of $P_{n,m}^{(z)}(\overline{u}|\overline{v})$
and  $Q_{n,m}^{(z)}(\overline{u}|\overline{v})$
\begin{align}
&P_{n,m}^{(z)}(\overline{u}|\overline{v}) \nn \\
=&\prod_{1 \le i < j \le m} \frac{1}{v_j-v_i}
\det_{1 \le i,j \le m} \Bigg(
(v_i+c)^{j-1} \prod_{k=1}^n (v_i-u_k-c)-z v_i^{j-1} \prod_{k=1}^n (v_i-u_k)
\Bigg), \label{toseesymmetryuvariables} \\
&Q_{n,m}^{(z)}(\overline{u}|\overline{v}) \nn \\
=&(1-z)^{m-n} \prod_{1 \le i < j \le n} \frac{1}{ u_j-u_i}
\det_{1 \le i,j \le n} \Bigg(
u_i^{j-1} \prod_{k=1}^m (v_k-u_i-c)
-z (u_i+c)^{j-1} \prod_{k=1}^m (v_k-u_i) 
\Bigg).
\end{align}
Lemma \ref{degreesymmetrylemma} can be checked
in the same way with the trigonometric case using these determinant forms.
\end{proof}
Due to Lemma \ref{degreesymmetrylemma},
one notes that it is enough to check $P_{n,m}^{(z)}(\overline{u}|\overline{v})=Q_{n,m}^{(z)}(\overline{u}|\overline{v})$
for $(n+1)^m$ distinct points in $(v_1,\dots,v_m)$,
and we check the equality for the following $2n \times (2n-1) \times \cdots \times (2n-m+1) (\ge (n+1)^m)$
distinct points
$v_1=u_{p_1}+\epsilon_1, \ v_2=u_{p_2}+\epsilon_2, \ \dots, v_m=u_{p_m}+\epsilon_m$,
$p_1,\dots,p_m \in \{1,\dots,n \}$, $\epsilon_1,\dots,\epsilon_m \in \{ 0,c \}$
and $v_i \neq v_j$ for $i \neq j$.
It is easy to check the following cases which correspond to $p_i=p_j$ for some $i \neq j$.

\begin{lemma} \label{vanishinglemma}
If $v_i=u_k$, $v_j=u_k+c$ for some $i,j,k \ (i \neq j)$,
we have 
\begin{align}
P_{n,m}^{(z)}(\overline{u}|\overline{v})=Q_{n,m}^{(z)}(\overline{u}|\overline{v})=0.
\end{align}
\end{lemma}
\begin{proof}
We check the case $k=n$ as the other cases can be checked in the same way.
It is easy to see from the expressions in \eqref{polyrationalKajiharaone}
that substituting $v_m=u_n+c$ in $P_{n,m}^{(z)}(\overline{u}|\overline{v})$
and $Q_{n,m}^{(z)}(\overline{u}|\overline{v})$, we have
\begin{align}
P_{n,m}^{(z)}(\overline{u}|\overline{v})|_{v_m=u_n+c}
&=-cz \prod_{j=1}^{n-1} (u_n-u_j+c) \prod_{j=1}^{m-1} (v_j-u_n)
P_{n-1,m-1}^{(z)}(u_1,\dots,u_{n-1}|v_1,\dots,v_{m-1}), \label{subone} \\
Q_{n,m}^{(z)}(\overline{u}|\overline{v})|_{v_m=u_n+c}
&=-cz \prod_{j=1}^{n-1} (u_n-u_j+c) \prod_{j=1}^{m-1} (v_j-u_n)
Q_{n-1,m-1}^{(z)}(u_1,\dots,u_{n-1}|v_1,\dots,v_{m-1}). \label{subtwo}
\end{align}
Since there is an overall factor $\prod_{j=1}^{m-1} (v_j-u_n)$ in \eqref{subone}
and \eqref{subtwo}, 
further substituting $v_j=u_n$ for some $j \ (1 \le j \le m-1)$,
both $P_{n,m}^{(z)}(\overline{u}|\overline{v})|_{v_m=u_n+c}$
and $Q_{n,m}^{(z)}(\overline{u}|\overline{v})|_{v_m=u_n+c}$ vanish.
Since both $P_{n,m}^{(z)}(\overline{u}|\overline{v})$
and $Q_{n,m}^{(z)}(\overline{u}|\overline{v})$
are symmetric polynomials in $v$ variables, the other cases follow by symmetry.
\end{proof}

We next assume $p_i \neq p_j$ for $i \neq j$.
For $I=\{ 1 \le i_1 < i_2 < \cdots < i_{|I|} \le n \}$ and $J=\{1 \le j_1 < j_2 < \cdots < j_{|J|} \le n \}$
such that $I \cup J$ is a set of $m$ distinct integers in $[ 1,\dots,n ]$,
let us denote by $\overline{v}=\{ \overline{u}_I, \overline{u}_J+c \}$ the 
following type of substitution
\begin{align}
v_{q_1}=u_{i_1}, \dots, v_{q_{|I|}}=u_{i_{|I|}}, v_{r_1}=u_{j_1}+c, \dots, v_{r_{|J|}}=u_{j_{|J|}}+c,
\end{align}
for some $\{ q_1, \dots, q_{|I|}, r_1,\dots, r_{|J|} \}=[1,\dots,m ]$.
Note $|I|+|J|=m$.

We have the following explicit evaluations.
\begin{proposition} \label{PQspecializations}
We have
\begin{align}
P_{n,m}^{(z)}(\overline{u}|\{ \overline{u}_I, \overline{u}_J+c \})
=(-z)^{|J|} \prod_{ \substack{ i \in I \\ j \in J  }  } (u_j-u_i)
\prod_{ \substack{ i \in I \\ 1 \le j \le n  }  } (u_i-u_j-c)
 \prod_{ \substack{ j \in J \\ k \in [1,\dots,n] \backslash I  }  } (u_j-u_k+c), \\
Q_{n,m}^{(z)}(\overline{u}|\{ \overline{u}_I, \overline{u}_J+c \})
=(-z)^{|J|} \prod_{ \substack{ i \in I \\ j \in J  }  } (u_j-u_i)
\prod_{ \substack{ i \in I \\ 1 \le j \le n  }  } (u_i-u_j-c)
 \prod_{ \substack{ j \in J \\ k \in [1,\dots,n] \backslash I  }  } (u_j-u_k+c).
\end{align}
\end{proposition}
\begin{proof}
It is easy to check the case for
$P_{n,m}^{(z)}(\overline{u}|\{ \overline{u}_I, \overline{u}_J+c \})$.
After substituting
$
v_{q_1}=u_{i_1}, \dots, v_{q_{|I|}}=u_{i_{|I|}}, v_{r_1}=u_{j_1}+c, \dots, v_{r_{|J|}}=u_{j_{|J|}}+c$,
one notes from the factors
$\displaystyle \prod_{\substack{ j \not\in K  \\ 1 \le k \le n }} (v_j-u_k-c)$
and
$\displaystyle \prod_{\substack{ i \in K  \\ 1 \le k \le n }} (v_i-u_k)$
that only the summand corresponding to $K=\{ r_1,\dots,r_{|J|} \}$  survives,
and we get
\begin{align}
&P_{n,m}^{(z)}(\overline{u}|\{ \overline{u}_I, \overline{u}_J+c \}) \nn \\
=&(-z)^{|J|} \prod_{\substack{i \in J \\ j \in  I}} \frac{u_i-u_j}{u_i-u_j+c}
\prod_{\substack{ j \in I  \\ 1 \le k \le n }} (u_j-u_k-c)
\prod_{\substack{ j \in J  \\ 1 \le k \le n }} (u_j-u_k+c) \nn \\
=&(-z)^{|J|} \prod_{ \substack{ i \in I \\ j \in J  }  } (u_j-u_i)
\prod_{ \substack{ i \in I \\ 1 \le j \le n  }  } (u_i-u_j-c)
 \prod_{ \substack{ j \in J \\ k \in [1,\dots,n] \backslash I  }  } (u_j-u_k+c).
\end{align}

Next, we check the case for 
$Q_{n,m}^{(z)}(\overline{u}|\{ \overline{u}_I, \overline{u}_J+c \})$.
Substituting
$
v_{q_1}=u_{i_1}, \dots, v_{q_{|I|}}=u_{i_{|I|}}, v_{r_1}=u_{j_1}+c, \dots, v_{r_{|J|}}=u_{j_{|J|}}+c$,
we note from the factor
$\displaystyle \prod_{\substack{ i \in K  \\ 1 \le k \le m }} (v_k-u_i)$
that summands labeled by $K$ satisfying $K \cap I \neq \phi$ vanish.
We also note from the factor
$\displaystyle \prod_{\substack{ j \not\in K  \\ 1 \le k \le m }} (v_k-u_j-c)$
that summands vanish unless $J \subset K$.
Therefore we can restrict the sum to
$\displaystyle \sum_{J \subset K \subset [1,\dots,n] \backslash I}$
and after simplifications, we can write 
$Q_{n,m}^{(z)}(\overline{u}|\{ \overline{u}_I, \overline{u}_J+c \})$ as
\begin{align}
&Q_{n,m}^{(z)}(\overline{u}|\{ \overline{u}_I, \overline{u}_J+c \}) \nn \\
=&
\prod_{ \substack{ i \in I \\ j \in J  }  } (u_j-u_i)
\prod_{ \substack{ i \in I \\ 1 \le j \le n  }  } (u_i-u_j-c)
 \prod_{ \substack{ j \in J \\ k \in [1,\dots,n] \backslash I  }  } (u_j-u_k+c)
(-z)^{|J|}
(1-z)^{m-n} \nn \\
&\times \sum_{J \subset K \subset [1,\dots,n] \backslash I} 
(-z)^{|K|-|J|}
\prod_{\substack{i \in K \backslash J \\ j \in [1,\dots,n] \backslash \{I,K \} }}
\frac{u_i-u_j+c}{u_i-u_j}. \label{beforesimplyfyingq}
\end{align}
Next, we note the factorization of the sum in \eqref{beforesimplyfyingq}
\begin{align}
&\sum_{J \subset K \subset [1,\dots,n] \backslash I} 
(-z)^{|K|-|J|}
\prod_{\substack{i \in K \backslash J \\ j \in [1,\dots,n] \backslash \{I,K \} }}
\frac{u_i-u_j+c}{u_i-u_j} \nn \\
=&\sum_{\ell=|J|}^{n-|I|}
\sum_{\substack{ K: |K|=\ell   \\ J \subset K \subset [1,\dots,n] \backslash I }}
(-z)^{\ell-|J|}
\prod_{\substack{i \in K \backslash J \\ j \in [1,\dots,n] \backslash \{I,K \} }}
\frac{u_i-u_j+c}{u_i-u_j} \nn \\
=&\sum_{\ell=|J|}^{n-|I|}
(-z)^{\ell-|J|} \dbinom{ n-|I|-|J|    }{\ell- |J | }=(1-z)^{n-|I|-|J|}=(1-z)^{n-m}. \label{beforesimplyfyingqtwo}
\end{align}
Here we used the well-known identity
\begin{align}
\sum_{K \subset [1,\dots,n]: |K|=\ell }
\prod_{\substack{i \in K \\ j \in [1,\dots,n] \backslash K  }}
\frac{u_i-u_j+c}{u_i-u_j}=\dbinom{ n  }{\ell}.
\end{align}
Combining
\eqref{beforesimplyfyingq} and \eqref{beforesimplyfyingqtwo} gives
\begin{align}
Q_{n,m}^{(z)}(\overline{u}|\{ \overline{u}_I, \overline{u}_J+c \})
=(-z)^{|J|} \prod_{ \substack{ i \in I \\ j \in J  }  } (u_j-u_i)
\prod_{ \substack{ i \in I \\ 1 \le j \le n  }  } (u_i-u_j-c)
 \prod_{ \substack{ j \in J \\ k \in [1,\dots,n] \backslash I  }  } (u_j-u_k+c).
\end{align}

\end{proof}
From
Lemma \ref{vanishinglemma} and Proposition \ref{PQspecializations},
we checked the equality for enough specializations of $(v_1,\dots,v_m)$ and
hence, \eqref{polyrationalKajiharaone} follows.

{\color{black}
The case $m>n$ can be proved in the same way with the case $m \leq n$.
The following lemmas and proposition correspond to
Lemma \ref{degreesymmetrylemma},
Lemma \ref{vanishinglemma} and Proposition \ref{PQspecializations}.
\begin{lemma} 
$P_{n,m}^{(z)}(\overline{u}|\overline{v})$ and $Q_{n,m}^{(z)}(\overline{u}|\overline{v})$
are symmetric polynomials in $u_1, \dots, u_n$.
The degree of $u_i \ (i=1,2,\dots,n)$ is at most  $m$
for both polynomials.
\end{lemma}
\begin{lemma} 
If $u_i=v_k$, $u_j=v_k-c$ for some $i,j,k \ (i \neq j)$,
we have 
\begin{align}
P_{n,m}^{(z)}(\overline{u}|\overline{v})=Q_{n,m}^{(z)}(\overline{u}|\overline{v})=0.
\end{align}
\end{lemma}
For $I=\{ 1 \le i_1 < i_2 < \cdots < i_{|I|} \le m \}$ and $J=\{1 \le j_1 < j_2 < \cdots < j_{|J|} \le m \}$
such that $I \cup J$ is a set of $n$ distinct integers in $[ 1,\dots,m ]$,
let us denote by $\overline{u}=\{ \overline{v}_I, \overline{v}_J-c \}$ the 
following type of substitution
\begin{align}
u_{q_1}=v_{i_1}, \dots, u_{q_{|I|}}=v_{i_{|I|}}, u_{r_1}=v_{j_1}-c, \dots, u_{r_{|J|}}=v_{j_{|J|}}-c,
\end{align}
for some $\{ q_1, \dots, q_{|I|}, r_1,\dots, r_{|J|} \}=[1,\dots,n ]$.
Note $|I|+|J|=n$.
\begin{proposition}
We have
\begin{align}
P_{n,m}^{(z)}(\{ \overline{v}_I, \overline{v}_J-c \}|\overline{v})
=&(-z)^{|J|} (1-z)^{m-n} \nn \\
&\times \prod_{ \substack{ i \in I \\ j \in J  }  } (v_i-v_j)
\prod_{ \substack{ i \in I \\ 1 \le j \le m  }  } (v_j- v_i-c)
 \prod_{ \substack{ j \in J \\ k \in [1,\dots,m] \backslash I  }  } (v_k-v_j+c), \\
Q_{n,m}^{(z)}(\{ \overline{v}_I, \overline{v}_J-c \}|\overline{v})
=&(-z)^{|J|} (1-z)^{m-n} \nn \\
&\times \prod_{ \substack{ i \in I \\ j \in J  }  } (v_i-v_j)
\prod_{ \substack{ i \in I \\ 1 \le j \le m  }  } (v_j- v_i-c)
 \prod_{ \substack{ j \in J \\ k \in [1,\dots,m] \backslash I  }  } (v_k-v_j+c).
\end{align}
\end{proposition}
}

\subsection{Rational version of Lascoux's symmetrization formulas}
We prove symmetrization formulas
which are rational version of Lascoux's symmetrization formulas \cite{Lascoux},
Theorem 3 and Theorem 4, which special cases of rational source functions appear.
We introduce notations for divided difference operators and symmetrization.
{\color{black}
For a function $f(u_1,\dots,u_n)$ of $n$-variables $(n \geq 2)$,
we introduce divided difference operators
$\partial_k$ $k=1,\dots,n-1$ by acting on $f(u_1,\dots,u_n)$ as
\begin{align}
&\partial_k f(u_1,\dots,u_n) 
:=\frac{f(u_1,\dots,u_k,u_{k+1},\dots,u_n)-f(u_1,\dots,u_{k+1},u_{k},\dots,u_n)}{u_k-u_{k+1}}.
\end{align}
For $n \geq 1$, we introduce the following symmetrizing operation
\begin{align}
&\mathrm{Sym}_c (f(u_1,\dots,u_n)):=\sum_{w \in S_n} w \cdot \Bigg(  \Delta(1,2,\dots,n) f(u_1,\dots,u_n) \Bigg),
\label{rationalsymmetrization}
\end{align}
where $\displaystyle \Delta(k_1,k_2,\dots,k_n):=\prod_{1 \le i < j \le n} \frac{u_{k_i}-u_{k_j}-c}{u_{k_i}-u_{k_j}}$.
}
We also define $\theta$ to be shifting the index of variables
$\theta u_k:=u_{k+1}$ with $u_{k+n}:=u_k+c$.

The following is a rational version of {\color{black} \cite[Theorem 3]{Lascoux}.}
{\color{black}
\begin{theorem}
For $n \geq 2$,
we have
\begin{align}
&\mathrm{Sym}_c \Bigg( (1-\theta)^{n-1} \prod_{j=2}^n \prod_{k=1}^n (u_j-v_k) f(u_1) 
\Bigg) \nn \\
=&(n-1)!(-c)^{n-1} \frac{\prod_{i=1}^n \prod_{k=1}^n (v_i-u_k)(v_i-u_k-c)}{\prod_{1 \le i < j \le n} (v_j-v_i) \prod_{1 \le i < j \le n} (u_i-u_j)} 
\nn \\
&\times
\det_{1 \le j,k \le n} \Bigg(\frac{1}{(v_j-u_k)(v_j-u_k-c)}  \Bigg)
\partial_{n-1} \cdots \partial_1 f(u_1).
\label{rationalLascouxsymmetrization}
\end{align}
\end{theorem}
Before giving a proof, let us give some remarks.
When $n=2$,
the left hand side of \eqref{rationalLascouxsymmetrization} is explicitly
\begin{align}
&\frac{(u_1-u_2-c)(u_2-v_1)(u_2-v_2)+(u_2-u_1-c)(u_2-v_1+c)(u_2-v_2+c)}{u_1-u_2} f(u_1)
\nn \\
&+\frac{(u_1-u_2-c)(u_1-v_1+c)(u_1-v_2+c)+(u_2-u_1-c)(u_1-v_1)(u_1-v_2)}{u_2-u_1} f(u_2),
\label{ntwoexamplelhs}
\end{align}
and the right hand side is explicitly
\begin{align}
&\frac{-c \prod_{i=1}^2 \prod_{k=1}^2 (v_i-u_k)(v_i-u_k-c)}{(v_2-v_1)(u_1-u_2)} 
\det_{1 \le j,k \le 2} \Bigg(\frac{1}{(v_j-u_k)(v_j-u_k-c)}  \Bigg) \frac{ f(u_1)-f(u_2) }{u_1-u_2}.
\label{ntwoexamplerhs}
\end{align}
One can check that all the factors
\begin{align}
&\frac{(u_1-u_2-c)(u_2-v_1)(u_2-v_2)+(u_2-u_1-c)(u_2-v_1+c)(u_2-v_2+c)}{u_1-u_2},
\nn \\
&\frac{(u_1-u_2-c)(u_1-v_1+c)(u_1-v_2+c)+(u_2-u_1-c)(u_1-v_1)(u_1-v_2)}{u_2-u_1},
\nn \\
&\frac{-c \prod_{i=1}^2 \prod_{k=1}^2 (v_i-u_k)(v_i-u_k-c)}{(v_2-v_1)(u_1-u_2)} 
\det_{1 \le j,k \le 2} \Bigg(\frac{1}{(v_j-u_k)(v_j-u_k-c)}  \Bigg),
\end{align}
can be expanded as
\begin{align}
-c(c^2+cu_1+cu_2-cv_1-cv_2-u_1v_1-u_2v_1-u_1v_2-u_2v_2+2u_1 u_2+2v_1v_2),
\end{align}
and both \eqref{ntwoexamplelhs} and \eqref{ntwoexamplerhs} can be rewritten as
\begin{align}
-c(c^2+cu_1+cu_2-cv_1-cv_2-u_1v_1-u_2v_1-u_1v_2-u_2v_2+2u_1 u_2+2v_1v_2) \frac{ f(u_1)-f(u_2) }{u_1-u_2}.
\end{align}

Although we assume $n \geq 2$ in the statement,
one can also include the case $n=1$ by naturally interpreting
the $n=1$ case of
$(1-\theta)^{n-1}$, $\prod_{j=2}^n \prod_{k=1}^n (u_j-v_k)$ and
$\partial_{n-1} \cdots \partial_1 f(u_1)$ as 1, 1, and $f(u_1)$.
One can see that both hand sides of \eqref{rationalLascouxsymmetrization} give $f(u_1)$.

From the determinant representation of $F_{n,n}^{(z)}(\overline{u}|\overline{v})$
\eqref{oneparameterdefFodaWheeler}, we have
\begin{align}
&P_{n,n}^{(z=1)}(\overline{u}|\overline{v}) \nn \\
=&
\frac{\prod_{i=1}^n \prod_{k=1}^n (v_i-u_k)(v_i-u_k-c)}{\prod_{1 \le i < j \le n} (v_j-v_i) \prod_{1 \le i < j \le n} (u_i-u_j)} 
\det_{1 \le j,k \le n} \Bigg( \frac{1}{v_j-u_k}-\frac{1}{v_j-u_k-c}  \Bigg) \nn \\
=&
(-c)^n \frac{\prod_{i=1}^n \prod_{k=1}^n (v_i-u_k)(v_i-u_k-c)}{\prod_{1 \le i < j \le n} (v_j-v_i) \prod_{1 \le i < j \le n} (u_i-u_j)} 
\det_{1 \le j,k \le n} \Bigg(\frac{1}{(v_j-u_k)(v_j-u_k-c)}  \Bigg),
\label{Pnnz=1}
\end{align}
which is nothing but the Gaudin-Izergin-Korepin determinant \cite{Izergin,Korepin,Gaudin},
and \eqref{rationalLascouxsymmetrization} can be rewritten as
\begin{align}
&\mathrm{Sym}_c \Bigg( (1-\theta)^{n-1} \prod_{j=2}^n \prod_{k=1}^n (u_j-v_k) f(u_1) 
\Bigg) 
=\frac{(n-1)!}{-c} P_{n,n}^{(z=1)}(\overline{u}|\overline{v})
\partial_{n-1} \cdots \partial_1 f(u_1), \label{rationalLascoux}
\end{align}
if $c \neq 0$. When $c=0$, due to the factor $(-c)^{n-1}$
in the right hand side and $n \geq 2$, one notes that
\eqref{rationalLascouxsymmetrization} can be written as
\begin{align}
\mathrm{Sym}_{c=0} \Bigg( (1-\theta)^{n-1} \prod_{j=2}^n \prod_{k=1}^n (u_j-v_k) f(u_1) 
\Bigg)=0.
\label{rationalLascouxcspecial}
\end{align}
Finally, let us remark that the factor $[n]!$ in Equation (6) in
\cite[Theorem 3]{Lascoux} is a typo of $[n-1]!$.
\cite[Theorem 4]{Lascoux} which follows from the theorem does not need to be corrected.
}
\begin{proof}
{\color{black}
We first check when $c=0$ which is explicitly \eqref{rationalLascouxcspecial}.
Let us denote $\prod_{j=2}^n \prod_{k=1}^n (u_j-v_k) f(u_1)$ as $g(u_1,u_2,\dots,u_n)$.
Using $\displaystyle \sum_{k=0}^{n-1} (-1)^k \binom{n-1}{k}=0$,
one can check the left hand side of \eqref{rationalLascouxcspecial} vanishes in the following way
\begin{align}
&\mathrm{Sym}_{c=0} \Big( (1-\theta)^{n-1} g(u_1,u_2,\dots,u_n)
\Big) \nn \\
=&\sum_{k=0}^{n-1} (-1)^k \binom{n-1}{k} \mathrm{Sym}_{c=0} \Big( 
g(u_{k+1},u_{k+2},\dots,u_n,u_1,\dots,u_k)
\Big) \nn \\
=&\sum_{k=0}^{n-1} (-1)^k \binom{n-1}{k} \mathrm{Sym}_{c=0} \Big( 
g(u_{1},u_{2},\dots,u_n) \Big)=0.
\end{align}

Next, we prove the case when $c \neq 0$
which can be rewritten as \eqref{rationalLascoux}.}
This can be proved in the same way with the proof of Theorem 3 in \cite{Lascoux}.
One can show that both hand sides of \eqref{rationalLascoux}
are symmetric in $u_1,\dots,u_n$.
The symmetry of the $u$ variables of the right hand side can be easily checked from the form \eqref{toseesymmetryuvariables}.
Since both hand sides are symmetric in $u_1,\dots,u_n$,
it is enough to check the equality between coefficients of $f(u_1)$ of both hand sides.
Using $\displaystyle (1-\theta)^{n-1}=\sum_{\ell=1}^n \binom{n-1}{\ell-1} (-1)^{\ell-1} \theta^{\ell-1}$,
we can write the left hand side of \eqref{rationalLascoux} as
\begin{align}
&\sum_{\ell=1}^n  {\color{black} (-1)^{\ell-1} } \binom{n-1}{\ell-1}
\sum_{w \in S_n} w \cdot
\Bigg(
{\color{black}\Delta(1,2,\dots,n)} \nn \\
&\times \prod_{j=2}^{n+1-\ell} \prod_{k=1}^n (u_{j+\ell-1}-v_k) \prod_{j=n+2-\ell}^n \prod_{k=1}^n
(u_{j+\ell-1-n}-v_k+c) f(u_\ell)
\Bigg) \nn \\
=
&\sum_{\ell=1}^n  {\color{black}(-1)^{\ell-1} } \binom{n-1}{\ell-1}
\sum_{w \in S_n} w \cdot
\Bigg(
{\color{black}\Delta(\ell,2,3,\dots,\ell-1,1,\ell+1,\dots,n)} \nn \\
&\times \prod_{j=\ell+1}^{n} \prod_{k=1}^n (u_{j}-v_k) \prod_{j=2}^\ell \prod_{k=1}^n
(u_{j}-v_k+c) f(u_1)
\Bigg).
\end{align}
On the other hand,
the coefficient of $f(u_1)$ of the right hand side of \eqref{rationalLascoux} is
\begin{align}
\frac{(n-1)!}{{\color{black}-c} \prod_{j=2}^n (u_1-u_j) } P_{n,n}^{(z=1)}(\overline{u}|\overline{v}),
\end{align}
hence what we need to show is the following identity
\begin{align}
&\sum_{\ell=1}^n {\color{black} (-1)^{\ell-1} } \binom{n-1}{\ell-1}
\sum_{w \in S_n: w(1)=1} w \cdot
\Bigg(
{\color{black}\Delta(\ell,2,3,\dots,\ell-1,1,\ell+1,\dots,n)} \nn \\
&\times \prod_{j=\ell+1}^{n} \prod_{k=1}^n (u_{j}-v_k) \prod_{j=2}^\ell \prod_{k=1}^n
(u_{j}-v_k+c)
\Bigg)=\frac{(n-1)!}{{\color{black}-c} \prod_{j=2}^n (u_1-u_j) } P_{n,n}^{(z=1)}(\overline{u}|\overline{v}).
\label{reductiontoshow}
\end{align}
We prove this identity by viewing both hand sides as polynomials in $v_1,\dots,v_n$.
One can easily see that the degree of $v_j, j=1,\dots,n$ of the left hand side
of \eqref{reductiontoshow} is at most $n-1$.
From the determinant representation \eqref{Pnnz=1},
we can check that the degree of $v_j, j=1,\dots,n$ of $P_{n,n}^{(z=1)}(\overline{u}|\overline{v})$ and hence
the right hand side of \eqref{reductiontoshow} is also at most $n-1$.

Since the degree of $v_j$ $( j=1,\dots, n)$ is at most $n-1$, it is enough to check the equality
\eqref{reductiontoshow} for $n^n$ distinct points in $(v_1,\dots,v_n)$.
We check the following $(2n-1) \times (2n-2) \times \cdots \times n$ $(\ge n^n)$ cases: specializing $v_1,\dots,v_n$ to subsets of
$\{ u_1,u_2,u_2+c,\dots,u_n,u_n+c \}$.

It is easy to see from the expression that the left hand side of \eqref{reductiontoshow}
vanishes if $v_i=u_k$, $v_j=u_k+c$ for some $i,j,k \ (i \neq j)$ and $2 \le k \le n$.
The right hand side also vanishes due to Lemma \ref{vanishinglemma}.
We next consider the other cases which we denote by
$\overline{u}=\{\overline{v}, \overline{v}+c \}$ in the previous subsection.
By symmetry, it is enough to take $I={1,2,\dots,|I|}$ and $J={|I|+1,\dots,n}$,
i.e., consider the substitution $v_1=u_1,\dots,v_{|I|}=u_{|I|},v_{|I|+1}=u_{|I|+1}+c,\dots,v_n=u_n+c$.
One notes that the summands in the left hand side which survive after this substitution
are those which satisfy $\ell=|I|$ and
$\{ w(2), \dots, w(|I|) \}=\{ 2,\dots,|I| \}$, $\{ w(|I|+1),\dots,w(n) \}=\{ |I|+1,\dots,n \}$
in addition to $w(1)=1$.
The left hand side reduces to
\begin{align}
& {\color{black} (-1)^{|I|-1} } \binom{n-1}{|I|-1}
\sum_{\substack{ w \in S_n: w(1)=1 \\ \{ w(2), \dots, w(|I|) \}=\{ 2,\dots,|I| \} \\
\{ w(|I|+1),\dots,w(n) \}=\{ |I|+1,\dots,n \}
}}
w \cdot
\Bigg(
{\color{black}\Delta(|I|,2,3,\dots,|I|-1,1,|I|+1,\dots,n) \Bigg)} \nn \\
&\times \Bigg( \prod_{j=|I|+1}^{n} \prod_{k=1}^n (u_{j}-v_k) \prod_{j=2}^{|I|} \prod_{k=1}^n
(u_{j}-v_k+c)
\Bigg) \Bigg|_{v_1=u_1,\dots,v_{|I|}=u_{|I|},v_{|I|+1}=u_{|I|+1}+c,\dots,v_n=u_n+c}.
\label{reductiontoshowspecialization}
\end{align}
Using the well-known identity
\begin{align}
\sum_{w \in S_n} \Delta(k_1,k_2,\dots,k_n)=n!,
\end{align}
one finds
\begin{align}
&\sum_{\substack{ w \in S_n: w(1)=1 \\ \{ w(2), \dots, w(|I|) \}=\{ 2,\dots,|I| \} \\
\{ w(|I|+1),\dots,w(n) \}=\{ |I|+1,\dots,n \}
}}
w \cdot
\Bigg(
{\color{black}\Delta(|I|,2,3,\dots,|I|-1,1,|I|+1,\dots,n) \Bigg)} \nn \\
=&(|I|-1)! (n-|I|)! \prod_{j=2}^{|I|} \prod_{k=|I|+1}^n \frac{u_j-u_k-c}{u_j-u_k}
\prod_{j=2}^{|I|} \frac{u_j-u_1-c}{u_j-u_1} \prod_{k=|I|+1}^n \frac{u_1-u_k-c}{u_1-u_k}.
\label{reductiontoshowspecializationone}
\end{align}
We can also easily see
\begin{align}
&\Bigg( \prod_{j=|I|+1}^{n} \prod_{k=1}^n (u_{j}-v_k) \prod_{j=2}^{|I|} \prod_{k=1}^n
(u_{j}-v_k+c)
\Bigg) \Bigg|_{v_1=u_1,\dots,v_{|I|}=u_{|I|},v_{|I|+1}=u_{|I|+1}+c,\dots,v_n=u_n+c} \nn \\
=&\prod_{j=|I|+1}^n \prod_{k=1}^{|I|} (u_j-u_k)
\prod_{j=|I|+1}^n \prod_{k=|I|+1}^n (u_j-u_k-c)
\prod_{j=2}^{|I|} \prod_{k=1}^{|I|} (u_j-u_k+c) \prod_{j=2}^{|I|} \prod_{k=|I|+1}^{n} (u_j-u_k).
\label{reductiontoshowspecializationtwo}
\end{align}
Inserting
\eqref{reductiontoshowspecializationone} and \eqref{reductiontoshowspecializationtwo}
into \eqref{reductiontoshowspecialization} and simplifying and rearranging, 
we have the following expression for 
\eqref{reductiontoshowspecialization} which is the specialization of the left hand side of
\eqref{reductiontoshow}
\begin{align}
&\sum_{\ell=1}^n {\color{black} (-1)^{\ell-1}} \binom{n-1}{\ell-1}
\sum_{w \in S_n: w(1)=1} w \cdot
\Bigg(
{\color{black}\Delta(\ell,2,3,\dots,\ell-1,1,\ell+1,\dots,n)} \nn \\
&\times \prod_{j=\ell+1}^{n} \prod_{k=1}^n (u_{j}-v_k) \prod_{j=2}^\ell \prod_{k=1}^n
(u_{j}-v_k+c)
\Bigg) \nn \\
=&\frac{(n-1)! (-1)^{n-|I|}}{{\color{black}-c} \prod_{j=2}^n (u_1-u_j)} 
\prod_{j=|I|+1}^n \prod_{k=1}^{|I|} (u_j-u_k)
\prod_{j=1}^{|I|} \prod_{k=1}^n (u_j-u_k-c)
\prod_{j=|I|+1}^n \prod_{k=|I|+1}^n (u_j-u_k+c).
\label{specializationlhs}
\end{align}
On the other hand,
from the case $m=n$, $I={1,2,\dots,|I|}$ and $J={|I|+1,\dots,n}$ of
Proposition \ref{PQspecializations}, we find the specialization of the 
right hand side of \eqref{reductiontoshow} is
\begin{align}
&\frac{(n-1)!}{{\color{black}-c} \prod_{j=2}^n (u_1-u_j) } P_{n,n}^{(z=1)}(\overline{u}|\overline{v})
\Bigg|_{v_1=u_1,\dots,v_{|I|}=u_{|I|},v_{|I|+1}=u_{|I|+1}+c,\dots,v_n=u_n+c}
\nn \\
=&\frac{(n-1)!(-1)^{n-|I|}}{{\color{black}-c} \prod_{j=2}^n (u_1-u_j)} 
\prod_{j=|I|+1}^n \prod_{k=1}^{|I|} (u_j-u_k)
\prod_{j=1}^{|I|} \prod_{k=1}^n (u_j-u_k-c)
\prod_{j=|I|+1}^n \prod_{k=|I|+1}^n (u_j-u_k+c),
\end{align}
which is exactly the same with the specialization of the left hand side \eqref{specializationlhs}.

Since we checked the equality for enough specializations of $(v_1,\dots,v_n)$,
we conclude the identity \eqref{reductiontoshow} and hence
\eqref{rationalLascoux} holds.

\end{proof}

The following identity is the rational version of {\color{black} \cite[Theorem 4]{Lascoux}.}
\begin{theorem}
{\color{black}
We have
\begin{align}
&\mathrm{Sym}_c \Bigg(
(1-\tau)^n \prod_{j=1}^n \prod_{k=1}^n \frac{u_j-v_k-c}{u_j-v_k}
\Bigg) \nn \\
=&
\frac{n! c^n \prod_{i=1}^n \prod_{k=1}^n (v_i-u_k+c)}{\prod_{1 \le i < j \le n} (v_j-v_i) \prod_{1 \le i < j \le n} (u_i-u_j)} 
\det_{1 \le j,k \le n} \Bigg(\frac{1}{(v_j-u_k+c)(v_j-u_k)}  \Bigg),
\label{corrationallascouxexplicit}
\end{align}
}

Here $\tau$ acts on variables and functions as $\tau u_j:=u_{j+1}$, $j=1,\dots,n-1$ and 
$\tau \left( \frac{u_n-v_k-c}{u_n-v_k} \right):=1$.
\end{theorem}
{\color{black}
Let us give some remarks.
For $n=1$, both hand sides of \eqref{corrationallascouxexplicit} are $\displaystyle \frac{-c}{u_1-v_1}$.
For $n=2$, the left hand side of \eqref{corrationallascouxexplicit} is
\begin{align}
\Bigg(
\frac{u_1-v_1-c}{u_1-v_1} \frac{u_1-v_2-c}{u_1-v_2} \frac{u_2-v_1-c}{u_2-v_1} \frac{u_2-v_2-c}{u_2-v_2}
-2 \frac{u_2-v_1-c}{u_2-v_1} \frac{u_2-v_2-c}{u_2-v_2}+1
\Bigg) \frac{u_1-u_2-c}{u_1-u_2} \nn \\
+\Bigg(
\frac{u_2-v_1-c}{u_2-v_1} \frac{u_2-v_2-c}{u_2-v_2} \frac{u_1-v_1-c}{u_1-v_1} \frac{u_1-v_2-c}{u_1-v_2}
-2 \frac{u_1-v_1-c}{u_1-v_1} \frac{u_1-v_2-c}{u_1-v_2}+1
\Bigg) \frac{u_2-u_1-c}{u_2-u_1},
\end{align}
which can be simplified as
\begin{align}
\frac{2c^2(c^2-cu_1-cu_2+cv_1+cv_2-u_1 v_1-u_2 v_1-u_1 v_2-u_2 v_2+2 u_1 u_2+2 v_1 v_2)}
{(u_1-v_1)(u_1-v_2)(u_2-v_1)(u_2-v_2)}. \label{checkexampleLascouxsecond}
\end{align}
One can also check that
the right hand side of \eqref{corrationallascouxexplicit} for $n=2$
can also be simplified as \eqref{checkexampleLascouxsecond}.
Also note that
using \eqref{Pnnz=1}, 
\eqref{corrationallascouxexplicit} can be rewritten as
\begin{align}
\mathrm{Sym}_c \Bigg(
(1-\tau)^n \prod_{j=1}^n \prod_{k=1}^n \frac{u_j-v_k-c}{u_j-v_k}
\Bigg)
=
\frac{n!}{ \prod_{j=1}^n \prod_{k=1}^n (u_j-v_k)  } P_{n,n}^{(z=1)}(\overline{u}|\overline{v}+c).
\label{corrationallascoux}
\end{align}
}

\begin{proof}
{\color{black} For $c=0$, it is easy to see that both hand sides of \eqref{corrationallascouxexplicit} vanish.
Let us now assume $c \neq 0$.}
This can be shown in the same way with the proof of Theorem 4 in \cite{Lascoux}.
First, by direct computation, we find the following relation
\begin{align}
&(1-\tau)^n \prod_{j=1}^n \prod_{k=1}^n \frac{u_j-v_k-c}{u_j-v_k} \nn \\
=&\prod_{j=1}^n \prod_{k=1}^n \frac{1}{u_j-v_k} \sum_{j=1}^n
  (-1)^{j-1} \binom{n-1}{j-1}  
\prod_{\ell=1}^{j-1} \prod_{k=1}^n (u_\ell-v_k)
\prod_{\ell=j+1}^{n} \prod_{k=1}^n (u_\ell-v_k-c)
f(u_j), \label{tocorlasone}
\end{align}
where
\begin{align}
f(u_j)=\prod_{k=1}^n (u_j-v_k-c)-\prod_{k=1}^n (u_j-v_k). \label{tocorlaspoly}
\end{align}
We also note the following equality holds
\begin{align}
&\sum_{j=1}^n
  (-1)^{j-1} \binom{n-1}{j-1}  
\prod_{\ell=1}^{j-1} \prod_{k=1}^n (u_\ell-v_k)
\prod_{\ell=j+1}^{n} \prod_{k=1}^n (u_\ell-v_k-c)
f(u_j) \nn \\
=&(1-\theta)^{n-1} \prod_{j=2}^n \prod_{k=1}^n (u_j-v_k-c) f(u_1). \label{tocorlastwo}
\end{align}
Combining \eqref{tocorlasone} and \eqref{tocorlastwo},
we get
\begin{align}
(1-\tau)^n \prod_{j=1}^n \prod_{k=1}^n \frac{u_j-v_k-c}{u_j-v_k} 
=\prod_{j=1}^n \prod_{k=1}^n \frac{1}{u_j-v_k} (1-\theta)^{n-1} \prod_{j=2}^n \prod_{k=1}^n (u_j-v_k-c) f(u_1),
\end{align}
and further combining with \eqref{rationalLascoux}, we have
\begin{align}
&-c \ \mathrm{Sym}_c \Bigg(
(1-\tau)^n \prod_{j=1}^n \prod_{k=1}^n \frac{u_j-v_k-c}{u_j-v_k}
\Bigg) \nn \\
=&-c \ \mathrm{Sym}_c \Bigg(
\prod_{j=1}^n \prod_{k=1}^n \frac{1}{u_j-v_k} (1-\theta)^{n-1} \prod_{j=2}^n \prod_{k=1}^n (u_j-v_k-c) f(u_1)
\Bigg) \nn \\
=& \prod_{j=1}^n \prod_{k=1}^n \frac{1}{u_j-v_k} \times (-c) \
\mathrm{Sym}_c
\Bigg(
(1-\theta)^{n-1} \prod_{j=2}^n \prod_{k=1}^n (u_j-v_k-c) f(u_1)
\Bigg) \nn \\
=&
\frac{(n-1)!}{ \prod_{j=1}^n \prod_{k=1}^n (u_j-v_k)  } P_{n,n}^{(z=1)}(\overline{u}|\overline{v}+c)
\partial_{n-1} \cdots \partial_1 f(u_1). \label{corrationallascouxpre}
\end{align}
Finally, we note the explicit form of $\partial_{n-1} \cdots \partial_1 f(u_1)$ where $f(u_1)$
is the $j=1$ case of \eqref{tocorlaspoly} can be obtained from the coefficient
of $u_1^{n-1}$ of $f(u_1)$ since lower degree terms vanish after applying the divided difference operators.
We find $\partial_{n-1} \cdots \partial_1 f(u_1)=-nc$ and
inserting into \eqref{corrationallascouxpre} gives
\eqref{corrationallascoux}.
\end{proof}

\subsection{Geometric derivation of rational source identity}

The rational version \eqref{rationalKajihara} of the 
source identity is also obtained from the cohomological 
version of the wall-crossing formula.
We follow the similar arguments as in \S 
\ref{subsec:geom} taking the usual Euler class 
\[
\eu^{\theta}(M^{\zeta^{\pm}}_{Q}(\ell))=
\eu(TM^{\zeta^{\pm}}_{Q} (\ell) 
\otimes \mathbb{C}_{e^{\theta}})
\]
instead of $\eu^{K}(M^{\zeta^{\pm}}_{Q} (\ell))$.


We see that $\int_{M^{\zeta^{+}}_{Q}(\ell)} 
\eu^{\theta}(M^{\zeta^{+}}_{Q}(\ell) )$ is equal to
\begin{align*}
&\sum_{K \subset [n+1,\dots, n +m] \atop 
\lvert K \rvert=\ell}
\prod_{i \in K \atop j \in [n+1,\dots, n +m] \setminus K} 
{a_j - a_i + \theta \over a_j - a_i}
\prod_{ i \in K  \atop 1 \le k \le n } 
{a_i - a_k + \theta \over a_i - a_k}
\\
=&\sum_{K \subset [1,\dots,m] \atop 
\lvert K \rvert=\ell}
\prod_{i \in K \atop j \not\in K} 
{v_i-v_j-c \over v_i-v_j}
\prod_{ i \in K  \atop 1 \le k \le n } 
{v_i-u_k \over v_i-u_k-c},
\end{align*}
and $\int_{M^{\zeta^{-}}_{Q}(\ell)} 
\Eu^{\theta}(M^{\zeta^{-}}_{Q}(\ell) )$ is equal to
\begin{align*}
& \sum_{K \subset [1,\dots,n] \atop \lvert K \rvert=\ell} 
\prod_{i \in K \atop j \in [1, \ldots, n] \setminus  K} 
{ a_i - a_j+\theta \over a_i - a_j}
\prod_{ i \in K  \atop n+1 \le k \le n+m } 
{a_k - a_i + \theta \over a_k - a_i}
\\
=& \sum_{K \subset [1,\dots,n] \atop 
\lvert K \rvert=\ell} 
\prod_{i \in K \atop j \not\in K} 
{u_i-u_j+c \over u_i-u_j}
\prod_{ i \in K  \atop 1 \le k \le m } 
{u_i-v_k \over u_i-v_k+c},
\end{align*}
where we put $\theta=-c$, and $a_{j} = \begin{cases} -u_{j} -c & j=1, \ldots, n \\ - v_{j} & j=n+1, \ldots, n+m \end{cases}$.
These are coefficients of $z^{\ell}$ in generating series 
appearing in both sides of 
\eqref{rationalKajihara}.  
\\
Wall-crossing formula \cite[Theorem 7.2]{O} induces
\begin{align*}
&
\int_{M^{\zeta^{+}}_{Q}(\ell)} \Eu^{\theta}(M^{\zeta^{+}}_{Q}(\ell) ) 
- \int_{M^{\zeta^{-}}_{Q}(\ell)} 
\Eu^{\theta}(M^{\zeta^{-}}_{Q}(\ell) ) \\
&=
\sum_{k=1}^{\ell} 
\sum_{\mbi{\mk I} \in \Dec(\ell) \atop |\mbi{\mk I}|=k} 
{ (\ell-k)! \over \ell ! }
\prod_{i=1}^{j} ( d_{i} -1)! 
\gamma_{d_{i}} (\theta) 
( s( \mk I_{i}, \mbi{\mk I}_{> i}) -  (n-m)d_{i} ) 
\int_{ M^{\zeta^{-}}_{Q}(\ell -k)} 
\Eu^{\theta}(M^{\zeta^{-}}_{Q}(\ell) ).
\end{align*}
Here we set 
\[
s(\mk I_{1}, \mk I_{2}) = \left| \lbrace (i, j) \in 
\mk I_{1} \times \mk I_{2} \mid i < j \rbrace \right| - 
\left| \lbrace (i, j) \in \mk I_{1} \times \mk I_{2} 
\mid i > j \rbrace \right|,
\]
and we have
\begin{align}
\label{gamma}
\gamma_{d}(\theta) = \int_{M_{Q^{\sharp}}(d)} 
\Eu^{\theta}(TM^{\zeta^{-}}_{Q^{\sharp}}(d) )  
= \begin{cases} 1 & d=1\\ 0 & d \neq 1. \end{cases}
\end{align}

By \eqref{gamma}, we can re-write wall-crossing 
formula as
\begin{align*}
&
\int_{M^{\zeta^{+}}_{Q}(\ell)} \Eu^{\theta} 
((M^{\zeta^{+}}_{Q}(\ell) ))
- \int_{M^{\zeta^{-}}_{Q}(\ell)} \Eu^{\theta}
(M^{\zeta^{-}}_{Q}(\ell) ) \\
&=
\sum_{k=1}^{\ell} 
\sum_{\mbi{\mk I} \in \Dec(\ell) 
\atop 
{|\mbi{\mk I}|=k
\atop |\mk I_{1}|= \cdots = |\mk I_{k}|=1}}
{ (\ell-k)! \over \ell ! }
\prod_{i=1}^{j}  ( s( \mk I_{i}, \mbi{\mk I}_{> i}) +  (m-n) ) 
\int_{ M^{\zeta^{-}}_{Q}(\ell -k)} \Eu^{\theta}
((M^{\zeta^{-}}_{Q}(\ell) )).
\end{align*}

Since \eqref{rationalKajihara} is equivalent to 
\[
\int_{M^{\zeta^{+}}_{Q}(\ell)} \Eu^{\theta}
((M^{\zeta^{+}}_{Q}(\ell) )) 
- \int_{M^{\zeta^{-}}_{Q}(\ell)} \Eu^{\theta} 
((M^{\zeta^{-}}_{Q}(\ell) ))\\
=
\sum_{k=1}^{\ell} 
\binom{m-n}{k} \int_{ M^{\zeta^{-}}_{Q}(\ell -k)} 
\Eu^{\theta}((M^{\zeta^{-}}_{Q}(\ell) )),
\]
we must show that
\begin{align}
\label{source}
\sum_{{\ell \ge h_{1} > \cdots > h_{k} \ge 1}} 
{ (\ell-k)!  \over \ell ! }
\prod_{i=1}^{k}  ( s( \lbrace h_{i} \rbrace, 
[1, \ldots, \ell] \setminus \lbrace h_{1}, \ldots, h_{i} 
\rbrace ) +  (m-n) ) 
=
\binom{m-n}{k}.
\end{align}
This follows from taking the limit $t \to 1$
of \eqref{qidentityforgeometricderivation}.

\section*{Acknowledgement}
The authors thank Vitaly Tarasov for useful discussions
and sending their preprint.
The authors would like to thank the reviewers for their careful reading and various comments and suggestions.
K.M. is supported by Grant-in-Aid for Scientific Research 21K03176, 20K03793, JSPS.
R.O. is supported by Grant-in-Aid for Scientific Research 21K03180, 
JSPS.


\begin{thebibliography}{00}
%
%


\bibitem{Kajihara}
Kajihara Y 2004
Euler transformation formula for multiple basic hypergeometric series of
type $A$ and some applications
\emph{Adv. Math.} {\bf 187} 53

\bibitem{GF}
Gomis J and Le Floch B 2016
M2-brane surface operators and gauge theory dualities in Toda
\emph{JHEP} {\bf 04} 183

\bibitem{HYY}
Hwang C, Yi P and Yoshida Y 2017
Fundamental Vortices, Wall-Crossing, and Particle-Vortex duality
\emph{JHEP} {\bf 05} 099

\bibitem{OhkawaYoshida}
Ohkawa R and Yoshida Y 2023
Wall-crossing for vortex partition function and handsaw quiver varierty
\emph{J. Geom. Phys.} {\bf 191} 104904

\bibitem{KajiharaNoumi}
Kajihara Y and Noumi M 2003
Multiple elliptic hypergeometric series --An approach from the Cauchy determinant--
\emph{Indag. Math.} {\bf 14} 395

\bibitem{Rosengren}
Rosengren H 2006
New transformations for elliptic hypergeometric series on the root system $A_n$
\emph{Ramanujan J.} {\bf 12} 155

\bibitem{KN}
Kirillov A N and Noumi M 1999 $q$-difference raising operators for Macdonald polynomials and the integrality of transition coefficients
\emph{Algebraic Methods and $q$-Special Functions (Montr\'eal, QC, 1996), CRM Proc. Lecture Notes} {\bf 22} 227-243 (Amer. Math. Soc., Providence, RI)

\bibitem{MN}
Mimachi K and Noumi M 1997
An integral representation of eigenfunctions for Macdonald's $q$-difference operators
\emph{Tohoku Math. J.} {\bf 49} 517

\bibitem{GZZ}
Gorsky A, Zabrodin A and Zotov A 2014
Spectrum of Quantum Transfer Matrices via Classical Many-Body Systems
\emph{JHEP} {\bf 01} 070

\bibitem{BSV}
Belliard S, Slavnov N A and Vallet B 2018
Scalar product of twisted XXX modified Bethe vectors
\emph{J. Stat. Mech.} 093103

\bibitem{BLZZ}
Beketov M, Liashyk A, Zabrodin A and Zotov A 2016
Trigonometric version of quantum-classical duality in integrable systems
\emph{Nucl. Phys. B} {\bf 903} 150

\bibitem{O}Ohkawa R 2023
Wall-crossing formula for framed quiver moduli
(arXiv:2305.09217)

\bibitem{OS}Ohkawa R and Shiraishi J 2024
$K$-theoretic wall-crossing and transformation formulas 
for multiple hypergeometric series
(arXiv:2402.04571)



\bibitem{Izergin}
Izergin A G 1987
Partition function of the six-vertex model in a finite volume
\emph{ Sov. Phys. Dokl.}
{\bf 32} 878

\bibitem{Korepin}
Korepin V 1982
Calculation of norms of Bethe wave functions
\emph{ Commun. Math. Phys.} {\bf 86} 391



\bibitem{GSV}
Gromov N, Sever A and Vieira P 2012
Tailoring three-point functions and integrability III
 \emph{JHEP} {\bf 07}  044


\bibitem{Kostovone}
Kostov I 2012
Classical Limit of the Three-Point Function
 \emph{ Phys. Rev. Lett.} {\bf 108} 261604

\bibitem{Kostovtwo}
Kostov I 2012
Three-point function of semiclassical states
at weak coupling
\emph{J. Phys. A:Math. Theor.} {\bf 45} 494018

\bibitem{FodaWheeler}
Foda O and Wheeler M 2012
Partial domain wall partition functions
\emph{JHEP} {\bf 7} 186

\bibitem{BS}
Belliard S and Slavnov N A 2021
Overlap between usual and modified Bethe vectors
\emph{Theor. Math. Phys.} {\bf 209}  1387

\bibitem{MPT}
Minin M, Pronko A and Tarasov V 2023
Construction of determinants for the six-vertex model with domain wall boundary conditions
\emph{J. Phys. A: Math. Theor.} {\bf 56} 295204

\bibitem{PT}
Pronko A and Tarasov V 2023
Polynomial structure in determinants for Izergin-Korepin partition function, preprint.

\bibitem{BPS}
Belliard S, Pimenta R A and Slavnov N A 2024
Modified rational six vertex model on the rectangular lattice
\emph{SciPost Phys.} {\bf 16} 009

\bibitem{Gaudin}
Gaudin M 1983
\emph{La fonction d'onde de Bethe} (Masson, Paris),
Gaudin M 2014
\emph{The Bethe Wavefunction} (translation by J.-S. Caux) (Cambridge University Press)


\bibitem{Slavnov}
Slavnov N A 1989
Calculation of scalar products of wave functions and form factors in the
framework of the algebraic Bethe Ansatz
\emph{Theor. Math. Phys.} {\bf 79} 502




\bibitem{Lascoux}
Lascoux A 2007
{\color{black}
Gaudin functions and Euler-Poincar\'e characteristics
(arXiv:0709.1635)
}
\bibitem{RS}
Rosengren H and Schlosser M 2006
Elliptic determinant evaluations and the Macdonald identities
for affine root systems
\emph{Compositio Math.} {\bf 142} 937

\bibitem{Katori}
Katori M 2023
\emph{Elliptic Extensions in Statistical and Stochastic Systems},
\emph{Springer Briefs in Mathematical Physics} {\bf 47}
(Springer Verlag, Singapore)

\bibitem{Frobenius}
Frobenius G 1882
\"Uber die elliptischen Funktionen zweiter Art
\emph{J. Reine und Angew. Math.} {\bf 93} 53

\bibitem{Rosengrenlecturenote}
Rosengren H 2020
Elliptic hypergeometric functions
\emph{Lectures on Orthogonal Polynomials and Special Functions}
(Cambridge University Press, Cambridge)

\end{thebibliography}
\end{document}